\documentclass[11pt,b5paper,notitlepage]{article}
\usepackage[b5paper, margin={0.5in,0.75in}]{geometry}
\usepackage{amsmath,amscd,amssymb,amsthm,mathrsfs,amsfonts,layout,indentfirst,graphicx,caption,mathabx, stmaryrd,appendix,calc,imakeidx,upgreek} 
\usepackage{palatino}  
\usepackage{slashed} 
\usepackage{mathrsfs} 
\usepackage{extarrows} 
\usepackage{enumitem} 

\usepackage{fancyhdr} 

\usepackage[nottoc]{tocbibind}   

\makeindex

\usepackage{lipsum}
\let\OLDthebibliography\thebibliography
\renewcommand\thebibliography[1]{
	\OLDthebibliography{#1}
	\setlength{\parskip}{0pt}
	\setlength{\itemsep}{2pt} 
}

\allowdisplaybreaks  
\usepackage{latexsym}
\usepackage{chngcntr}
\usepackage[colorlinks,linkcolor=blue,anchorcolor=blue, linktocpage,
]{hyperref}
\hypersetup{ urlcolor=cyan,
	citecolor=[rgb]{0,0.5,0}}

\counterwithin{figure}{section}

\theoremstyle{definition}
\newtheorem{df}{Definition}[section]
\newtheorem{eg}[df]{Example}
\newtheorem{rem}[df]{Remark}
\newtheorem{ass}[df]{Assumption}
\newtheorem{cv}[df]{Convention}

\theoremstyle{plain}
\newtheorem{thm}[df]{Theorem}

\newtheorem{pp}[df]{Proposition}
\newtheorem{co}[df]{Corollary}
\newtheorem{lm}[df]{Lemma}

\newcommand{\fk}{\mathfrak}
\newcommand{\mc}{\mathcal}
\newcommand{\wtd}{\widetilde}
\newcommand{\wht}{\widehat}
\newcommand{\wch}{\widecheck}
\newcommand{\ovl}{\overline}
\newcommand{\tr}{\mathrm{t}} 

\newcommand{\End}{\mathrm{End}} 
\newcommand{\id}{\mathbf{1}}

\newcommand{\Res}{\mathrm{Res}}

\newcommand{\Span}{\mathrm{Span}}

\newcommand{\bk}[1]{\langle {#1}\rangle}

\newcommand{\scr}{\mathscr}

\newcommand{\yk}{\mathfrak y}

\newcommand{\im}{\mathbf{i}}

\newcommand{\sgm}{\varsigma}
\newcommand{\SX}{S_{\fk X}}

\newcommand{\mbb}{\mathbb}

\newcommand{\blt}{\bullet}

\newcommand{\Vbb}{\mathbb V}
\newcommand{\Wbb}{\mathbb W}
\newcommand{\Mbb}{\mathbb M}
\newcommand{\Gbb}{\mathbb G}
\newcommand{\Cbb}{\mathbb C}
\newcommand{\Nbb}{\mathbb N}
\newcommand{\Zbb}{\mathbb Z}
\newcommand{\Pbb}{\mathbb P}

\newcommand{\Ebb}{\mathbb E}
\newcommand{\cbf}{\mathbf c}
\newcommand{\wt}{\mathrm{wt}}

\newcommand{\btl}{\blacktriangleleft}
\newcommand{\btr}{\blacktriangleright}
\newcommand{\svir}{\mathcal V\!\mathit{ir}}
\newcommand{\Ker}{\mathrm{Ker}}

\newcommand{\Sbf}{\mathbf{S}}

\newcommand{\Lss}{{L_{0,\mathrm{s}}}}
\newcommand{\Lni}{{L_{0,\mathrm{n}}}}

\numberwithin{equation}{section}

\title{Convergence of Sewing Conformal Blocks}
\author{{\sc Bin Gui}
}
\date{}
\begin{document}\sloppy 
	\pagenumbering{arabic}

	\maketitle



\tableofcontents

	







\newpage

\begin{abstract}
In a recent work \cite{DGT19b}, Damiolini-Gibney-Tarasca showed that for a $C_2$-cofinite rational vertex operator algebra $\Vbb$, sheaves of conformal blocks are locally free and satisfy the factorization property. In this article, we  prove that if $\Vbb$ is $C_2$-cofinite, the sewing of conformal blocks is convergent. This proves a conjecture proposed by Zhu \cite{Zhu94} and Huang \cite{Hua16}, generalizing previous results on the convergence of products, iterates, and traces (but not pseudo-traces) of vertex operators and intertwining operators in \cite{Zhu96,Hua95,Hua05b,HLZ11} and on the convergence of projective factors related to the central charge of the Virasoro algebra in \cite{Hua97}.
\end{abstract}

\section*{Introduction}

In conformal field theory (CFT), conformal blocks are linear functionals defined for $N$-pointed compact Riemann surfaces, together with a vertex operator algebra (VOA) $\Vbb$ and $\Vbb$-modules $\Wbb_1,\dots,\Wbb_N$. One can sew a (possibly disconnected) $N$-pointed compact Riemann surface along pairs of points (with local coordinates) to obtain a new pointed compact Riemann surface with possibly higher genus or more marked points. Corresponding to this geometric sewing construction, one can also sew conformal blocks by taking contractions \cite{Seg88,Vafa87,TK88,TUY89,BFM91,Zhu94,Hua97,Hua05a,Hua05b,NT05,Hua16,DGT19b}. In this article, we prove that for $\Vbb$ satisfying natural conditions,  sewing conformal blocks is convergent, solving a conjecture  proposed in \cite[Conj. 8.1]{Zhu94} and \cite[Problem 2.2]{Hua16}. Our result provides a necessary step for the construction of higher genus rational (Euclidean) CFT in the sense of Segal \cite{Seg88}. We hope it will also benefit other approaches to Euclidean CFT, such as the functional analytic one in \cite{Ten17,Ten19a,Ten19b,Ten19c}.

Conformal blocks were first studied by physicists \cite{BPZ84,FS87,MS89}. In mathematics, conformal blocks were defined and explored by Tsuchiya-Kanie \cite{TK88} for  Weiss-Zumino-Witten (WZW) models and genus $0$ Riemann surfaces, and were generalized by Tsuchiya-Ueno-Yamada to stable curves of all genera in \cite{TUY89}. For minimal models, conformal blocks over any stable curve was studied by  Beilinson-Feigein-Mazur \cite{BFM91}. In particular, factorization was proved for these models in \cite{TUY89,BFM91}.   A definition of conformal blocks for quasi-primarily generated VOAs was given by Zhu in \cite{Zhu94}, and was generalized to any ($\Nbb$-graded) VOA on families of compact Riemann surfaces by Frenkel and Ben-Zvi in \cite{FB04}. In \cite{DGT19a,DGT19b}, Damiolini-Gibney-Tarasca defined conformal blocks for VOAs associated to (algebraic) families of stable curves, and showed that when $\Vbb$ is $C_2$-cofinite and rational,  sheaves of conformal blocks are vector bundles with projectively flat connections, and proved the factorization property and a (formal) sewing theorem. Their results generalize those of Tsuchiya-Ueno-Yamada \cite{TUY89} for WZW models, those of Nagatomo-Tsuchiya \cite{NT05} for genus $0$ curves, and part of the results of Huang for genus $0$ and $1$ curves \cite{Hua95,Hua05a,Hua05b}. In \cite{DGT19c}, the authors proved Chern characters of these bundles give cohomological field theories, and called CFT-type $C_2$-cofinite rational VOAs of CohFT-type. In the conformal net approach to CFT, conformal blocks were defined and investigated by Bartels-Douglas-Henriques \cite{BDH17}.

\subsection*{Sewing conjecture}

Conformal blocks are the chiral halves of correlation functions of full CFT. Assume throughout this article that $\Vbb$ is a positive-energy vertex operator algebra (VOA), i.e. the $L_0$-grading on $\Vbb$ is $\Vbb=\bigoplus_{n\in\Nbb}\Vbb(n)$ where each $\Vbb(n)$ is finite-dimensional. For an $N$-pointed  compact Riemann surface $\fk X=(C;x_1,\dots,x_N)$ where  $x_1,\dots,x_N$ are distinct points of a compact Riemann surface $C$\footnote{We do not assume $C$ to be connected},  one associates to these points $\Vbb$-modules $\Wbb_1,\dots,\Wbb_N$. Then a conformal block $\uppsi$ is a  linear functional on $\Wbb_\blt=\Wbb_1\otimes\cdots\otimes\Wbb_N$ ``invariant" under the actions of certain global sections  related to $\Vbb$. If one has a family of $N$-pointed compact Riemann surfaces parametrized by a base manifold $\mc B$, then one allows $\uppsi$ to vary holomorphically over $\mc B$. The vertex operator $Y(\cdot,z)$ can be regarded as a conformal block associated to $\Pbb^1$ with distinct points $0,z,\infty$ and modules $\Vbb,\Vbb,\Vbb'$ (where $\Vbb'$ is the contragredient module of $\Vbb$). More generally, if $\Wbb_1,\Wbb_2,\Wbb_3$ are $\Vbb$-modules, an intertwining operator $\mc Y(\cdot,z)$ of type $\Wbb_3\choose\Wbb_1\Wbb_2$ (as defined in \cite{FHL93}) corresponds to a conformal block associated to $(\Pbb^1,0,z,\infty)$ and $\Wbb_1,\Wbb_2,\Wbb_3'$.

A fundamental problem in rational CFT is to prove the convergence of sewing conformal blocks.  Suppose  we have an $(N+2)$-pointed compact Riemann surface
\begin{align*}
\wtd{\fk X}=(\wtd C;x_1,\dots,x_N,x',x'').
\end{align*}
Then we can sew $\wtd{\fk X}$ along the pair of points $x',x''$ to obtain another Riemann surface with possibly higher genus. More precisely, we choose $\xi,\varpi$ to be local coordinates of $\wtd C$ at $x',x''$. Namely, they are univalent  (i.e. holomorphic and injective) functions defined respectively in neighborhoods $U'\ni x',U''\ni x''$ satisfying $\xi(x')=0,\varpi(x'')=0$. For each $r>0$ we let $\mc D_r=\{z\in\Cbb:|z|<r\}$ and $\mc D_r^\times=\mc D_r-\{0\}$. We choose $r,\rho>0$ so that the neighborhoods $U',U''$ can be chosen to satisfy that $\xi(U')=\mc D_r$ and $\varpi(U'')=\mc D_\rho$, that  $U'\cap U''=\emptyset$, and that none of $x_1,\dots,x_N$ is in $U'$ or $U''$. Then, for each $q\in\mc D_{r\rho}^\times$, we remove the closed subdiscs of $U',U''$ determined respectively by $|\xi|\leq \frac {|q|}\rho$ and $|\varpi|\leq \frac{|q|}r$, and glue the remaining part using the relation $\xi\varpi=q$. Then we obtain an $N$-pointed compact Riemann surface
\begin{align*}
\fk X_q=(\mc C_q;x_1,\dots,x_N)
\end{align*}
which clearly depends on $\xi$ and $\varpi$.

Now, if we associate finitely-generated $\Vbb$-modules $\Wbb_1,\dots,\Wbb_N,\Mbb,\Mbb'$ (where $\Mbb'$ is the contragredient (i.e. dual) module of $\Mbb$) to $x_1,\dots,x_N,x',x''$, and choose a conformal block $\uppsi$ associated to $\wtd{\fk X}$ and these $\Vbb$-modules, then its \textbf{sewing} $\mc S\uppsi$ is an $\Wbb_\blt^*=(\Wbb_1\otimes\cdots\Wbb_N)^*$-valued formal series of $q$ defined by sending each $w_\blt=w_1\otimes\cdots\otimes w_N\in\Wbb_\blt$ to
\begin{align*}
\mc S\uppsi(w_\blt)=\uppsi(w_\blt\otimes q^{L_0}\btr\otimes\btl)\quad\in\Cbb\{q\}[\log q]
\end{align*}
where $\btr\otimes\btl$ is the element of the ``algebraic completion" of $\Mbb\otimes\Mbb'$ corresponding to the identity element of $\End_\Cbb(\Mbb)$, and $L_0$ is the zero mode of the Virasoro operators $\{L_n:n\in\Zbb\}$. (Cf. \cite{Seg88,Vafa87,TUY89,Hua97,DGT19b}.) The \textbf{sewing conjecture}, as proposed in \cite[Conj. 8.1]{Zhu94} and \cite[Problem 2.2]{Hua16}, says that if $\Vbb$ satisfies nice properties (such as the $C_2$-cofiniteness), then $\mc S\uppsi(w_\blt)$ converges absolutely to a (possibly) multivalued function on $\mc D_{r\rho}^\times$. Moreover, for each $q\in\mc D_{r\rho}^\times$, $\mc S\uppsi(\cdot,q)$ defines a conformal block associated to $\fk X_q$ and $\Wbb_1,\dots,\Wbb_N$. If we sew $\wtd C$ along $n$ pairs of points, and if we let  $x_1,\dots,x_N$ and  $\wtd C$ and $\uppsi$    vary and be parametrized holomorphically by variables $\tau_\blt=(\tau_1,\dots,\tau_m)$, then  sewing  conformal blocks is also absolutely  convergent with respect to $q_1,\dots,q_n$ and (locally) uniform with respect to $\tau_\blt$. 

In this article, we give a complete proof of the sewing conjecture (see Section \ref{lb56} for the main result), which complements the results of \cite{DGT19b} on local freeness and factorization of sheaves of conformal blocks, and provides a necessary step of constructing rational conformal field theories on arbitrary (families of) compact Riemann surfaces. We remark that a sewing theorem (Thm. 8.5.1) was proved in \cite{DGT19b}. In that theorem, one treats the (infinitesimal) formal disc $\mathrm{Spec}\mbb C[[q]]$ instead of the analytic disc $\mc D_{r\rho}$, which is sufficient for application  in the algebraic category.  In particular, the convergence of sewing is not needed and not proved in \cite{DGT19b}. In the analytic category, which is the one we are working in, the convergence is necessary. (Also, note  that $\mc D_{r\rho}$ is not well-defined in the algebraic category.) 

Historically, the sewing conjecture was proved in some special cases. Our result is general in the following aspects:
\begin{enumerate}[label=(\alph*)]
\item We consider any $\Nbb$-graded $C_2$-cofinite VOA $\Vbb$ and any finitely-generated (admissible) $\Vbb$-module.
\item We consider  compact $N$-pointed Riemann surfaces of all genera.
\item We consider any (analytic) coordinates $\xi,\varpi$ at $x',x''$, and prove the convergence on $\mc D_{r\rho}$ whenever $r,\rho$ satisfy the previously described conditions. (Namely, we  prove the convergence not only when $q$ is small.)
\item We consider sewing along several pairs of points, and allow $\wtd{\fk X}$ and $\uppsi$ to be parametrized holomorphically by some $\tau_\blt$. 
\end{enumerate}
To our knowledge, no previous results have covered all these aspects. Nevertheless, even those partial results have played extremely important roles in the development of a rigorous mathematical theory of conformal field theory. For instance: convergence in the genus $0$ case is necessary for the statement of braiding and operator product expansions (fusion) of intertwining operators \cite{TK88,Hua05a}; convergence of self-sewing a $3$-pointed $\Pbb^1$ (which leads to a $1$-pointed elliptic curve) is necessary for the statement of modular invariance of VOA characters \cite{Zhu96}; convergence of sewing a general $N$-pointed $\Pbb^1$ is necessary for the proof of Verlinde conjecture and the rigidity and modularity of the tensor category of VOA modules \cite{Hua05b,Hua08a,Hua08b}.

\subsection*{History of the proof of convergence}

In \cite{TK88}, Tsuchiya and Kanie proved for type $A_1$ WZW models the convergence of the  products of intertwining operators, i.e. the convergence of sewing conformal blocks of such VOAs from (possibly disconnected) genus $0$ to genus $0$ Riemann surfaces. Their method applies directly to any WZW model. The local coordinates $\xi,\varpi$ in their result are $z,z^{-1}$ at $0,\infty$.  Their method is to show that the formal series $\mc S\uppsi$ satisfies a differential equation (the Knizhnik–Zamolodchikov equation) with simple pole at $q=0$, and the coefficients of the differential equation is a (matrix-valued) analytic function of $q$ and $\tau_\blt$ (parametrizing $\wtd{\fk X}$ and $\uppsi$). Any result about convergence proved after \cite{TK88}, including ours, follows this pattern. The difficulty is, of course,  to find such differential equations. 

In \cite{TUY89}, Tsuchiya-Ueno-Yamada showed that for all WZW-models and all compact Riemann surfaces, there exist local coordinates $\xi,\varpi$ at $x',x''$ (which are called $z,w$ and described in \cite[Lemma 6.1.2]{TUY89} and \cite[Lemma 5.3.1]{Ueno97}) such that such differential equations exist for small $q$. This leads immediately to the convergence of $\mc S\uppsi$ under those conditions, which was later explicitly claimed in \cite[Thm. 5.3.4]{Ueno97}. Generalizing their result  to any $\xi,\varpi$ is not  straightforward.

In \cite{Zhu96}, Zhu proved the convergence for any CFT-type $C_2$-cofinite VOA and ordinary $\Vbb$-modules,  for  self-sewing an $(N+2)$-pointed $\Pbb^1$ along $0,\infty$ with respect to local coordinates $z,z^{-1}$ to an $N$-pointed elliptic curve,  assuming that the $N$  $\Vbb$-modules $\Wbb_1,\dots,\Wbb_N$ not associated $0,\infty$ are the vacuum module $\Vbb$.  Its generalization to any $\Vbb$-modules $\Wbb_1,\dots,\Wbb_N$ is non-trivial and proved by Huang in \cite{Hua05b}. When  $\Wbb_1,\dots,\Wbb_N$ are grading-restricted generalized $\Vbb$-modules  (equivalently, finitely-generated admissible $\Vbb$-modules \cite{Hua09}), the convergence was proved by Fiordalisi in \cite{Fio16} involving pseudo-traces. (Note that our results do not cover pseudo-traces.)  Meanwhile, for $C_2$-cofinite VOAs and ordinary $\Vbb$-modules, for sewing several pointed $\Pbb^1$ to a pointed  $\Pbb^1$, and assuming the local coordinates at sewing points are $z$ or $z^{-1}$, Huang proved in \cite{Hua05a} the convergence of sewing conformal blocks for any $\Vbb$-modules. This result, with the help of \cite{Hua98}, can be generalized to any local coordinates $\xi,\varpi$ at sewing points. When $\Wbb_1,\dots,\Wbb_N$ are grading-restricted generalized $\Vbb$-modules, the convergence is due to Huang-Lepowski-Zhang \cite{HLZ11}.  The solution of sewing conjecture is more or less complete in genus $0$. When $\Vbb$ is holomorphic, the convergence of sewing with respect to suitable local coordinates was proved by \cite{Cod19}.

\subsection*{Idea of the proof}

Our proof of convergence is  motivated by \cite{DGT19b}. To explain the idea, we assume for simplicity that the $\Vbb$-modules are semisimple. It was shown in \cite{DGT19b} that by sewing a conformal block $\uppsi$, we get a formal conformal block $\mc S\uppsi$ as a formal series of $q$ (Thm. 8.5.1)  annihilated by $\nabla_{q\partial_q}$ (Rem. 8.3.3) where $\nabla$ is a connection of the sheaf of conformal blocks on the infinitesimal disk $\mathrm{Spec}(\Cbb[[q]])$ defined in \cite{DGT19a}. The definition of $\nabla$ is unique up to a projective term. In other words, if we choose $\nabla$ to be an (analytic) connection defined on the analytic disc $\mc D_{r\rho}$, then $\mc S\uppsi$ is annihilated by $\nabla_{q\partial_q}$ plus a projective term $f$ which  is a priori only a formal power series of $q$. A key step of proving the convergence of $\mc S\uppsi$ is to show that $f$ converges. Then, using a finiteness theorem \ref{lb21} analogous to \cite[Thm. 8.4.2]{DGT19b}, we obtain the desired differential equation.

It turns out that the connection $\nabla$  is determined by a (relative) projective structure $\fk P$ on $\wtd{\fk X}$. Moreover,   when $\xi,\varpi$ belong to $\fk P$, the projective term $f$ equals $0$. Thus, for a chosen $\fk P$ and the corresponding $\nabla$, if we assume $\xi,\varpi$ belong to $\fk P$, then $\nabla_{q\partial_q}\mc S\uppsi=0$, which will provide the differential equation.\footnote{This observation is due to Liang Kong and Hao Zheng.} The vanishing of $f$ is due to that of the Schwarzian derivatives between local coordinates belonging to $\fk P$. In general, one cannot expect that $\xi,\varpi$ belong to the same projective structure. To resolve this issue, we fix a projective structure $\fk P$, and find an explicit formula of $f$ in terms of $\fk P,\xi,\varpi$, and the other local coordinates. That formula (see \eqref{eq91}) shows that $f$ is analytic.

\subsection*{Outline}

To carry out the above ideas, we  first define and study some basic properties of sheaves of conformal blocks on complex curves and analytic families of curves. This is achieved in Sections 3 and 6. To prepare for this task, we first review Huang's change of coordinate formulas \cite{Hua97} in Section 1. This formula is used to define sheaves of VOAs on curves and families of curves in Sections 2 and 5. In Section 4, we follow \cite[Sec. 6.1]{TUY89} and give a precise description of how to sew a compact Riemann surface and form a family of curves over the disc $\mc D_{r\rho}$. In fact, we describe the simultaneous sewing for a family of compact Riemann surfaces over a complex manifold $\wtd{\mc B}$, which yield a family of complex curve over $\mc B:=\wtd{\mc B}\times\mc D_{r\rho}$. 

In Section 7, we prove a finiteness theorem which will turn the relation $(\nabla_{q\partial_q}+f)\mc S\uppsi=0$ into a differential equation with analytic coefficients (provided that $f$ converges) and simple poles at $q=0$. Section 8 recalls some basic facts about Schwarzian derivatives, and Section 9 prepares for the calculation of the projective term $f$. In Section 10, we give a proof that $\mc S\uppsi$ is a formal conformal block using an argument similar to but slightly different from the one in \cite{DGT19b}. Then, in Section 11,  we prove the convergence of sewing conformal blocks associated to a family $\wtd{\fk X}$ of compact Riemann surfaces along a pair of (families of) points. This result is generalized to sewing along several pairs of points in Section 13. In particular, our convergence theorem in the most general form is given in Theorem \ref{lb55} of that section. In Section 12, we show that the sewing map $\uppsi\mapsto\mc S\uppsi$ (defined in a suitable and natural way) is injective, assuming $\Vbb$ is $C_2$-cofinite and the modules for sewing are semisimple.  If $\Vbb$ is also CFT-type and rational, then this map is also bijective due to the factorization property proved in \cite{DGT19b}.  Using these results, we give an analytic version of the sewing theorem (Thm. 8.5.1) of \cite{DGT19b}.

We would like to point out that although many ideas are common in the analytical and algebraic settings, there are some subtle differences that will lead to different proof strategies. See Remark \ref{lb54} for instance.

\subsection*{Acknowledgment}

I would like to thank Chiara Damiolini, Angela Gibney, Yi-Zhi Huang, Liang Kong, Robert McRae, Nicola Tarasca, Chuanhao Wei, Shilin Yu, Hao Zheng for helpful discussions. I am grateful to Xiaojun Wu for answering my many questions in complex geometry. Section 12 is motivated by a  \href{https://mathoverflow.net/questions/300460/linear-independence-of-genus-one-correlation-functions/321362#321362}{Mathoverflow question}  asked by Andr\'e Henriques. The author is partially supported by an AMS-Simons travel grant.

\section{Change of coordinates}\label{lb58}

Throughout this article, we let $\Nbb=\{0,1,2,\dots\}$ and $\Zbb_+=\{1,2,3,\dots\}$. Also, $\Cbb^\times=\Cbb-\{0\}$. If $W$ is a vector space and $z$ is a (formal) variable, we \index{z@$[[z]],[[z^{\pm 1}]],((z)),\{z\}$} define 
\begin{gather*}
W[[z]]=\bigg\{\sum_{n\in\mathbb N}w_nz^n:\text{each }w_n\in W\bigg\},\\
W[[z^{\pm 1}]]=\bigg\{\sum_{n\in\mathbb Z}w_nz^n:\text{each }w_n\in W\bigg\},\\
W((z))=\Big\{f(z):z^kf(z)\in W[[z]]\text{ for some }k\in\mbb Z \Big\},\\
W\{z\}=\Big\{\sum_{n\in\mbb C}w_nz^n :\text{each $w_n\in W$}\Big\}.
\end{gather*}
$W[\log z]\{z\}$ and $W\{z\}[\log z]$ are understood in the obvious way by treating  $z$ and $\log z$ as two unrelated formal variables. In particular, they are subspaces of $W\{\log z,z\}$. On   $W\{\log z,z\}$, we define a linear operator $\partial_z$ by (choosing $w\in W$)
\begin{align}
\partial_z(w\cdot z^k(\log z)^l)=w\cdot(kz^{k-1}(\log z)^l+lz^{k-1}(\log z)^{l-1}). \label{eq101}	
\end{align}
Clearly $z\partial_z$ preserves the subspaces $W[\log z]\{z\}$ and $W\{z\}[\log z]$.

Let $\Vbb$ be a vertex operator algebra (VOA for short) in the sense of \cite{FHL93}. We let $\id$ \index{1@$\id$} and $\cbf$ \index{c@$\cbf$} be respectively the vacuum vector and the conformal vector of $\Vbb$. For each $u\in\Vbb$, we write the vertex operator as $Y(v,z)=\sum_{n\in\Zbb}Y(v)_nz^{-n-1}$ where each $Y(v)_n\in\End(\Vbb)$. Then $\{L_n=Y(\cbf)_{n+1} \}$ are the Virasoro operators with a central charge  $c\in\Cbb$. We write $\Vbb=\bigoplus_{n\in\Zbb}\Vbb(n)$ where $\Vbb(n)$ is the eigenspace of $L_0$ with eigenvalue $n$. We write $\wt(v)=n$ if $v\in\Vbb(n)$. We assume $L_0$ has no negative eigenvalues (i.e., $\Vbb=\oplus_{n\in\Nbb}\Vbb(n)$), and each $\Vbb(n)$ is finite dimensional.

In this article,  a \textbf{$\Vbb$-module} $\Wbb$ (with vertex operators $Y_\Mbb(v,z)=\sum_{n\in\Zbb}Y_\Mbb(v)_nz^{-n-1}$ and Virasoro operators $L_n=Y_\Mbb(\cbf)_{n+1}$) means a \textbf{finitely-admissible $\Vbb$-module}. This means that $\Wbb$ is a weak $\Vbb$-module in the sense of \cite{DLM97}, that $\Wbb$ is equipped with a diagonalizable operator $\wtd L_0$ satisfying $[\wtd L_0,Y_\Wbb(v)_n]=Y_\Wbb(L_0 v)_n-(n+1)Y_\Wbb(v)_n$, \index{L0@$\wtd L_0$}  that the eigenvalues of $\wtd L_0$ are in $\Nbb$, and that each eigenspace $\Wbb(n)$ is finite-dimensional. (Without the last finite-dimension condition, $\Wbb$ is an admissible (i.e. $\Nbb$-gradable) module in the usual sense.) Let \index{W@$\Wbb(n),\Wbb_{(n)}$}
\begin{align*}
\Wbb=\bigoplus_{n\in\Nbb}\Wbb(n)	
\end{align*}
be the grading given by $\wtd L_0$, then \begin{align}
Y_\Wbb(v)_m\Wbb(n)\subset \Wbb(n+\wt(v)-m-1).\label{eq40}
\end{align}

In \eqref{eq40}, by taking $v=\cbf$, we see that  each $\Wbb(n)$ is $L_0$-invariant, or equivalently, $[\wtd L_0,L_0]=0$. We can decompose $L_0|_{\Wbb(n)}$ and hence $L_0$ into mutually commuting the semi-simple part and nilpotent part $L_0=\Lss+\Lni$.  \index{L0@$\Lss,\Lni$}  $\wtd L_0$ clearly commutes with $\Lss$ and $\Lni$. We let
\begin{align*}
\Wbb=\bigoplus_{n\in\Cbb}\Wbb_{[n]}	
\end{align*}
be the $\Lss$-grading of $\Wbb$. Then we have (cf. \cite[Prop. 2.19]{HLZ14})
\begin{align}
	Y_\Wbb(v)_m\Wbb_{[n]}\subset \Wbb_{[n+\wt(v)-m-1]},
\end{align}
or equivalently, $[\Lss,Y_\Wbb(v)_n]=Y_\Wbb(L_0 v)_n-(n+1)Y_\Wbb(v)_n$. By Jacobi identity, the same relation holds when $L_{0,s}$ is replaced by $L_0$. So both $\wtd L_0-\Lss$ and $\Lni=L_0-\Lss$ commute with the action of $\Vbb$ on $\Wbb$. Thus, we see that if $\Wbb$ is finitely generated, then there is $K\in\Nbb$ such that $\Lni^kw=0$ for all $w\in\Wbb$.

We say that $\Wbb$ is \textbf{finitely $\Lss$-semisimple} if there exists a finite set $E\subset\Cbb$ such that 
\begin{align*}
\Wbb=\bigoplus_{n\in E+\Nbb}\Wbb_{[n]},\qquad \dim\Wbb_{[n]}<+\infty.
\end{align*}
In this case, we can choose $E$ such that any two elements of $E$ do not differ by an integer. Then for each $\alpha\in E$, $\Wbb_\alpha=\bigoplus_{n\in\alpha+\Nbb}\Wbb_{[n]}$ is a submodule of $\Wbb$. Such $\Wbb_\alpha$ is called \textbf{$\Lss$-simple}. Thus, every finitely $\Lss$-semisimple $\Vbb$-module is a finite direct sum of $\Lss$-simple ones. It is clear that $\Wbb$ is $\Lss$-simple if and only if  $\wtd L_0-\Lss$ is a constant.

A vector $w\in\Wbb$ is called $\wtd L_0$- (resp. $\Lss$-) homogeneous with weight $n$ if $w\in\Wbb(n)$ (resp. $w\in\Wbb_{[n]}$). In this case, we write $\wtd\wt(w)=n$ (resp. $\wt(w)=n$). \index{wt@$\wt,\wtd\wt$} Note that the $\wtd L_0$-weights are natural numbers but the $\Lss$-weights are not necessarily.  

Warning: For each finitely-admissible $\Vbb$-module $\Wbb$, the choice of $\wtd L_0$ is not unique. We fix $\wtd L_0$ for each $\Wbb$ obeying the following rule:

\begin{cv}\label{lb1}
For the vacuum module $\Vbb$, we choose $\wtd L_0$ to be $L_0$. So $\Vbb(n)=\Vbb_{(n)}$. If $\Wbb$ is a $\Vbb$-module with simple $\Lss$-grading, we choose $\wtd L_0$ such that it equals a constant plus $\Lss$. If $\Wbb$ is finitely $\Lss$-semisimple, we choose $\wtd L_0$ such that for each $\alpha\in E$ as  above, $\wtd L_0$ leaves $\Wbb_\alpha$ invariant and equals $\Lss$ plus a constant.
\end{cv}

We now recall the formula for changing coordinates discovered by Huang \cite{Hua97}. To begin with, we let $\scr O_{\Cbb,0}$ be the stalk of the sheaf of holomorphic functions of $\Cbb$ at $0$. Namely, an element in $\scr O_{\Cbb,0}$ is precisely a formal power series $f(z)=\sum_{n\in\Nbb}a_nz^n$ ($a_n\in\Cbb$) converging absolutely in a neighborhood of $0$. We consider the subset $\Gbb$ \index{G@$\Gbb$} of all $\rho\in\scr O_{\Cbb,0}$ satisfying $\rho(0)=0$ and $\rho'(0)\neq 0$. Then $\Gbb$ becomes a group if we define the multiplication of two elements $\rho_1,\rho_2$ to be their composition $\rho_1\circ\rho_2$. The identity element of $\Gbb$ is the standard coordinate $z$ of $\Cbb$.

For each $\rho\in\Gbb$, we can find $c_0,c_1,c_2,\dots\in\Cbb$ such that
\begin{align*}
\boxed{~~\rho(z)=c_0\cdot\exp\Big(\sum_{n>0}c_nz^{n+1}\partial_z \Big)z~~}
\end{align*}
For instance, if we write
\begin{align}
\rho(z)=a_1z+a_2z^2+a_3z^3+\cdots,\label{eq2}
\end{align}
then one has
\begin{gather}
c_0=a_1,\nonumber\\
c_1c_0=a_2,\nonumber\\
c_2c_0+c_1^2c_0=a_3.\nonumber
\end{gather}
In particular, one has
\begin{align*}
c_0=\rho'(0).
\end{align*}
We define \index{U@$\mc U(\rho)$} $\mc U(\rho)\in\End(\Wbb)$ to be
\begin{align}
\boxed{~~\mc U(\rho)=\rho'(0)^{\wtd L_0}\cdot \exp\Big(\sum_{n>0}c_n L_n\Big)~~}\label{eq1}
\end{align}
Notice $a_n=\rho^{(n)}(0)/n!$, we have
\begin{gather}
c_1=\frac 12\frac{\rho''(0)}{\rho'(0)},\nonumber\\
c_2=\frac 16 \frac{\rho'''(0)}{\rho'(0)}-\frac 14\Big(\frac{\rho''(0)}{\rho'(0)}\Big)^2.\label{eq48}
\end{gather}

\begin{rem}
Considering the action of $\mc U(\varrho)$ on $\Wbb$ might be inconvenient since $\Wbb$ is not finite-dimensional. On the other hand, $\Wbb(n)$ might not be preserved by $\mc U(\rho)$. Thus, it would be better to consider $\Wbb^{\leq n}=\bigoplus_{k\leq n}\Wbb(k)$\index{VW@$\Vbb^{\leq n},\Wbb^{\leq n}$} which is finite-dimensional and preserved by $\mc U(\varrho)$.

Since $L_m\Wbb^{\leq n}\subset\Wbb^{\leq n-1}$ when $m>0$, from \eqref{eq1} it is easy to see that for any $w\in\Wbb(n)$,
\begin{align}
\mc U(\rho)w=\rho'(0)^nw~~\mod~~ \Wbb^{n-1}.\label{eq5}
\end{align}
In other words, the action of $\mc U(\rho)$ on $\Wbb^{\leq n}/\Wbb^{\leq n-1}$ is $\rho'(0)^n\id$.
\end{rem}

The following was (essentially) proved in \cite{Hua97} section 4.2:
\begin{thm}
	For each $\Vbb$-module $\Wbb$, $\mc U$ is a representation of $\Gbb$ on $\Wbb$. Namely, we have $\mc U(\rho_1\circ\rho_2)=\mc U(\rho_1)\mc U(\rho_2)$ for each $\rho_1,\rho_2\in\Gbb$.
\end{thm}

\begin{eg}
It is easy to see that $(c_1z^2\partial_z)^nz=n!c_1^nz^{n+1}$. Thus $\exp(c_1z^2\partial_z)z=\sum_{n=0}^{\infty}c_1^nz^{n+1}=z/(1-c_1z)$. For each $\xi\in\Cbb^\times$, we set $\upgamma_\xi\in\Gbb$ \index{zz@$\upgamma_\xi$} to be
\begin{align}
\upgamma_\xi(z)=\frac 1{\xi+z}-\frac 1\xi.
\end{align}
If we set $\alpha(z)=-\xi^{-2}z$, then $\upgamma_\xi(z)=\alpha/(1-\xi \alpha)=\exp(c_1\alpha^2\partial_\alpha)(\alpha)$. Thus, by \eqref{eq1} and that $\mc U$ preserves composition, we obtain
\begin{align}
\mc U(\upgamma_\xi)=e^{\xi L_1}(-\xi^{-2})^{\wtd L_0}.
\end{align}
In particular,
\begin{align}
\mc U(\upgamma_1)=e^{L_1}(-1)^{\wtd L_0}.
\end{align}
It is easy to see $\upgamma_\xi(\xi z)=\xi^{-1}\upgamma_1(z)$. Thus
\begin{align}
\mc U(\upgamma_\xi)\xi^{\wtd L_0}=\xi^{-\wtd L_0}\mc U(\upgamma_1).\label{eq22}
\end{align}
\end{eg}

\begin{rem}\label{lb2}
Let $X$ be a complex manifold and $\rho:X\rightarrow\mbb G,x\mapsto \rho_x$ a function. We say that $\rho$ is a \textbf{holomorphic family} of transformations if for any $x\in X$, there exists an open subset $V\subset X$ containing $x$ and an open subset $U\subset\Cbb$ containing $0$ such that  $(z,y)\in U\times V\mapsto \rho_y(z)$ is a holomorphic function on $U\times V$. Then it is clear that the coefficients $a_1,a_2,\dots$ in \eqref{eq2} depend holomorphically on the parameter $x\in X$. Hence the same is true for $c_0,c_1,c_2,\dots$. Thus, by the formula \eqref{eq1}, for any $w\in\Wbb^{\leq n}$, $x\in X\mapsto \mc U(\rho_x)w$ is a $\Wbb^{\leq n}$-valued holomorphic function on $X$. Thus $\mc U(\rho)$ can be regarded as an isomorphism of  $\scr O_X$-modules 
\begin{align}
\mc U(\rho): \Wbb^{\leq n}\otimes_{\Cbb}\scr O_X\xrightarrow{\simeq}\Wbb^{\leq n}\otimes_{\Cbb}\scr O_X
\end{align}
sending each $\Wbb^{\leq n}$-valued function $w$ to the section $x\mapsto \mc U(\rho_x)w(x)$. \index{U@$\mc U(\rho)$} Its inverse is  $\mc U(\rho^{-1})$.
\end{rem}

If $\Wbb$ is a $\Vbb$-module, then its \textbf{contragredient module} $\Wbb'$ \index{W'@$\Wbb'$} can be describe using $\mc U(\upgamma_z)$. As a vector space,
\begin{align*}
\Wbb'=\bigoplus_{n\in\Cbb}\Wbb{(n)}^*
\end{align*}
where $\Wbb{(n)}^*$ is the dual space of $\Wbb{(n)}$. For each $v\in\Vbb$, the vertex operator $Y_{\Wbb'}(v,z)$ is defined such that if $w\in\Wbb,w'\in\Wbb'$, then, using $\bk{,}$ to denote the natural pairing of $\Wbb$ and $\Wbb'$, we have 
\begin{align}
\bk{Y_{\Wbb'}(v,z)w',w}=\bk{w',Y_\Wbb(\mc U(\upgamma_z)v,z^{-1})w},\label{eq62}
\end{align}
recalling that $\mc U(\upgamma_z)=e^{zL_1}(-z^{-2})^{L_0}$. That  $(\Wbb',Y_{\Wbb'})$ satisfies the definition of a $\Vbb$-module follows from \cite{FHL93}. We have 
\begin{align}
	\wtd L_0^\tr=\wtd L_0,\label{eq99}
\end{align}
i.e., $\bk{\wtd L_0w,w'}=\bk{w,\wtd L_0w'}$ for each $w\in\Wbb,w'\in\Wbb'$. Also, by choosing the $v$ in \eqref{eq62} to be the conformal vector, we see $L_n^\tr=L_{-n}$ and in particular $L_0^\tr=L_0$. So $\Lss^\tr=\Lss,\Lni^\tr=\Lni$.

\section{Sheaves of VOAs on complex curves}\label{lb5}

In this article, for any complex manifold or complex space $X$, we let $\scr O_X$ denote the sheaf (of germs) of holomorphic functions of $X$. So for each open subset $U\subset X$, $\scr O_X(U)$ (written also as $\scr O(U)$ for short) \index{OX@$\scr O_X,\scr O(U)$} is the algebra of holomorphic functions on $U$. For every (sheaf of) $\scr O_X$-modules $\scr E$, recall the usual notation that $\scr E_x$ \index{Ex@$\scr E_x$} is the stalk of $\scr E$ at $x\in X$. The dual $\scr O_X$-module of $\scr E$ is denoted by $\scr E^*$, or $\scr E^{-1}$ when $\scr E$ is a line bundle. For two $\scr O_X$-modules $\scr E,\scr F$, we write their tensor product $\scr E\otimes_{\scr O_X}\scr F$ as $\scr E\otimes\scr F$ for short. If $Y$ is a complex submanifold or complex subspace of $X$, we let $\scr E|Y$ (also written as $\scr E|_Y$) denote the restriction of $\scr E$ to $X$, namely, the pullback of $\scr E$ along the inclusion map $Y\hookrightarrow X$. \index{EY@$\scr E\lvert Y=\scr E\lvert_Y,\scr E\lvert x$} The restriction of a section $s$ of $\scr E$ is denoted by $s|X$ or $s|_X$. In the case $Y$ is a single point $\{x\}$, the restriction $\scr E|x$ can be naturally identified with $\scr E_x/\fk m_x\scr E_x$ where $\fk m_x$ the  ideal of all $f\in\scr O_{X,x}$ vanishing at $x$. If $s$ is a section defined near $x$, we let $s(x)$ be the restriction $s|x$. If we consider $s(x)$ as an element of $\scr E_x/\fk m_x\scr E_x$, then $s(x)$ equals $s_x+\fk m_x\scr E_x$, where $s_x\in\scr E_x$ is the germ of $s$ at $x$.

By a complex curve $C$, we mean  either a compact Riemann surface\footnote{Unless otherwise stated, compact Riemann surfaces are \emph{not} assumed to be connected.} or a (simple) nodal curve. For simplicity, we assume throughout this article that a nodal curve has only one (simple) node. We let $\omega_C$ denote the dualizing sheaf of $C$, which is the sheaf of holomorphic $1$-forms when $C$ is smooth (i.e., a compact Riemann surface). \index{zz@$\Theta_C,\omega_C$} Its dual sheaf is denoted by $\Theta_C=\omega_C^{-1}$, which is the (holomorphic) tangent bundle when $C$ is smooth. In the case that $C$ is nodal, $\scr O_C,\omega_C,\Theta_C$ are described as follows.

Assume $C$ has only one simple node. Then $C$ can be obtained by gluing two distinct points $y',y''$ of a compact Riemann surface $\wtd C$ (the normalization of $C$).\footnote{To simplify the following discussions, when $C$ is smooth, we let $\wtd C$ be $C$ and $\nu$ be the identity map.} The gluing map is denoted by $\nu:\wtd C\rightarrow C$. We identify $\wtd C-\{y',y''\}$ with $C-\{x'\}$ (where $x'=\nu(y')=\nu(y'')$) via $\nu$. Then  $\scr O_C(U),\omega_C(U),\Theta_C(U)$ agree with $\scr O_{\wtd C}(U),\omega_{\wtd C}(U),\Theta_{\wtd C}(U)$ when $x'\neq U$. If $x'\in U$ and $U$ is small enough such that $\nu^{-1}(U)$ is a disjoint union of neighborhoods $V'\ni y',V''\ni y''$, and that there exist univalent functions $\xi\in\scr O(V'),\varpi\in\scr O(V'')$ satisfying $\xi(y')=\varpi(y'')=0$, then $\scr O_C(U)$ consists of all $f\in\scr O_{\wtd C}(V'\cup V'')$ satisfying $f(y')=f(y'')$; $\Theta_C(U)$ is the (free) $\scr O_C(U)$-submodule of $\Theta_{\wtd C}(V'\cup V'')$ generated by the tangent fields whose restrictions to $V'\cup V''$ are 
\begin{align}
\xi\partial_\xi,\qquad \text{resp.}\qquad -\varpi\partial_\varpi;\label{eq9}
\end{align}   
$\omega_C(U)$ is the (free) $\scr O_C(U)$-submodule of $\omega_{\wtd C}(\nu^{-1}(U-\{x'\}))$ generated by
\begin{align*}
\xi^{-1}d\xi,\qquad \text{resp.}\qquad -\varpi^{-1}d\varpi.
\end{align*}  
We refer the reader to \cite[Chap.X]{ACG11} for  basic facts about nodal curves.

\subsection*{Definition of $\scr V_C$}

Let $\Vbb$ be a VOA. The \textbf{sheaf of VOA} $\scr V_C$ on a complex curve $C$ is defined when $C$ is smooth by \cite{FB04} and generalized to (simple) nodal curves by \cite{DGT19a}. Let us recall the definition.\index{VC@$\scr V_C,\scr V_C^{\leq n}$} $\scr V_C$ is defined by the filtration
\begin{align*}
\scr V_C=\varinjlim_{n\in\Nbb}\scr V_C^{\leq n},
\end{align*}
where each $\scr V_C^{\leq n}$ is a  locally free $\scr O_C$-module (i.e., a vector bundle) of rank $\dim\Vbb^{\leq n}$. We need some preparation before we describe $\scr V_C^{\leq n}$.

We first assume $C$ is a smooth curve or the smooth open subset of a nodal curve. Let $U,V$ be open subsets of $C$, equipped with univalent functions $\eta\in\scr O(U),\mu\in\scr O(V)$.   Define a holomorphic family $\varrho(\eta|\mu):U\cap V\rightarrow\Gbb$ \index{zz@$\varrho(\eta\lvert\mu)$} as follows. For any $p\in U\cap V$, $\eta-\eta(p)$ and $\mu-\mu(p)$ are local coordinates at $p$. We set $\varrho(\eta|\mu)_p\in\Gbb$ satisfying
\begin{align}
\eta-\eta(p)=\varrho(\eta|\mu)_p(\mu-\mu(p)).\label{eq3}
\end{align}
Let $z\in\scr O_{\Cbb,0}$ be the standard coordinate. Then, by composing both sides of \eqref{eq3} with $\mu^{-1}(z+\mu(p))$, we find the equivalent formula
\begin{align}
\varrho(\eta|\mu)_p(z)=\eta\circ\mu^{-1}(z+\mu(p))-\eta(p),
\end{align}
which justifies that the family of transformation $\varrho(\eta|\mu)$ is holomorphic. It is also clear that if $\eta_1,\eta_2,\eta_3$ are three local coordinates, then on their common domain the following cocycle condition holds:
\begin{align}
\varrho(\eta_3|\eta_1)=\varrho(\eta_3|\eta_2)\varrho(\eta_2|\eta_1).\label{eq19}
\end{align}
By Remark \ref{lb2}, for each $n\in\Nbb$ we have an isomorphism of $\scr O_{U\cap V}$-modules
\begin{align*}
\mc U(\varrho(\eta|\mu)):\Vbb^{\leq n}\otimes_\Cbb\scr O_{U\cap V}\xrightarrow{\simeq}\Vbb^{\leq n}\otimes_\Cbb\scr O_{U\cap V}.
\end{align*}

The vector bundle $\scr V_C^{\leq n}$ is defined such that its transition functions are given by $\mc U(\varrho(\eta|\mu))$. Thus, for any open subset $U\subset C$ and a univalent $\eta\in\scr O(U)$, we have a trivilization, i.e., an isomorphism of \index{U@$\mc U_\varrho(\eta)$} $\scr O_U$-modules 
\begin{align}\label{eq4}
\mc U_\varrho(\eta):\scr V_C^{\leq n}|_U\xrightarrow{\simeq} \Vbb^{\leq n}\otimes_{\Cbb}\scr O_U.
\end{align}
If $V\subset C$ is also open and  $\mu\in\scr O(V)$ is univalent, then  on $U\cap V$ we have
\begin{align}
\mc U_\varrho(\eta)\mc U_\varrho(\mu)^{-1}=\mc U(\varrho(\eta|\mu)).\label{eq30}
\end{align}
From \eqref{eq5}, we can compute that for any section $v$ of $\Vbb^{\leq n}\otimes_{\Cbb}\scr O_{U\cap V}$,
\begin{align}
\mc U_\varrho(\eta)\mc U_\varrho(\mu)^{-1}\cdot v=(\partial_\mu\eta)^n\cdot v~~\mod ~~\Vbb^{\leq n-1}\otimes_{\Cbb}\scr O_{U\cap V}.\label{eq8}
\end{align}
By comparing the transition functions, we see that $\scr V_C^{\leq n}/\scr V_C^{\leq n-1}$ is naturally equivalent to $\Vbb(n)\otimes_\Cbb\Theta_C^{\otimes n}$ (cf. \cite[Sec. 6.5.9]{FB04}).

We now assume $C$ has a node $x'$. We shall define $\scr V_C^{\leq n}$ to be an $\scr O_C$-submodule of $\scr V^{\leq n}_{C-\{x'\}}$ as follows. Let $\scr V_C^{\leq n}$ be equal to $\scr V^{\leq n}_{C-\{x'\}}$ outside $x'$. To describe $\scr V_C^{\leq n}$ near $x'$, we choose a neighborhood $U$ of $x'$, and choose $y',y'',V',V'',\xi\in\scr O(V'),\varpi\in\scr O(V'')$ as previously. Then $U-\{x'\}$ can be identified with $(V'-\{y'\})\cup (V''-\{y''\})$ via $\nu$. We let $\scr V_C^{\leq n}(U)$ be the $\scr O_C(U)$-submodule of $\scr V_{C-\{x'\}}^{\leq n}(U-\{x'\})$ generated by
\begin{align}
\boxed{~~\mc U_\varrho(\xi)^{-1}\big(\xi^{L_0}v\big)+\mc U_\varrho(\varpi)^{-1}\big(\varpi^{L_0}\mc U(\upgamma_1)v\big)~~}\qquad (\forall v\in\Vbb^{\leq n}).\label{eq6}
\end{align}
To be more precise, \eqref{eq6} defines a section on $(V'-\{y'\})\sqcup (V''-\{y''\})$ which equals $\mc U_\varrho(\xi)^{-1}\big(\xi^{L_0}v\big)$ on $(V'-\{y'\})$ and $\mc U_\varrho(\varpi)^{-1}\big(\varpi^{L_0}\mc U(\upgamma_1)v\big)$ on $V''$. Also, $\xi^{L_0}$ is an element of $\scr O_{\wtd C}(V'-\{y'\})$ acting on the constant section $v\in\Vbb^{\leq n}\subset\Vbb^{\leq n}\otimes_{\Cbb}\scr O_{\wtd C}(V'-\{y'\})$, and $\varpi^{L_0}\mc U(\upgamma_1)v$ is understood in a similar way.  It is easy to see that $\scr V_C^{\leq n}(U)$ is $\scr O_C(U)$-generated freely by \eqref{eq6} for all $v\in E$ where $E$ is any basis of $\Vbb^{\leq n}$ whose elements are homogenous. Since $\upgamma_1=\upgamma_1^{-1}$, $\scr V_C^{\leq n}(U)$ is also generated freely by
\begin{align}
\mc U_\varrho(\xi)^{-1}\big(\xi^{L_0}\mc U(\upgamma_1)v\big)+\mc U_\varrho(\varpi)^{-1}\big(\varpi^{L_0}v\big).
\end{align}
for all $v\in E$.

\begin{pp}\label{lb3}
Let $C$ be a complex curve and $n\in\Nbb$. Then we have the following isomorphism of  $\scr O_C$-modules:
\begin{align}
\scr V_C^{\leq n}/\scr V_C^{\leq n-1}\simeq\Vbb(n)\otimes_{\Cbb}\Theta_C^{\otimes n}.\label{eq7}
\end{align}
Under this isomorphism, if $U\subset C$ is open and smooth, and $\eta\in\scr O(U)$ is univalent, then for any $v\in\Vbb(n)$, $v\otimes \partial_\eta^n$ is identified with the equivalence class of $\mc U_\varrho(\eta)^{-1}v$.
\end{pp}

Cf. \cite[Sec.6.5.9]{FB04} and \cite{DGT19b}.

\begin{proof}
By the transition function \eqref{eq8}, we obtain a surjective $\scr O_{C\setminus\{x'\}}$-module morphism $\Psi:\scr V_{C\setminus\{x'\}}^{\leq n}\rightarrow\Vbb(n)\otimes\Theta_{C\setminus\{x'\}}^{\otimes n}$ sending $\mc U_\varrho(\eta)^{-1}v$ to  $v\otimes \partial_\eta^n$ if $v\in\Vbb(n)$, and to $0$ if $v\in\Vbb^{\leq n-1}$. $\Psi$ has kernel $\scr V_{C\setminus\{x'\}}^{\leq n-1}$. Now let $U$ be a neighborhood of $x'$ as in the setting of \eqref{eq6}. Then $\Psi$ sends \eqref{eq6} to $v\otimes \xi^n\partial_\xi^n|_{V'-\{y'\}}+v\otimes(-\varpi)^n\partial_\varpi^n|_{V''-\{y''\}}$ whenever $v\in\Vbb(n)\cap E$. (Recall that $E$ is a homogeneous basis of $\Vbb^{\leq n}$.) From this and \eqref{eq9} we see that $\Psi$ restricts to a surjective $\scr O_C$-module morphism $\Psi:\scr V_C^{\leq n}\rightarrow\Vbb(n)\otimes\Theta_C^{\otimes n}$ and that $\Ker\Psi(U)$ is $\scr O_C(U)$-generated by  \eqref{eq6} for all $v\in \Vbb^{\leq n-1}\cap E$. Thus $\Psi$ descends to an isomorphism \eqref{eq7}.
\end{proof}

\subsection*{A vanishing theorem}

In the remaining part of this section, we use Proposition \ref{lb3} to prove a vanishing theorem for $\scr V_C^{\leq n}\otimes\omega_C$. By an (analytic) \textbf{local coordinate $\eta$ of $C$ at} a smooth point $x$, we mean  a (holomorphic) univalent function $\eta$ defined on a smooth neighborhood of $x$ satisfying $\eta(x)=0$. By an \textbf{$N$-pointed complex curve with local coordinates}
\begin{align*}
\fk X=(C;x_1,\dots,x_N;\eta_1,\dots,\eta_N)
\end{align*}
we mean a complex curve $C$ together with $N$ distinct smooth points $x_1,x_2,\dots,x_N\in C$ such that each $x_i$ is associated with a local coordinate $\eta_i$ at $x_i$. If local coordinates are not specified, we simply say an $N$-pointed complex curve $\fk X=(C;x_1,\dots,x_N)$. Unless otherwise stated, \emph{we assume that each irreducible component of $C$} (equivalently, each connected component of $\wtd C$) \emph{contains at least one of $x_1,\dots,x_N$}. 

We set \index{SX@$\SX,\SX(b)$}
\begin{align*}
\SX=x_1+x_2+\cdots+x_N,
\end{align*}
which can be considered as a divisor both of $C$ and of $\wtd C$. Note that by Kodaira vanishing theorem, if $C$ is smooth and connected with genus $g$, and $\scr L$ is a line bundle on $C$, then $H^1(C,\scr L\otimes\omega_C)=0$ and (by Serre duality) equivalently  $H^0(C,\scr L^{-1})=0$ whenever $\deg\scr L>0$. Since $\deg\omega_C=2g-2$, we conclude that $H^1(C,\scr L)=0$ whenever $\deg\scr L>2g-2$.

Recall that $\wtd C$ is the normalization of the complex curve $C$. We let $\wtd g$ be the largest genus of the connected components of $\wtd C$. Let $M\in\{0,1\}$ be the number of nodes. Notice the following elementary fact:

\begin{lm}\label{lb4}
Choose any integer $n\geq -1$. Then $H^1(C,\Theta_C^{\otimes n}(k\SX))=0$ whenever  $k>(n+1)(2\wtd g-2)+2M$.
\end{lm} 

Recall that for any locally free $\scr E$ (i.e., $\scr E$ is a vector bundle), \index{SX@$k\SX,\blt\SX$} $\scr E(k\SX)\simeq\scr E\otimes\scr O(k\SX)$ is the sheaf of meromorphic sections of $\scr E$ whose only possible poles are at $x_1,\dots,x_N$ and of orders at most $k$. Also, this lemma can be easily generalized to curves with more than one node.

\begin{proof}
Assume first of all that $C$ is smooth. Then it suffices to assume $C$ is connected. Then $\wtd g$ is the genus $g$ of $C$. Since $\deg\Theta_C=2-2g$, the degree of $\Theta_C^{\otimes n}(k\SX)$ is no less than $n(2-2g)+k$, which is greater than $2g-2$ if $k>(n+1)(2g-2)$. In that case, we have $H^1(C,\Theta_C^{\otimes n}(k\SX))=0$ by the discussion before the lemma.

Next, assume $C$ has one node $x'=\nu(y')=\nu(y'')$. Choose any $k>(n+1)(2\wtd g-2)+2$. To prove the vanishing of $H^1$, it suffices, by Serre duality, to prove $H^0(C,\omega_C^{\otimes(n+1)}(-k\SX))=0$. Note that $H^0(C,\omega_C^{\otimes(n+1)}(-k\SX))$ can be viewed as a subspace of $H^0(\wtd C,\scr L)$ where $\scr L=\omega_{\wtd C}^{\otimes(n+1)}(-k\SX+y'+y'')$. We shall prove that for each connected component $\wtd C_0$ of $\wtd C$, $H^0(\wtd C_0,\scr L|_{\wtd C_0})=0$. This is true since
\begin{align*}
\deg(\scr L|_{\wtd C_0})\leq (n+1)\deg(\omega_{\wtd C_0})-k+2\leq (n+1)(2\wtd g-2)-k+2<0.
\end{align*}
\end{proof}

\begin{thm}\label{lb9}
Choose any $n\in\Nbb$, and assume $k>n\cdot\max\{0,2\wtd g-2\}+2M$. Then
\begin{align*}
H^1(C,\scr V_C^{\leq n}\otimes\omega_C(k\SX))=0.
\end{align*}
\end{thm}

\begin{proof}
Since $\Vbb^{\leq-1}$ is trivial, we have $\scr V_C^{\leq 0}\simeq \scr V_C^{\leq 0}/\scr V_C^{\leq -1}\simeq\Vbb(0)\otimes_\Cbb\scr O_C$ by Proposition \ref{lb3}. Thus, the claim follows from Lemma \ref{lb4}. Now, suppose the claim is true for $n-1$. We shall prove that it is also true for $n$. Assume $k>n\cdot\max\{0,2\wtd g-2\}+2M$.  By Proposition \ref{lb3}, we have an exact sequence
\begin{align*}
H^1(C,\scr V_C^{\leq n-1}\otimes\omega_C(k\SX))&\rightarrow H^1(C,\scr V_C^{\leq n}\otimes\omega_C(k\SX))\nonumber\\
&\rightarrow H^1\big(C,\Vbb(n)\otimes_{\Cbb}\Theta_C^{\otimes(n-1)}(k\SX)\big).
\end{align*}
The first term vanishes by induction, and the last term vanishes by Lemma \ref{lb4}. Thus the middle term vanishes.
\end{proof}

\section{Conformal blocks associated to complex curves}\label{lb15}

We recall the definition of conformal blocks associated to complex curves in \cite{FB04} and \cite{DGT19a,DGT19b}, but rephrase it in the analytic language suitable for our purpose in this article.

\subsection*{Action of $H^0(C,\scr V_C\otimes\omega_C(\blt S_{\fk X}))$ on $\Wbb_\blt$}

Let $\fk X=(C;x_1,\dots,x_N;\eta_1,\dots,\eta_N)$ be an $N$-pointed complex curve with local coordinates. If $\Wbb$ is a $\Vbb$-module, then, by considering $\Vbb$ as the subspace of constant sections of $\Vbb((z))$, we can extend the linear map $\Vbb\otimes_\Cbb\Wbb\rightarrow\Wbb((z)),v\otimes w\mapsto Y_\Wbb(v,z)w$ uniquely to an $\Cbb((z))$-module homomorphism  $\Vbb((z))\otimes_\Cbb\Wbb\rightarrow\Wbb((z))$.

Let $\Wbb_1,\Wbb_2,\dots,\Wbb_N$ be $\Vbb$-modules. Set \index{Ww@$\Wbb_\blt,w_\blt$} $\Wbb_\blt=\Wbb_1\otimes \Wbb_2\otimes\cdots\otimes \Wbb_N$.

\begin{cv}
	By $w\in\Wbb_\blt$, we mean a vector of $\Wbb_1\otimes\cdots\otimes \Wbb_N$.  By $w_\blt\in \Wbb_\blt$, we mean a vector of the form $w_1\otimes w_2\otimes\cdots\otimes w_N$, where $w_1\in \Wbb_1,\dots,w_N\in\Wbb_N$.
\end{cv}

For each $\scr O_C$-module $\scr E$,  we \index{SX@$k\SX,\blt\SX$} set
\begin{align*}
\scr E(\blt\SX)=\varinjlim_{k\in\Nbb}\scr E(k\SX),
\end{align*}
whose sections are meromorphic sections of $\scr E$ whose only possible poles are at $x_1,\dots,x_N$. For each $1\leq i\leq N$, if $v$ is a section of $\scr V_C\otimes\omega_C(\blt S_{\fk X})$ defined near $x_i$, we define a linear action of $v$ on $\Wbb_i$ as follows. Choose a neighborhood $U_i$ of $x_i$ on which $\eta_i$ is defined. By tensoring with the identity map of $\omega_{U_i}$,  the map \eqref{eq4} induces naturally an $\scr O_{U_i}$-module isomorphism (also denoted by $\mc U_\varrho(\eta_i)$):
\begin{align*}
\mc U_\varrho(\eta_i):\scr V_C|_{U_i}\otimes\omega_{U_i}(\blt S_{\fk X})\xrightarrow{\simeq}\Vbb\otimes_{\Cbb}\omega_{U_i}(\blt S_{\fk X}).
\end{align*}
Identify $U_i$ and $\eta_i(U_i)$ via $\eta_i$. Then $\eta_i$ as a variable equals the standard variable $z$ of $\Cbb$. The action of $v$ on any $w_i\in\Wbb_i$ \index{vw@$v\cdot w_i,v\cdot w_\blt$} is  then
\begin{align}
\boxed{~~v\cdot w_i=\Res_{z=0}Y_{\Wbb_i}(\mc U_\varrho(\eta_i)v,z)w_i~~}\label{eq26}
\end{align}
(Here the $\eta_i$ in $\mc U_\varrho(\eta_i)$ is understood as a coordinate but not a variable. So it is different from $z$.)

Define a linear action of $H^0(C,\scr V_C\otimes\omega_C(\blt S_{\fk X}))$ on $\Wbb_\blt$ as follows.  
If $v\in H^0(C,\scr V_C\otimes\omega_C(\blt S_{\fk X}))$ and $w_\blt=w_1\otimes\cdots\otimes w_N\in\Wbb_\blt$, then
\begin{align}
v\cdot w_\blt=\sum_{i=1}^N w_1\otimes w_2\otimes\cdots \otimes (v|_{U_i})\cdot w_i \otimes\cdots\otimes w_N.\label{eq27}
\end{align}

\subsection*{Coordinate-independent definition}

Let $\fk X=(C;x_1,\dots,x_N)$ be an $N$-pointed complex curve without specifying local coordinates, and let $\Wbb_1,\dots,\Wbb_N$ be $\Vbb$-modules. Define a vector space $\scr W_{\fk X}(\Wbb_\blt)$ isomorphic to $\Wbb_\blt$ as follows. $\scr W_{\fk X}(\Wbb_\blt)$ is a (possibly infinite rank) vector bundle on the $0$-dimensional manifold $\{C\}$ (consider as the base manifold of the family $C\rightarrow \{C\}$). \index{WX@$\scr W_{\fk X}(\Wbb_\blt)$} For any choice of local coordinates $\eta_\blt=(\eta_1,\dots,\eta_N)$ of $x_1,\dots,x_N$ respectively, we have a trivialization \index{U@$\mc U(\eta_\blt)$}
\begin{align}
\mc U(\eta_\blt):\scr W_{\fk X}(\Wbb_\blt)\xrightarrow{\simeq} \Wbb_\blt  \label{eq34}
\end{align}
such that if $\mu_\blt$ is another set of local coordinates, then the transition function is
\begin{align}
\mc U(\eta_\blt)\mc U(\mu_\blt)^{-1}=&\mc U(\eta_\blt\circ\mu_\blt^{-1})\nonumber\\
:=&\mc U(\eta_1\circ\mu_1^{-1})\otimes \mc U(\eta_2\circ\mu_2^{-1})\otimes\cdots\otimes \mc U(\eta_N\circ\mu_N^{-1}).\label{eq32}
\end{align}
If $v\in H^0(C,\scr V_C\otimes\omega_C(\blt S_{\fk X}))$ and $w\in\scr W_{\fk X}(\Wbb_\blt)$, we set
\begin{align}
v\cdot w=\mc U(\eta_\blt)^{-1}\cdot v\cdot\mc U(\eta_\blt)\cdot w,\label{eq29}
\end{align}
where the action of $v$ on $\mc U(\eta_\blt)w$ (which depends on $\eta_\blt$) is defined by \eqref{eq26} and \eqref{eq27}.

\begin{thm}[Cf. \cite{FB04} Thm. 6.5.4]\label{lb13}
The action of $H^0(C,\scr V_C\otimes\omega_C(\blt S_{\fk X}))$ on $\scr W_{\fk X}(\Wbb_\blt)$  defined by \eqref{eq29} is independent of the choice of $\eta_\blt$.
\end{thm}

\begin{proof}
We prove this theorem for the case $N=1$. The general cases can be proved in a similar way. Choose local coordinates $\eta,\mu$ at $x=x_1$ defined on a neighborhood $U$. We identify $U$ with $\mu(U)$ via $\mu$. So $\mu$ as a coordinate  is identified with the standard coordinate $\id_\Cbb$ of $\Cbb$, and $\eta\in\Gbb$. As a variable, $\mu$ is identified with the standard one $z$ of $\Cbb$. Also, identify $\scr W_{\fk X}(\Wbb)$ (where $\Wbb=\Wbb_1=\Wbb_\blt$) with $\Wbb$ via $\mc U(\mu)$. So $\mc U(\mu)=\mc U(\id_\Cbb)=\id$.  Choose any $w\in\Wbb$ , choose any section $v$ of $\scr V_C\otimes\omega_C(\blt S_{\fk X})$ defined on $U$. So
\begin{gather*}
\mc U_\varrho(\mu)v=u(z)dz
\end{gather*}
for some  $u=u(z)\in\Vbb\otimes_\Cbb\scr O_\Cbb(\blt 0)(U)$. Then, by \eqref{eq30},
	\begin{gather*}
	\mc U_\varrho(\eta)v=\mc U(\varrho(\eta|\id_\Cbb))u(z)dz=\mc U(\varrho(\eta|\id_\Cbb)_z)u(z)dz.
	\end{gather*}
Set variable $\zeta=\eta(z)$. Then
\begin{align*}
\mc U(\eta)^{-1}\cdot v\cdot\mc U(\eta)\cdot w=\Res_{\zeta=0}~\mc U(\eta)^{-1}Y_\Wbb\big(\mc U_\varrho(\eta)v,\zeta\big)\mc U(\eta)w
\end{align*}
where the $\zeta$ in $Y_\Wbb$ is due to the fact that $\eta$ (as a variable) equals $\zeta$ when $U$ is identified with $\eta(U)$ via $\eta$. (Such identification is needed in the definition of the coordinate-dependent action \eqref{eq26}.)  This expression equals
\begin{align*}
\Res_{z=0}~\mc U(\eta)^{-1}Y_\Wbb\big(\mc U(\varrho(\eta|\id_\Cbb)_z)u(z),\eta(z)\big)\mc U(\eta)w\cdot dz,
\end{align*}
which by Theorem \ref{lb14} (with $\alpha=\eta$) equals
	\begin{align*}
\Res_{z=0}~Y_\Wbb\big( u(z),z\big)w\cdot dz=\mc U(\mu)^{-1}\cdot v\cdot\mc U(\mu)\cdot w.
	\end{align*}
	The proof is complete.
\end{proof}

In the above proof, we have used the following theorem of Huang \cite{Hua97}; see also \cite[Lemma 6.5.6]{FB04}.

\begin{thm}\label{lb14}
	Let $\Wbb$ be a $\Vbb$-module. Let $U\subset\Cbb$ be a neighborhood of $0$. Let $\alpha\in\scr O(U)$ be a local coordinate at $0$ (so $\alpha(0)=0$). Let $\id_\Cbb\in\Gbb$ be the standard coordinate of $\Cbb$ (i.e.  the identity element of $\Gbb$). Then for any $v\in\Vbb$ and $w\in\Wbb$, we have the following equation of elements in $\Wbb((z))$:
	\begin{align}
	\mc U(\alpha)Y_\Wbb(v,z)\mc U(\alpha)^{-1}\cdot w=Y_{\Wbb}\big(\mc U(\varrho(\alpha|\id_\Cbb))v,\alpha(z)\big)\cdot w.\label{eq28}
	\end{align}
	Note that $U(\varrho(\alpha|\id_\Cbb))v$ is in $\Vbb\otimes_{\Cbb}\scr O(U)$ and hence can be regarded as an element of $\Vbb((z))$. Of course, \eqref{eq28} also holds in an obvious way for any $v\in\Vbb((z))$.
\end{thm}

For instance, take $\alpha(z)=\lambda z$ where $\lambda\in\Cbb^\times$. Then $\varrho(\alpha|\id_\Cbb)$ is constantly $\lambda$. It follows that
\begin{align}
\lambda^{\wtd L_0}Y_\Wbb(v,z)\lambda^{-\wtd L_0}=Y_\Wbb(\lambda^{L_0}v,\lambda z).\label{eq63}
\end{align}

We now define  \textbf{space of covacua} 
\begin{gather}
\scr T_{\fk X}(\Wbb_\blt)=\frac{\scr W_{\fk X}(\Wbb_\blt)}{H^0(C,\scr V_C\otimes\omega_C(\blt S_{\fk X}))\cdot \scr W_{\fk X}(\Wbb_\blt)}\label{eq31}
\end{gather}
(we have omitted $\Span_\Cbb$ in the denominator), whose dual vector space is denoted by $\scr T_{\fk X}^*(\Wbb_\blt)$ \index{T@$\scr T_{\fk X}^*(\Wbb_\blt),\scr T_{\fk X}^*(\Wbb_\blt)(\mc B)$} and called the \textbf{space of conformal blocks}.  Elements of $\scr T_{\fk X}^*(\Wbb_\blt)$ are called \textbf{conformal blocks} associated to $\Wbb_\blt$ and $\fk X$. They are the linear functionals of $\scr W_{\fk X}(\Wbb_\blt)$ vanishing on the denominator of \eqref{eq31}.

\section{Sewing families of compact Riemann surfaces}\label{lb6}

By  a (holomorphic) \textbf{family of compact Riemann surfaces} $\fk X=(\pi:\mc C\rightarrow\mc B)$ we mean $\mc B,\mc C$ are complex manifolds, $\mc B$ has finitely many connected components,\footnote{This is assumed only for simplicity.} $\pi$ is a proper surjective holomorphic submersion, and each fiber $\mc C_b:=\pi^{-1}(b)$ \index{Cb@$\mc C_b=\pi^{-1}(b)$} (where $b\in\mc B$) is a compact Riemann surface.

Note that by Ehresmann's theorem, when $\mc B$ is connected,  as a family of differential manifolds, $\fk X$  is trivial, i.e., $\fk X$ is equivalent to the projection $\mc C_b\times\mc B\rightarrow\mc B$ onto the $\mc B$-component. In particular, all fibers are diffeomorphic.

We say that
\begin{align*}
\fk X=(\pi:\mc C\rightarrow\mc B;\sgm_1,\dots,\sgm_N;\eta_1,\dots,\eta_N)
\end{align*}
is a \textbf{family of $N$-pointed compact Riemann surfaces with local coordinates}, if 
\begin{enumerate}[label=(\alph*)]
\item $\pi:\mc C\rightarrow\mc B$ is a family of compact Riemann surfaces.
\item Each $\sgm_i$ is a section of the family, namely, each $\sgm_i:\mc B\rightarrow\mc C$ is a holomorphic map satisfying $\pi\circ\sgm_i=\id_{\mc B}$.
\item Each $\eta_i$ is a \textbf{local coordinate at $\sgm_i(\mc B)$}, which means that there is an open subset $U_i\subset\mc C$ containing $\sgm_i(\mc B)$ such that $\eta_i\in\scr O(U_i)$ is univalent on each fiber $U_{i,b}=U_i\cap\mc C_b$ of $U_i$. In that case, $(\eta_i,\pi)$ is a biholomorphic map from $U_i$ to an open subset of $\Cbb\times\mc B$. Moreover, we assume the restriction of $\eta_i$ to $\sgm_i(\mc B)$ is $0$.
\item $\sgm_i(\mc B)\cap\sgm_j(\mc B)=\emptyset$ whenever $i\neq j$.
\item For each $b\in\mc B$, every connected component of the fiber $\mc C_b$ contains at lease one of the marked points $\sgm_1(b),\dots,\sgm_N(b)$.
\end{enumerate}
In the case that the local coordinates $\eta_1,\dots,\eta_N$ are not assigned (and hence condition (c) is not assumed), we say that $\fk X=(\pi:\mc C\rightarrow\mc B;\sgm_1,\dots,\sgm_N)$ is a family of $N$-pointed compact Riemann surfaces. \index{SX@$k\SX,\blt\SX$} We set
\begin{align*}
\SX=\sum_{i=1}^N\sgm_i(\mc B)
\end{align*}
to be a divisor of $\mc C$. Then for each $b\in\mc B$,
\begin{align*}
\SX(b):=\sum_{i=1}^N\sgm_i(b)
\end{align*}
is a divisor of $\mc C_b$. The definitions of $\SX$ and $\SX(b)$ will also apply to the case that $\fk X$ is formed by sewing a smooth family.

We say that $\fk X$ is a \textbf{family of $N$-pointed complex curves (resp. with local coordinates)}, if $\fk X$ is either a family of $N$-pointed compact Riemann surfaces (resp. with local coordinates) (in that case we say $\fk X$ is a smooth family), or if $\fk X$ is formed by sewing a smooth family $\wtd{\fk X}$, whose meaning is  explained below.

\subsection*{Sewing open discs}

We first describe how to sew a pair of open discs $\mc D_r,\mc D_\rho$. 

For any $r>0$, let  $\mc D_r=\{z\in\mbb C:|z|<r \}$ and $\mc D_r^\times=\mc D_r-\{0\}$.\index{Dr@$\mc D_r,\mc D_r^\times$} If $r,\rho>0$, we define \index{zz@$\pi_{r,\rho}:\mc D_r\times\mc D_\rho\rightarrow\mc D_{r\rho}$}
\begin{align}
\pi_{r,\rho}:\mc D_r\times\mc D_\rho\rightarrow\mc D_{r\rho},\qquad (\xi,\varpi)\mapsto \xi\varpi.
\end{align}
$d\pi_{r,\rho}$ is surjective at $(\xi,\varpi)$ whenever $\xi\neq 0$ or $\varpi\neq 0$. Denote also by $\xi$ and $\varpi$ the standard coordinates of $\mc D_r$ and $\mc D_\rho$, which can be extended constantly to $\xi:\mc D_r\times\mc D_\rho\rightarrow\mc D_r,\varpi:\mc D_r\times\mc D_\rho\rightarrow\mc D_\rho$. Set $q=\pi_{r,\rho}$, i.e.,
\begin{align*}
q:\mc D_r\times\mc D_\rho\rightarrow\Cbb,\qquad q=\xi\varpi.
\end{align*}
Then $(\xi,\varpi),(\xi,q),(\varpi,q)$ are coordinates  of $\mc D_r\times \mc D_\rho,\mc D_r^\times\times \mc D_\rho,\mc D_r\times \mc D_\rho^\times$ respectively. The standard tangent vectors of the coordinates $(\xi,\varpi),(\xi,q)$ are related by
\begin{gather}
\left\{ \begin{array}{l}
\partial_\xi=\partial_\xi-\xi^{-1}\varpi\cdot\partial_\varpi\\
\partial_q=\xi^{-1}\partial_\varpi
\end{array} \right.
\qquad
\left\{ \begin{array}{l}
\partial_\xi=\partial_\xi+\xi^{-1}q\cdot\partial_q\\
\partial_\varpi=\xi\partial_q
\end{array} \right.\label{eq23}
\end{gather}
The formulae between $(\xi,\varpi),(\varpi,q)$ are similar.

It is easy to see that  $(\xi,q)(\mc D_r^\times\times\mc D_\rho)$ (resp. $(\varpi,q)(\mc D_r\times\mc D_\rho^\times)$) is precisely the subset of all $(\xi_0,q_0)\in\mc D_r\times\mc D_{r\rho}$ (resp. $(\varpi_0,q_0)\in\mc D_\rho\times\mc D_{r\rho}$) satisfying
\begin{align}
\frac{|q_0|}{\rho}<|\xi_0|<r\qquad\text{resp.}\qquad  \frac{|q_0|}{r}<|\varpi_0|<\rho.
\end{align}
We choose closed subsets $E_{r,\rho}'\subset \mc D_r\times\mc D_{r\rho}$ and $E_{r,\rho}''\subset\mc D_\rho\times\mc D_{r\rho}$ such that
\begin{gather}
(\xi,q):\mc D_r^\times\times\mc D_\rho\xrightarrow{\simeq}\mc D_r\times\mc D_{r\rho}-E_{r,\rho}',\nonumber\\
(\varpi,q): \mc D_r\times\mc D_\rho^\times\xrightarrow{\simeq}\mc D_\rho\times\mc D_{r\rho}-E_{r,\rho}''\label{eq11}
\end{gather}
are bijective.

\subsection*{Sewing a family of compact Riemann surfaces}

We discuss how to simultaneously sew a family of compact Riemann surfaces. This construction is also known as smoothing in the world of algebraic geometry. See \cite[Sec. 6.1]{TUY89}, \cite[Sec. 5.3]{Ueno97}, or \cite[Sec. XI.3]{ACG11}. Its algebraic version is given in \cite[Sec. 6]{Loo10} and \cite[Sec. 8.1]{DGT19b}.

Consider a family of $(N+2)$-pointed compact Riemann surfaces with local coordinates
\begin{align}
\wtd{\fk X}=(\wtd\pi:\wtd{\mc C}\rightarrow\wtd{\mc B};\sgm_1,\dots,\sgm_N,\sgm',\sgm'';\eta_1,\dots\eta_N,\xi,\varpi).
\end{align}
\begin{ass}\label{lb62}
We assume that for every $b\in\wtd{\mc B}$, each connected component of  $\wtd{\mc C}_b=\wtd\pi^{-1}(b)$ contains one of $\sgm_1(b),\dots,\sgm_N(b)$.
\end{ass}
Choose $r,\rho>0$ and a neighborhood $U'$ (resp. $U''$) of $\sgm'(\wtd {\mc B})$ (resp. $\sgm''(\wtd {\mc B})$) on which $\xi$ (resp. $\varpi$) is defined, such that \index{U@$U',U''$}
\begin{gather}
(\xi,\wtd\pi):U'\xrightarrow{\simeq} \mc D_{r}\times\wtd{\mc B}\qquad\text{resp.}\qquad (\varpi,\wtd\pi):U''\xrightarrow{\simeq} \mc D_{\rho}\times\wtd{\mc B}\label{eq84}
\end{gather}
is a biholomorphic map. We also assume that $U'$ and $ U''$ are disjoint and do not intersect $\sgm_1(\wtd{\mc B}),\dots,\sgm_N(\wtd{\mc B})$. Identify
\begin{gather*}
U'=\mc D_{r}\times\wtd{\mc B}\qquad\text{resp.}\qquad U''=\mc D_{\rho}\times\wtd{\mc B}
\end{gather*}
via the above maps. Then $\xi,\varpi$ (when restricted to the first components) become the standard coordinates of $\mc D_{r},\mc D_{\rho}$ respectively, and $\wtd\pi$ is the projection onto the $\wtd{\mc B}$-component. Set $q=\xi\varpi=\pi_{r,\rho}:\mc D_{r}\times\mc D_{\rho}\rightarrow\mc D_{r\rho}$ as previously.  

Set
\begin{align}
\mc B=\mc D_{r\rho}\times\wtd{\mc B}.
\end{align}
We now \textbf{sew the smooth family $\wtd{\fk X}$} to obtain a family 
\begin{align*}
\fk X=(\pi:\mc C\rightarrow\mc B;\sgm_1,\dots,\sgm_N;\eta_1,\dots,\eta_N)
\end{align*}
of $N$-pointed  complex curves  with local coordinates. We first explain how to obtain $\mc C$ and $\pi:\mc C\rightarrow\mc B$. We shall freely switch the orders of Cartesian products. Note \index{F@$F',F''$} that
\begin{gather*}
F':=E_{r,\rho}'\times\wtd{\mc B}\qquad\subset \mc D_{r}\times\mc D_{r\rho}\times\wtd{\mc B}\quad(= U'\times \mc D_{r\rho}),\\
F'':=E_{r,\rho}''\times\wtd{\mc B}\qquad\subset \mc D_{\rho}\times\mc D_{r\rho}\times\wtd{\mc B}\quad(= U''\times \mc D_{r\rho})
\end{gather*}
are subsets of $\wtd{\mc C}\times\mc D_{r\rho}$. They are the subsets  we should discard in the sewing process.  Then $\mc C$ is obtained by gluing $\wtd{\mc C}\times \mc D_{r\rho}$ (with $F',F''$ all removed) \index{W@$W,W',W''$} with
\begin{align}
W:=\mc D_{r}\times\mc D_{\rho}\times\wtd{\mc B}.\label{eq10}
\end{align}
To be more precise, we define
\begin{align}
\mc C=W\bigsqcup(\wtd{\mc C}\times \mc D_{r\rho}-F'- F'')\Big/\sim \label{eq67}
\end{align}
where the equivalence $\sim$ is described as follows. Consider the following subsets of $W$:
\begin{gather}
W'=\mc D_{r}^\times\times\mc D_{\rho}\times\wtd{\mc B},\\
W''= \mc D_{r}\times\mc D_{\rho}^\times\times\wtd{\mc B}.
\end{gather}
Then the relation $\sim$ identifies $W'$ and $W''$  respectively via $(\xi,q,\id_{\wtd{\mc B}})$ and $(\varpi,q,\id_{\wtd{\mc B}})$  to
\begin{gather}
\mc D_{r}\times\mc D_{r\rho}\times\wtd{\mc B}-F'\qquad (\subset U'\times\mc D_{r\rho}),\label{eq12}\\
\mc D_{\rho}\times\mc D_{r\rho}\times\wtd{\mc B}-F''\qquad (\subset U''\times\mc D_{r\rho})\label{eq13}
\end{gather}
(recall \eqref{eq11}), which are subsets of $\wtd{\mc C}\times \mc D_{r\rho}-F'-F''$. (In particular, certain open subsets of \eqref{eq12} and \eqref{eq13} are glued together and identified with $W'\cap W''$.)

We now define $\pi$. It is easy to see that the projection
\begin{align}
\wtd\pi\times \id:\wtd{\mc C}\times \mc D_{r\rho}\rightarrow \wtd{\mc B}\times \mc D_{r\rho}=\mc B, \label{eq68}
\end{align}
agrees with
\begin{gather}
\pi_{r,\rho}\times \id:W=\mc D_{r}\times\mc D_{\rho}\times\wtd{\mc B}\rightarrow \mc D_{r\rho}\times\wtd{\mc B}=\mc B \label{eq69}
\end{gather}
when restricted to $W',W''$.  Thus, we have a well-defined surjective holomorphic map $\pi:\mc C\rightarrow\mc B$ whose  restrictions to $\wtd{\mc C}\times \mc D_{r\rho}-F'-F''$ and to  $W$ are $\wtd\pi\otimes\id$ and $\pi_{r,\rho}\otimes\id$ respectively. 

Finally, we extend each $\sgm_i:\wtd{\mc B}\rightarrow\wtd{\mc C}$ constantly to $\wtd{\mc B}\times\mc D_{r\rho}\rightarrow\wtd{\mc C}\times\mc D_{r\rho}$, whose image is disjoint from $F',F''$. Thus $\sgm_i$ can be regarded as a section $\sgm_i:\mc B\rightarrow\mc C$. Likewise, we extend $\eta_i$ constantly over $\wtd{\mc B}$ to a local coordinate at $\sgm_i(\mc B)$. We say that  the $N$-points $\sgm_1,\dots,\sgm_N$ and the local coordinates $\eta_1,\dots,\eta_N$ are \textbf{constant with respect to sewing}.

\subsection*{A short exact sequence}

The goal of this subsection is to recall a short exact sequence \eqref{eq18} (cf. \cite[Eq. (4.2.3)]{Ueno97} or \cite[Eq. (4.6)]{AU07}) which plays an important role in defining a logarithmic connection on sheaves of conformal blocks. In Section \ref{lb52}, we will use this exact sequence to find the differential equations that sewn conformal blocks satisfy.

Consider the discriminant locus $\Delta$ and the critical locus $\Sigma$: \index{zz@$\Delta$} \index{zz@$\Sigma$}
\begin{gather*}
\Delta=\{0\}\times\wtd{\mc B}\qquad (\subset\mc D_{r\rho}\times\wtd{\mc B}=\mc B),\\
\Sigma=\{0\}\times\{0\}\times\wtd{\mc B}\qquad (\subset\mc D_r\times\mc D_\rho\times\wtd{\mc B}=W).
\end{gather*}
Then $\Delta$ is the set of all $b\in\mc B$ such that the fiber $\mc C_b=\pi^{-1}(b)$ is a nodal curve with one node. Outside $\Delta$, the fibers are compact Riemann surfaces. $\Sigma$ is the set of all points of $\mc C$ at which $d\pi$ is not surjective, and is also the set of all nodes of the fibers. Note that
\begin{align*}
\Sigma=W-(W'\cup W'').
\end{align*}
Also, $\Sigma$ is  the set of points of $\mc C$ not coming from $\wtd{\mc C}\times \mc D_{r\rho}-F'- F''$. The union of nodal fibers is \index{C@$\mc C_\Delta$} therefore
\begin{align*}
\mc C_\Delta=\pi^{-1}(\Delta).
\end{align*}

Let $\Theta_{\mc B}(-\log\Delta)$ and $\Theta_{\mc C}(-\log\mc C_\Delta)$ \index{zz@$\Theta_{\mc B}(-\log\Delta)$} \index{zz@$\Theta_{\mc C}(-\log\mc C_\Delta)$} be  the sheaves of sections of $\Theta_{\mc B}$ and $\Theta_{\mc C}$ tangent to $\Delta$ and $\mc C_\Delta$ respectively. Then the differential $d\pi:\Theta_{\mc C}\rightarrow\pi^*\Theta_{\mc B}$ of the map $\pi$ restricts to an $\scr O_{\mc C}$-module homomorphism
\begin{align}
d\pi:\Theta_{\mc C}(-\log\mc C_\Delta)\rightarrow \pi^*\Theta_{\mc B}(-\log\Delta)\label{eq15}
\end{align}
(the later is short for $\pi^*(\Theta_{\mc B}(-\log\Delta))$), which is indeed surjective. To understand the meaning and to see the claimed facts, let us describe the two sheaves and the morphism $d\pi$ using coordinates.

We assume $\wtd{\mc B}$ is small enough  to admit a  coordinate \index{zz@$\tau_\blt$} $\tau_\blt=(\tau_1,\dots,\tau_m):\wtd{\mc B}\rightarrow\Cbb^m$.  Then $(q,\tau_\blt)$ is a coordinate of $\mc B$ if we let $q$ be the standard coordinate of $\mc D_{r\rho}$. Then  $\Theta_{\mc B}(-\log\Delta)$ is an $\scr O_{\mc B}$-module generated freely by
\begin{align*}
q\partial_q,\partial_{\tau_1},\dots,\partial_{\tau_m}.
\end{align*}
Their pullback under $\pi$ will also be denoted by the same symbols $q\partial_q,\partial_{\tau_1},\dots,\partial_{\tau_m}$, for simplicity. Choose any $x\in\mc C$. 

Case I. $x\notin\Sigma$. Then $x$ can be regarded as a point  $(\wtd x,q_0)$ of $\wtd{\mc C}\times \mc D_{r\rho}$ disjoint from $F',F''$. Choose a neighborhood $\wtd U\subset\wtd{\mc C}$ of $\wtd x$ together with $\eta\in\scr O(\wtd U)$ univalent on each fiber of $\wtd U$. Choose a neighborhood $V$ of $q_0\in\mc D_{r\rho}$ such that $U:=\wtd U\times V$ is disjoint from $F',F''$. Write $\tau_\blt\circ\wtd\pi$ also as $\tau_\blt$ for short. Then $(\eta,\tau_\blt,q)$ is a coordinate of $U\ni x$. Note that $U\cap\mc C_\Delta$ is described by $q=0$.  The  $\scr O_V$-module $\Theta_{\mc C}(-\log\Delta)|_V$ is generated freely by
\begin{align}
\partial_\eta,q\partial_q,\partial_{\tau_1},\dots,\partial_{\tau_m},\label{eq14}
\end{align} 
and the morphism $d\pi$ in \eqref{eq15}  sends $\partial_\eta$ to $0$ and keeps the other elements of \eqref{eq14}.

Case II. $x\in\Sigma$.  Then $W=\mc D_r\times\mc D_\rho\times\wtd{\mc B}$ is a neighborhood of $x$, and has coordinate $(\xi,\varpi,\tau_\blt)$. Note that $W\cap\mc C_\Delta$ is described by $\xi\varpi=0$. The  $\scr O_W$-module $\Theta_{\mc C}(-\log\Delta)|_W$ is generated freely by
\begin{align}
\xi\partial_\xi,\varpi\partial_\varpi,\partial_{\tau_1},\dots,\partial_{\tau_m},\label{eq16}
\end{align} 
and the morphism $d\pi$ satisfies
\begin{align}
d\pi(\xi\partial_\xi)=d\pi(\varpi\partial_\varpi)=q\partial_q\label{eq81}
\end{align}
(note that $q\partial_q$ is short for $\pi^*(q\partial_q)$) and keeps the other elements of \eqref{eq16}.

It is clear that in both cases, $d\pi$ (in \eqref{eq15}) is surjective. Thus, by letting $\Theta_{\mc C/\mc B}$ \index{zz@$\Theta_{\mc C/\mc B}$} be the kernel of $d\pi$, we obtain an exact sequence of  $\scr O_{\mc C}$-modules
\begin{align}
\boxed{~~0\rightarrow \Theta_{\mc C/\mc B}\rightarrow \Theta_{\mc C}(-\log \mc C_\Delta)\xrightarrow{d\pi}\pi^*\Theta_{\mc B}(-\log \Delta)\rightarrow 0~~}\label{eq18}
\end{align}
In case I resp. case II, $\Theta_{\mc C/\mc B}|_V$ resp. $\Theta_{\mc C/\mc B}|_W$  is generated (freely) by
\begin{align}
\partial_\eta,\qquad\text{resp.}\qquad \xi\partial_\xi-\varpi\partial_\varpi.\label{eq17}
\end{align}
So $\Theta_{\mc C/\mc B}$ is locally free of rank $1$, whose dual module is the relative dualizing sheaf \index{zz@$\omega_{\mc C/\mc B}$} $\omega_{\mc C/\mc B}$. Using \eqref{eq23}, we see that when restricted to $W'$ (resp. $W''$) and under the coordinate $(\xi,q,\tau_\blt)$ (resp. $(\varpi,q,\tau_\blt)$), the section $\xi\partial_\xi-\varpi\partial_\varpi$ in \eqref{eq17} equals
\begin{align}
\xi\partial_\xi\qquad\text{resp.}\qquad -\varpi\partial_\varpi.\label{eq24}
\end{align}
Compare this with \eqref{eq9}, we see that for each $b\in\mc B$, there are natural equivalences
\begin{align*}
\Theta_{\mc C/\mc B}|\mc C_b\simeq\Theta_{\mc C_b},\qquad \omega_{\mc C/\mc B}|\mc C_b\simeq\omega_{\mc C_b}.
\end{align*}

Finally, we remark that when $\fk X$ is a smooth family, $d\pi:\Theta_{\mc C}\rightarrow\pi^*\Theta_{\mc B}$ is surjective. Thus, \eqref{eq18} becomes
\begin{align*}
0\rightarrow \Theta_{\mc C/\mc B}\rightarrow \Theta_{\mc C}\xrightarrow{d\pi}\pi^*\Theta_{\mc B}\rightarrow 0
\end{align*}
where $\Theta_{\mc C/\mc B}$ is the kernel of $d\pi$.

\subsection*{A theorem of Grauert}

Recall that if $C$ is a complex curve and $\scr E$ is a coherent $\scr O_C$-module, then  $H^p(C,\scr E)$ is always finite-dimensional by (for instance) the direct image theorem of Grauert. Also $H^p(C,\scr E)$ vanishes when $p>1$, which follows for instance from the fact that $C$ can be covered by two Stein open subsets. Thus, the character of $\scr E$ is $\chi(C,\scr E)=\dim H^0(C,\scr E)-\dim H^1(C,\scr E)$.

Given a family $\fk X=(\pi:\mc C\rightarrow\mc B)$ of complex curves, recall that we have assumed $\pi$ is proper when $\fk X$ is a smooth family. If $\fk X$ is formed by sewing $\wtd{\fk X}$, then properness still holds and is not hard to check. Finally, in both smooth and singular cases, the map $\pi$ is clearly open. So $\pi$ is flat by  \cite[Sec. 3.20]{Fis76} (see also \cite[Thm. II.2.13]{GPR94} or \cite[Thm. V.2.13]{BS76}). If $\scr E$ is locally free (of finite rank), then $\scr E$ is flat over $\mc B$. Therefore, by a theorem of Grauert \cite{Gra60} (cf. \cite[Thm. III.4.7]{GPR94} or \cite[Thm. III.4.12]{BS76} or \cite[Thm. 9.4.8]{EP96}), we have

\begin{thm}\label{lb11}
Let $\fk X=(\pi:\mc C\rightarrow\mc B)$ be a family of complex curves. Let $\scr E$ be a locally free $\scr O_{\mc C}$-module.

(a) The function
\begin{gather*}
\mc B\rightarrow \mathbb Z, \qquad b\mapsto\chi(\mc C_b,\scr E|\mc C_b)=\dim H^0(\mc C_b,\scr E|\mc C_b)-\dim H^1(\mc C_b,\scr E|\mc C_b)
\end{gather*}
is locally constant. 

(b) For any $p\in\mathbb N$, if the function  $b\mapsto\dim H^p(\mc C_b,\scr E|\mc C_b)$ is  locally constant, then the $\scr O_{\mc B}$-module $R^p\pi_*\scr E$ is  locally free of rank $\dim H^p(\mc C_b,\scr E|\mc C_b)$, and for any $b\in\mc B$, the linear map $(R^p\pi_*\scr E)_b\rightarrow H^p(\mc C_b,\scr E|\mc C_b)$ defined by  restricting the sections $s\mapsto s|{\mc C_b}$ induces an isomorphism of vector spaces
\begin{align*}
\frac{(R^p\pi_*\scr E)_b}{\fk m_b\cdot (R^p\pi_*\scr E)_b}\simeq H^p(\mc C_b,\scr E|\mc C_b).
\end{align*}
\end{thm}

Recall  $R^0\pi_*=\pi_*$, and the (higher) direct image sheaf $R^p\pi_*\scr E$ is an $\scr O_{\mc B}$-module associated to the presheaf $V\mapsto H^p(V,\scr E|_V)$ (for all open $V\subset\mc B$). As a consequence of this theorem, we see $R^p\pi_*\scr E=0$ when $\scr E$ is locally free and $p>1$.

\section{Sheaves of VOAs on families of complex curves}\label{lb23}

Let $\fk X=(\pi:\mc C\rightarrow\mc B)$ be a family of complex curves. Recall that $\Sigma$ is the critical locus, which is empty when the family is smooth. Let $U,V$ be open subsets of $\mc C-\Sigma$, and let $\eta\in\scr O(U),\mu\in\scr O(V)$ be univalent on each fiber of $U$ and $V$ respectively. Then  $(\eta,\pi)$ and $(\mu,\pi)$ are biholomorphic maps from $U$ resp. $V$ to open subsets of $\Cbb\times\mc B$.  \index{zz@$\varrho(\eta\lvert\mu)$} For each $p\in U\cap V$, we define $\varrho(\eta|\mu)_p\in\scr O_{\Cbb,0}$ by
\begin{align}
\varrho(\eta|\mu)_p(z)=\eta\circ(\mu,\pi)^{-1}\big(z+\mu(p),\pi(p)\big)-\eta(p).
\end{align}
Then $\varrho(\eta|\mu)_p$ is a holomorphic function of $z$ on  $\mu\big((U\cap V)_{\pi(p)}\big)$ where $(U\cap V)_{\pi(p)}$ is the fiber $U\cap V\cap \pi^{-1}(\pi(p))$. It is easy to check that for each $n\in\Nbb$,
\begin{align}
\partial_z^n\varrho(\eta|\mu)_p(0)=\partial_\mu^n \eta(p),
\end{align}
where the partial derivative $\partial_\mu$ is defined to be vertical to $d\pi$. From this, we see that $\varrho(\eta|\mu)_p(0)=0$ and  $\partial_z\varrho(\eta|\mu)_p(0)\neq0$. So $\varrho(\eta|\mu)_p$ is an element of $\Gbb$. We thus obtain a family of transformations $\varrho(\eta|\mu):U\cap V\rightarrow\Gbb,p\mapsto \varrho(\eta|\mu)_p$, which  is clearly holomorphic. According to Remark \ref{lb2},  we have an $\scr O_{U\cap V}$-module isomorphism
\begin{align*}
\mc U(\varrho(\eta|\mu)):\Vbb^{\leq n}\otimes_\Cbb\scr O_{U\cap V}\xrightarrow{\simeq} \Vbb^{\leq n}\otimes_\Cbb\scr O_{U\cap V}.
\end{align*}

As in Section \ref{lb5}, $\varrho(\eta|\mu)$ is also described by
\begin{align}
\eta-\eta(p)\big|_{(U\cap V)_{\pi(p)}}=\varrho(\eta|\mu)_p\big(\mu-\mu(p)\big|_{(U\cap V)_{\pi(p)}}\big).\label{eq20}
\end{align}
To see this, one composes both sides of \eqref{eq20} with $(\mu,\pi)^{-1}\big(z+\mu(p),\pi(p)\big)$.  Thus, we can get $\varrho(\eta|\mu)_p$ by restricting $\eta,\mu$ to the fiber $(U\cap V)_{\pi(p)}$ and then using Definition \eqref{eq3}. Therefore, the cocycle relation \eqref{eq19} still holds for holomorphic functions $\eta_1,\eta_2,\eta_3$ univalent on each fiber.

\subsection*{Definition of $\scr V_{\fk X}$}

Sheaves of VOAs were introduced in \cite{FB04} for algebraic families of smooth curves, and were generalized in \cite{DGT19a} to algebraic families of stable curves. Similar to their construction, we now define the sheaf of VOA $\scr V_{\fk X}$ associated to a VOA $\Vbb$ and the analytic family $\fk X$. 

We  \index{VX@$\scr V_{\fk X},\scr V_{\fk X}^{\leq n}$} set
\begin{align*}
\scr V_{\fk X}=\varinjlim_{n\in\Nbb}\scr V_{\fk X}^{\leq n}
\end{align*}
where each $\scr V_{\fk X}^{\leq n}$ is a locally free $\scr O_{\mc C}$-module of rank $\dim\Vbb^{\leq n}$ defined as follows. Outside $\Sigma$, $\scr V_{\fk X}^{\leq n}$ is a vector bundle with transition function $\mc U(\varrho(\eta|\mu))$. Thus, for each open $U\subset\mc C-\Sigma$ and $\eta\in\scr O(U)$ univalent on each fiber, we have an isomorphism of $\scr O_U$-modules \index{U@$\mc U_\varrho(\eta)$} (a trivilization) 
\begin{align}
\mc U_\varrho(\eta):\scr V_{\fk X}^{\leq n}|_U\xrightarrow{\simeq} \Vbb^{\leq n}\otimes_{\Cbb}\scr O_U.\label{eq35}
\end{align}
If $V$ is another open subset of $\mc C-\Sigma$ and $\mu\in\scr O(V)$ is also univalent on each fiber, then on $U\cap V$ we have
\begin{align}
\mc U_\varrho(\eta)\mc U_\varrho(\mu)^{-1}=\mc U(\varrho(\eta|\mu)).
\end{align}
Since \eqref{eq8} holds when restricting to each fiber, we again have that for any section $v$ of $\Vbb^{\leq n}\otimes_{\Cbb}\scr O_{U\cap V}$,
\begin{align}
\mc U_\varrho(\eta)\mc U_\varrho(\mu)^{-1}\cdot v=(\partial_\mu\eta)^n\cdot v~~\mod ~~\Vbb^{\leq n-1}\otimes_{\Cbb}\scr O_{U\cap V}.\label{eq25}
\end{align}

We now assume $\fk X$ is formed by sewing a smooth family $\wtd{\fk X}$ as in section \ref{lb6}. Let $W,W',W'',\xi,\varpi,q$ be as in that section. (See the discussion near \eqref{eq10}.) Then $(\xi,q,\id_{\wtd {\mc B}})$ and $(\varpi,q,\id_{\wtd{\mc B}})$ are respectively biholomorphic maps from $W'$ and $W''$ to complex manifolds, and the projection $\pi$ equals $(q,\id_{\wtd{\mc B}})$ when restricted to $W'$ or $W''$. Thus, $\xi,\varpi$ are univalent on fibers of $W',W''$ respectively.

We shall define $\scr V^{\leq n}_{\fk X}|_{W}$ to be an $\scr O_{W}$-submodule of $\scr V_{\fk X}^{\leq n}|_{W-\Sigma}$ generated (freely) by some sections on $W$ whose restrictions to $W'$ and $W''$ are described under the trivilizations $\mc U_\varrho(\xi)$ and $\mc U_\varrho(\varpi)$ respectively. For that purpose, we need to first calculate the transition function \begin{align*}
\mc U(\varrho(\varpi|\xi)):\Vbb^{\leq n}\otimes_{\Cbb}\scr O_{W'\cap W''}\xrightarrow{\simeq} \Vbb^{\leq n}\otimes_{\Cbb}\scr O_{W'\cap W''}.
\end{align*}

\begin{lm}\label{lb7}
	Choose any $p\in W'\cap W''$. Then we have
	\begin{gather*}
	\varrho(\varpi|\xi)_p(z)=q(p)\upgamma_{\xi(p)}(z)
	\end{gather*}
	and hence
	\begin{align*}
	\mc U(\varrho(\varpi|\xi)_p)=q(p)^{L_0}\mc U(\upgamma_{\xi(p)}).
	\end{align*}
\end{lm}
\begin{proof}
	Choose any $x\in (W'\cap W'')_{\pi(p)}$. Then $\pi(x)=\pi(p)$ and hence $q(x)=q(p)$.  Since $\varpi=\xi^{-1}q$, we have
	\begin{align*}
	\varpi(x)-\varpi(p)=q(p)(\xi(x)^{-1}-\xi(p)^{-1}).
	\end{align*}
	By \eqref{eq20}, we have
	\begin{align*}
	\varpi(x)-\varpi(p)=\varrho(\varpi|\xi)_p(\xi(x)-\xi(p)).
	\end{align*}
	If we compare these two equations and set $z=\xi(x)-\xi(p)$, we obtain
	\begin{align*}
	&\varrho(\varpi|\xi)_p(z)=\varrho(\varpi|\xi)_p(\xi(x)-\xi(p))=q(p)(\xi(x)^{-1}-\xi(p)^{-1})\\
	=&q(p)\big((\xi(p)+z)^{-1}-\xi(p)^{-1} \big)=q(p)\upgamma_{\xi(p)}(z).
	\end{align*}
\end{proof}

We   define $\scr V_{\fk X}^{\leq n}|_{W}$ to be the  $\scr O_{W}$-submodule of $\scr V_{\fk X-\Sigma}^{\leq n}|_{W-\Sigma}$ generated by any section on $W-\Sigma$ whose restrictions to $W'$ and $W''$ are
\begin{align}
\boxed{~~\mc U_\varrho(\xi)^{-1}\big(\xi^{L_0}v\big)\qquad\text{resp.}\qquad \mc U_\varrho(\varpi)^{-1}\big(\varpi^{L_0}\mc U(\upgamma_1)v\big)  ~~}\label{eq21}
\end{align}
where $v\in\Vbb^{\leq n}$. Since $\upgamma_1=\upgamma_1^{-1}$ and hence $\mc U(\upgamma_1)=\mc U(\upgamma_1)^{-1}$, this definition is symmetric with respect to $\xi$ and $\varpi$. To check that \eqref{eq21} is well-defined, we need:

\begin{lm}
	The two sections defined in \eqref{eq21} agree on $W'\cap W''$.
\end{lm}

\begin{proof}
Using   \eqref{eq22} and Lemma \ref{lb7}, we check that
	\begin{align*}
	&\mc U_\varrho(\varpi)\mc U_\varrho(\xi)^{-1}\xi^{L_0}v=\mc U(\varrho(\varpi|\xi))\xi^{L_0}v=q^{L_0}\mc U(\upgamma_{\xi})\xi^{L_0}v\\
	=&q^{L_0}\xi^{-L_0}\mc U(\upgamma_1)v=\varpi^{L_0}\mc U(\upgamma_1)v.
	\end{align*}
\end{proof}

It is easy to see that, if we take $v\in E$ where $E$ is a basis of $\Vbb^{\leq n}$ consisting of homogeneous vectors, then $\scr V_{\fk X}^{\leq n}|_{W}$ is generated freely by sections defined by \eqref{eq21} for all $v\in E$. We have completed the definition of the  locally free $\scr O_{\mc C}$-module $\scr V_{\fk X}^{\leq n}$.

\begin{rem}
Since the vacuum vector $\id$ is annihilated by $L_n$ ($n\geq 0$), we see that $\id$ is fixed by any transition function $\mc U(\varrho(\eta|\mu))$. Thus, we can define unambiguously an element $\id\in\scr V_{\fk X}(\mc C-\Sigma)$ (the \textbf{vacuum section}) such that for any open $U\subset\mc C-\Sigma$ and any $\eta\in\scr O(U)$ univalent on each fiber, $\mc U_\varrho(\eta)\id$ is the vaccum vector $\id$ (considered as a constant function). \index{1@$\id$} Also, by \eqref{eq21}, it is clear that
\begin{align*}
	\id\in\scr V_{\fk X}(\mc C).
\end{align*}
\end{rem}

\subsection*{Restriction to fibers}

By comparing the transition functions and looking at the generating sections near the nodes, it is easy to see:

\begin{pp}\label{lb10}
For any $n\in\Nbb$ and $b\in\mc B$, we have a natural isomorphism of $\scr O_{\mc C_b}$-modules
\begin{align*}
\scr V_{\fk X}^{\leq n}|\mc C_b\simeq\scr V_{\mc C_b}^{\leq n}.
\end{align*}
\end{pp}

Let now $\fk X=(\pi:\mc C\rightarrow\mc B;\sgm_1,\dots,\sgm_N)$ be a family of $N$-pointed complex curves.

\begin{thm}\label{lb8}
Let $n\in\Nbb$. Then  there exists $k_0\in\mbb N$ such that for any $k>k_0$, the $\scr O_{\mc B}$-module  $\pi_*\big(\scr V_{\fk X}^{\leq n}\otimes\omega_{\mc C/\mc B}(k\SX)\big)$ is locally free, and for any $b\in\mc B$ there is a natural isomorphism of vector spaces
\begin{align}
	\frac{\pi_*\big(\scr V_{\fk X}^{\leq n}\otimes\omega_{\mc C/\mc B}(k\SX)\big)_b}{~\fk m_b\cdot\pi_*\big(\scr V_{\fk X}^{\leq n}\otimes\omega_{\mc C/\mc B}(k\SX)\big)_b~}\simeq H^0\big(\mc C_b,\scr V_{\mc C_b}^{\leq n}\otimes\omega_{\mc C_b}(k\SX(b))\big)\label{eq36}
\end{align} 
defined by  restriction of sections. In particular,  $\dim  H^0\big(\mc C_b,\scr V_{\mc C_b}^{\leq n}\otimes\omega_{\mc C_b}(k\SX(b))\big)$ is locally constant over $b$.
\end{thm}

Recall that the left hand side of \eqref{eq36} is the fiber $\pi_*\big(\scr V_{\fk X}^{\leq n}\otimes\omega_{\mc C/\mc B}(k\SX)\big)\big|b$, which, by Cartan's Theorem A, is formed by the restrictions of global sections if $\mc B$ is Stein.

\begin{proof}
By Theorem \ref{lb9}, we can find $k_0\in\Nbb$ such that for any $k>k_0$, $H^1\big(\mc C_b,\scr V_{\mc C_b}^{\leq n}\otimes\omega_{\mc C_b}(k\SX(b))\big)$ vanishes for any $b\in\mc B$. Since the restriction of $\omega_{\mc C/\mc B}$ to $\mc C_b$ is $\omega_{\mc C_b}$, by Proposition \ref{lb10}, the restriction of $\scr V_{\fk X}^{\leq n}\otimes\omega_{\mc C/\mc B}(k\SX)$ to $\mc C_b$ is naturally equivalent to $\scr V_{\mc C_b}^{\leq n}\otimes\omega_{\mc C_b}(k\SX(b))$.  Thus, our theorem follows easily from Grauert's Theorem \ref{lb11}.
\end{proof}

For any $\scr O_{\mc C}$-module $\scr E$, we \index{SX@$\blt\SX$} set
\begin{align*}
\scr E(\blt\SX)=\varinjlim_{k\in\Nbb}\scr E(k\SX),
\end{align*}
whose sections are meromorphic sections of $\scr E$ with the only possible poles being $\sgm_1(\mc B),\dots,\sgm_N(\mc B)$.

\begin{co}\label{lb16}
If $\mc B$ is a Stein manifold, then for any $n\in\Nbb$ and any $b\in \mc B$, the restrictions to $\mc C_b$ of the elements of $H^0\big(\mc C,\scr V_{\fk X}^{\leq n}\otimes\omega_{\mc C/\mc B}(\blt\SX)\big)$  form the vector space $H^0\big(\mc C_b,\scr V_{\mc C_b}^{\leq n}\otimes\omega_{\mc C_b}(\blt\SX(b))\big)$.
\end{co}

\begin{proof}
Choose $k_0$ as in Theorem \ref{lb8}. Then by that theorem  and Cartan's Theorem A, we see the claim holds if $\blt\SX$ is replaced by $k\SX$ for all $k>k_0$. The original claim thus follows.
\end{proof}

\subsection*{The subsheaf $\svir_c$}

If we compare \eqref{eq25} with the transition functions of $\Theta_{\mc C/\mc B}$, and  compare \eqref{eq21} with \eqref{eq24}, we immediately see the following result:

\begin{pp}\label{lb12}
For any $n\in\Nbb$, we have an isomorphism of  $\scr O_C$-modules
	\begin{align}
	\scr V_{\fk X}^{\leq n}/\scr V_{\fk X}^{\leq n-1}\simeq\Vbb(n)\otimes_{\Cbb}\Theta_{\mc C/\mc B}^{\otimes n}.
	\end{align}
such that under this isomorphism, if $U\subset \mc C-\Sigma$ is open and smooth, and $\eta\in\scr O(U)$ is univalent on each fiber of $U$, then for any $v\in\Vbb(n)$, $v\otimes \partial_\eta^n$ is identified with the equivalence class of $\mc U_\varrho(\eta)^{-1}v$.
\end{pp}

We now assume for simplicity that $\fk X$ is a smooth family, and define an important $\scr O_{\mc C}$-submodule $\svir_c$ \index{Virc@$\svir_c$} of $\scr V_{\fk X}^{\leq 2}$ related to the conformal vector $\cbf\in\Vbb(2)$. See \cite[Sec. 8.2]{FB04}. If $U$ is an open subset of $\mc C$ equipped with  $\eta\in\scr O(U)$ univalent on each fiber, then $\svir_c|_U$ is the  $\scr O_U$-submodule of $\scr V_{\fk X}|_U$ generated (freely) by $\mc U_\varrho(\eta)^{-1}\cbf$ and the vacuum section $\id$, which is locally free of rank $2$. This definition is independent of the choice of $\eta$. Indeed, if $\mu:U\rightarrow\Cbb$ is also univalent on each fiber, then $\mc U_\varrho(\mu)\mc U_\varrho(\eta)^{-1}\cbf=\mc U(\varrho(\mu|\eta))\cbf$, which can be calculated using the actions of $L_n$ ($n\geq 0$) on $\cbf$, is an $\scr O_U$-linear combination of $\cbf$ and $\id$. Thus, by gluing all such $U$, we get $\svir_c$.

By Proposition \ref{lb12}, we have a short exact sequence
\begin{align*}
0\rightarrow\scr V_{\fk X}^{\leq 1} \rightarrow\scr V_{\fk X}^{\leq 2}\xrightarrow{\uplambda}\Vbb(2)\otimes_\Cbb\Theta_{\mc C/\mc B}^{\otimes 2}\rightarrow 0
\end{align*}
where $\uplambda$ is described locally  by sending  $\mc U_\varrho^{-1}(\eta)v$ (where $v\in\Vbb(2)$) to  $v\cdot \partial_\eta^2$ and sending the submodule $\scr V_{\fk X}^{\leq 1}$ to $0$. Using this description of $\uplambda$, it is easy to see that the restriction of $\uplambda$ to the subsheaf $\svir_c$ has image $\cbf\otimes_\Cbb \Theta_{\mc C/\mc B}^{\otimes 2}\simeq \Theta_{\mc C/\mc B}^{\otimes 2}$, and that its kernel is the $\scr O_{\mc C}$-submodule generated by $\id$. Thus, we obtain an exact sequence
\begin{align}
0\rightarrow\scr O_{\mc C} \rightarrow\svir_c\xrightarrow{\uplambda}\Theta_{\mc C/\mc B}^{\otimes 2}\rightarrow 0.
\end{align}
where $\scr O_{\mc C}\simeq\scr O_{\mc C}\cdot\id\subset\svir_c$. If we choose $U\subset\mc C$ and $\eta\in\scr O(U)$ holomorphic on each fiber,  then
\begin{gather*}
\uplambda:\quad \mc U_\varrho(\eta)^{-1}\cbf\mapsto \partial_\eta^2,\qquad \id\mapsto 0.
\end{gather*}
By tensoring with $\omega_{\mc C/\mc B}$, we get an exact sequence
\begin{gather}
0\rightarrow\omega_{\mc C/\mc B} \rightarrow\svir_c\otimes \omega_{\mc C/\mc B}\xrightarrow{\uplambda} \Theta_{\mc C/\mc B}\rightarrow 0\label{eq53}
\end{gather}
whose local expression is
\begin{gather}
\uplambda:\quad \mc U_\varrho(\eta)^{-1}\cbf ~d\eta\mapsto \partial_\eta,\qquad \id ~d\eta\mapsto 0.
\end{gather}

\section{Conformal blocks associated to families of complex curves}\label{lb28}

Let $\fk X=(\pi:\mc C\rightarrow\mc B;\sgm_1,\dots,\sgm_N)$ be a family of $N$-pointed complex curves. For each $b\in\mc B$, \index{Xb@$\fk X_b$} 
\begin{align*}
\fk X_b=(\mc C_b;\sgm_1(b),\dots,\sgm_N(b))
\end{align*}
is an $N$-pointed complex curve. If $\fk X$ is equipped with local coordinates $\eta_1,\dots,\eta_N$ at $\sgm_1(\mc B),\dots,\sgm_N(\mc B)$, then, by restricting to $\mc C_b$, $\fk X_b$ also has local coordinates.

Let $\Wbb_1,\dots,\Wbb_N$ be $V$-modules. Then for each $b\in\mc B$, $\scr T_{\fk X_b}^*(\Wbb_\blt)$ is the space of conformal blocks associated to $\Wbb_\blt$ and $\fk X_b$. In this section, we shall define conformal blocks $\upphi$ associated to the family $\fk X$. $\upphi$ is a function on $\mc B$ whose value $\upphi(b)$ at each $b\in\mc B$ is an element of $\scr T_{\fk X_b}^*(\Wbb_\blt)$. Moreover, we require that $\upphi$ is a holomorphic function in a certain sense. Since this property is local, we  assume  that $\mc B$ is small enough such that $\fk X$ admits local coordinates. This assumption will be dropped in the later half of this section.

\subsection*{Definition of conformal blocks}

We need to explain how the vector spaces $\scr W_{\fk X_b}(\Wbb_\blt)$ (for all $b\in\mc B$) form an (infinite rank) holomorphic vector bundle $\scr W_{\fk X}(\Wbb_\blt)$.

Define $\scr W_{\fk X}(\Wbb_\blt)$ \index{WX@$\scr W_{\fk X}(\Wbb_\blt)$} to be an infinite rank locally free $\scr O_{\mc B}$-module as follows. For any  local coordinates $\eta_1,\dots,\eta_N$ of $\fk X$ at $\sgm_1(\mc B),\dots,\sgm_N(\mc B)$ respectively, we have a trivialization \index{U@$\mc U(\eta_\blt)$}
\begin{align}
\mc U(\eta_\blt)\equiv \mc U(\eta_1)\otimes\cdots\otimes\mc U(\eta_N):\scr W_{\fk X}(\Wbb_\blt)\xrightarrow{\simeq}\Wbb_\blt\otimes_\Cbb\scr O_{\mc B}\label{eq37}
\end{align}
such that if $\mu_\blt$ is another set of local coordinates, then the transition function
\begin{align*}
\mc U(\eta_\blt)\mc U(\mu_\blt)^{-1}:\Wbb_\blt\otimes_\Cbb\scr O_{\mc B}\xrightarrow{\simeq}\Wbb_\blt\otimes_\Cbb\scr O_{\mc B}
\end{align*}
is defined such that for any constant section $w_\blt=w_1\otimes\cdots\otimes w_N\in\Wbb_\blt$, $\mc U(\eta_\blt)\mc U(\mu_\blt)^{-1}w_\blt$, as a $\Wbb_\blt$-valued holomorphic function, satisfies
\begin{align}
&\Big(\mc U\big(\eta_\blt\big)\mc U\big(\mu_\blt\big)^{-1}w_\blt\Big )(b)\equiv \Big(\mc U\big(\eta_\blt\big|\mu_\blt\big)^{-1}w_\blt\Big )(b)\nonumber\\
=&\mc U\big((\eta_1|\mu_1)_b \big)w_1\otimes\mc U\big((\eta_2|\mu_2)_b \big)w_2\otimes\cdots\otimes \mc U\big((\eta_N|\mu_N)_b \big)w_N \label{eq33}
\end{align}
for any $b\in \mc B$. Here, for each $1\leq i\leq N$, $(\eta_i|\mu_i)_b$ is the element in $\Gbb$ satisfying
\begin{align}
(\eta_i|\mu_i)_b(z)=\eta_i\circ(\mu_i,\pi)^{-1}(z,b).
\end{align}
If we compare the transition functions \eqref{eq32} and \eqref{eq33}, we see that there is a natural and coordinate-independent isomorphism of vector spaces
\begin{align*}
\scr W_{\fk X}(\Wbb_\blt)|b\simeq \scr W_{\fk X_b}(\Wbb_\blt)
\end{align*}
where $\scr W_{\fk X_b}(\Wbb_\blt)$ is defined near \eqref{eq34}. We shall identify these two spaces in the following.

\begin{df}\label{lb17}
A \textbf{conformal block} $\upphi$ associated to $\Wbb_\blt,\fk X$ (and defined over $\mc B$) is an $\scr O_{\mc B}$-module homomorphism $\upphi:\scr W_{\fk X}(\Wbb_\blt)\rightarrow\scr O_{\mc B}$ whose restriction to each fiber $\scr W_{\fk X_b}(\Wbb_\blt)$ is an element of $\scr T_{\fk X_b}^*(\Wbb_\blt)$, i.e., vanishes on $H^0\big(\mc C_b,\scr V_{\mc C_b}\otimes\omega_{\mc C_b}(\blt\SX(b))\big)\cdot\scr W_{\fk X_b}(\Wbb_\blt)$. The vector space of all such $\upphi$ is denoted \index{T@$\scr T_{\fk X}^*(\Wbb_\blt),\scr T_{\fk X}^*(\Wbb_\blt)(\mc B)$} by $\scr T_{\fk X}^*(\Wbb_\blt)(\mc B)$.
\end{df}

\begin{rem}
The holomorphicity of a conformal block $\upphi$ as a function on $\mc B$ would be easier to understand by shrinking $\mc B$ and choosing local coordinates $\eta_1,\dots,\eta_N$, and identifying $\scr W_{\fk X}(\Wbb_\blt)$  with $\Wbb_\blt\otimes_\Cbb\scr O_{\mc B}$ via $\mc U(\eta_\blt)$. That $\upphi$ is a morphism $\scr W_{\fk X}(\Wbb_\blt)\rightarrow\scr O_{\mc B}$  means precisely that $\upphi$ is an $(\Wbb_\blt)^*$-valued function on $\mc B$ whose evaluation with any vector of $\Wbb_\blt$ is a holomorphic function on $\mc B$; also, $\upphi$ is determined by and can be reconstructed from the corresponding $\scr O(\mc B)$-module homomorphism $\scr W_{\fk X}(\Wbb_\blt)(\mc B)\rightarrow\scr O(\mc B)$. Moreover, $\upphi$ is in $\scr T_{\fk X}^*(\Wbb_\blt)(\mc B)$ if and only if for each $b\in\mc B$, $\upphi(b)$ as a linear functional on $\Wbb_\blt\simeq\scr W_{\fk X_b}(\Wbb_\blt)$ is in $\scr T_{\fk X_b}^*(\Wbb_\blt)$.
\end{rem}

\subsection*{Action of $H^0\big(\mc C,\scr V_{\fk X}\otimes\omega_{\mc C/\mc B}(\blt\SX)\big)$  on $\scr W_{\fk X}(\Wbb_\blt)(\mc B)$}

The above definition of conformal blocks is fiberwise, i.e., the restriction of $\upphi$ to each fiber of the family is a conformal block defined in Section \ref{lb15}. We now give a global description which relates conformal blocks to the sheaf of VOA $\scr V_{\fk X}$.

Recall there are natural equivalences  $\scr V_{\fk X}|\mc C_b\simeq\scr V_{\mc C_b}$ and $\omega_{\mc C/\mc B}|\mc C_b\simeq\omega_{\mc C_b}$. Thus, we can define a linear action of $H^0\big(\mc C,\scr V_{\fk X}\otimes\omega_{\mc C/\mc B}(\blt\SX)\big)$  on $\scr W_{\fk X}(\Wbb_\blt)(\mc B)$ whose restriction to each fiber $H^0(\mc C_b,\scr V_{\mc C_b}\otimes\omega_{\mc C_b}(\blt\SX(b)))$ acts on $\scr W_{\fk X_b}(\Wbb_\blt)$. To see that the image of this action is holomorphic on $\mc B$, let us write down the action more explicitly.

Let $z$ be the standard variable of $\Cbb$. If $\Wbb$ is a $\Vbb$-module, we have a linear map \index{Y@$Y,Y_\Wbb$}
\begin{gather}
Y_\Wbb:\big(\Vbb\otimes_\Cbb\scr O(\mc B)((z))\big)\otimes_{\scr O(\mc B)} (\Wbb\otimes_\Cbb\scr O(\mc B))\rightarrow \Wbb\otimes_\Cbb\scr O(\mc B)((z)),\nonumber\\
v(b,z)\otimes w(b)\mapsto Y_\Wbb(v(b,z),z)w(b)
\end{gather}
which is clearly an $\scr O(\mc B)((z))$-module homomorphism. Note that a section of $\Vbb\otimes_\Cbb\scr O_{\Cbb\times\mc B}$ on a neighborhood of $\{0\}\times \mc B$ can be regarded as an element of $\Vbb\otimes_\Cbb\scr O(\mc B)((z))$ by taking series expansion.

Now, choose local coordinates $\eta_\blt$ of $\fk X$. For each $1\leq i\leq N$, choose a neighborhood $U_i$ of $\sgm_i(\mc B)$ on which $\eta_i$ is defined. Identify $\scr W_{\fk X}(\Wbb_\blt)$ with $\Wbb_\blt\otimes_\Cbb\scr O_{\mc B}$ via $\mc U(\eta_\blt)$. By tensoring with the identity map of $\omega_{\mc C/\mc B}$,  the map \eqref{eq35} induces naturally an $\scr O_{U_i}$-module isomorphism (also denoted by $\mc U_\varrho(\eta_i)$):
\begin{align*}
\mc U_\varrho(\eta_i):\scr V_{\fk X}\otimes\omega_{\mc C/\mc B}(\blt S_{\fk X})\big|_{U_i}\xrightarrow{\simeq}\Vbb\otimes_{\Cbb}\omega_{\mc C/\mc B}(\blt S_{\fk X})\big|_{U_i}.
\end{align*}
Identify $U_i$ with $(\eta_i,\pi)(U_i)$ via $(\eta_i,\pi)$. Then $\eta_i$ as a variable equals the standard variable $z$. If $v$ is a section of $\scr V_{\fk X}\otimes\omega_{\mc C/\mc B}(\blt S_{\fk X})$ defined near $\sgm_i(\mc B)$ and $w_i\in\Wbb_i\otimes_\Cbb\scr O(\mc B)$,  \index{vw@$v\cdot w_i,v\cdot w_\blt$}   then again we have
\begin{align}
v\cdot w_i=\Res_{z=0}Y_{\Wbb_i}(\mc U_\varrho(\eta_i)v,z)w_i,\label{eq55}
\end{align}
which is clearly also an element of $\Wbb_i\otimes_\Cbb\scr O(\mc B)$. If $v\in H^0\big(\mc C,\scr V_{\fk X}\otimes\omega_{\mc C/\mc B}(\blt\SX)\big)$, then as in \eqref{eq27}, $v$ acts on $\Wbb_\blt\otimes_\Cbb\scr O(\mc B)$ by summing up the actions on the tensor components.

\subsection*{Another description of conformal blocks}

We now give a global description of conformal blocks, followed by an application. We shall \emph{not} assume $\fk X$ admits local coordinates. Then $\scr W_{\fk X}(\Wbb_\blt)$ is still an infinite-rank locally free $\scr O_{\mc B}$-module whose local trivializations are given by \eqref{eq37}. Again, we define conformal blocks using Definition \ref{lb17}.

\begin{thm}\label{lb18}
Let $V\subset \mc B$ be an open subset of $\mc B$, and let $\upphi:\scr W_{\fk X_V}(\Wbb_\blt)\rightarrow\scr O_V$ be an $\scr O_V$-module homomorphism. Assume that $\mc B$ is Stein. Then $\upphi$ is in $\scr T_{\fk X_V}^*(\Wbb_\blt)(V)$ if and only if the  evaluation of $\upphi$ with the restriction  of any element \index{JB@$\scr J(\mc B)$} of 
\begin{align}
\scr J(\mc B)=H^0\big(\mc C,\scr V_{\fk X}\otimes\omega_{\mc C/\mc B}(\blt\SX)\big)\cdot \scr W_{\fk X}(\Wbb_\blt)(\mc B)\label{eq51}
\end{align}
to $V$ is the zero function.
\end{thm}

Note that we have omitted $\Span_\Cbb$ in \eqref{eq51}.

\begin{proof}
The ``only if" part is obvious by the definition of conformal blocks and does not require $\mc B$ to be Stein. The ``if" part follows from Corollary \ref{lb16}.
\end{proof}

For any open subset $V\subset\mc B$, let $\mc C_V=\pi^{-1}(V)$. \index{CVXV@$\mc C_V,\fk X_V$} Then $\fk X$ restricts to
\begin{align*}
\fk X_V:=(\pi:\mc C_V\rightarrow V;\sgm_1|_V,\dots,\sgm_N|_V).
\end{align*}
Any set of local coordinates $\eta_\blt$ of $\fk X$ restricts to one of $\fk X_V$.  We also write $\scr T_{\fk X_V}^*(\Wbb_\blt)(V)$ as $\scr T_{\fk X}^*(\Wbb_\blt)(V)$. The reason for this notation will be explained shortly.

\begin{pp}\label{lb36}
Let $\upphi:\scr W_{\fk X}(\Wbb_\blt)\rightarrow\scr O_{\mc B}$ be an $\scr O_{\mc B}$-module homomorphism. Assume that $\mc B$ is connected and contains a non-empty open subset $V$ such that $\upphi|_V:\scr W_{\fk X}(\Wbb_\blt)|_V\rightarrow\scr O_V$ is an element of $\scr T_{\fk X}^*(\Wbb_\blt)(V)$. Then $\upphi$ is an element of $\scr T_{\fk X}^*(\Wbb_\blt)(\mc B)$.
\end{pp}

\begin{proof}
We first assume $\mc B$ is also Stein. The evaluation of $\upphi$ with any element of $\scr J(\mc B)$ is a holomorphic function on $\mc B$ vanishing on $V$. So it must be constantly $0$. So $\upphi$ is a conformal block by Theorem \ref{lb18}. 

Now, we do not assume $\mc B$ is Stein. Let $A$ be the set of all $b\in\mc B$ such that $b$ has a  neighborhood $U$ satisfying that the restriction $\upphi|_U$ is a conformal block. Then $A$ is open and  non-empty. For any $b\in\mc B-A$, let $U$ be a connected Stein neighborhood of $b$. Then, by the first paragraph, $\upphi|_U$ is a conformal block if $U$ has a non-zero open subset $V$ such that $\upphi|_V$ is a conformal block. Therefore $U$ must be disjoint from $A$. This shows that $\mc B-A$ is open. So $\mc B=A$.
\end{proof}

\begin{rem}\label{lb22}
It is clear that all $\scr T_{\fk X}^*(\Wbb_\blt)(V)$ (where $V\subset\mc B$) form a sheaf of $\scr O_{\mc B}$-modules, which we denote by $\scr T_{\fk X}^*(\Wbb_\blt)$ \index{T@$\scr T_{\fk X}^*(\Wbb_\blt),\scr T_{\fk X}^*(\Wbb_\blt)(\mc B)$} and call the \textbf{sheaf of conformal blocks} associated to $\Wbb_\blt$ and $\fk X$. Let $\scr T_{\fk X}(\Wbb_\blt)$ be the sheaf of $\scr O_{\mc B}$-modules associated to the presheaf
\begin{align*}
V\mapsto \frac{\scr W_{\fk X}(\Wbb_\blt)(V)} {~H^0\big(\mc C_V,\scr V_{\fk X_V}\otimes\omega_{\mc C_V/V}(\blt S_{\fk X_V})\big)\cdot \scr W_{\fk X}(\Wbb_\blt)(V)~}
\end{align*}
and call it the \textbf{sheaf of covacua}. Then, by Theorem \ref{lb18}, $\scr T_{\fk X}^*(\Wbb_\blt)$ is the dual $\scr O_{\mc B}$-module of $\scr T_{\fk X}(\Wbb_\blt)$. Moreover, using Theorem \ref{lb8} or Corollary \ref{lb16}, it is easy to see that for each $b\in\mc B$, the clearly surjective linear map
\begin{align}
\frac{\scr T_{\fk X}(\Wbb_\blt)_b}{\fk m_b\cdot \scr T_{\fk X}(\Wbb_\blt)_b}\rightarrow\scr T_{\fk X_b}(\Wbb_\blt)\label{eq38}
\end{align}
defined by the restriction $\scr W_{\fk X}(\Wbb_\blt)_b\rightarrow \scr W_{\fk X}(\Wbb_\blt)|b$ is injective. Thus, the above two vector spaces have the same dimension, and are clearly equal to that of $\scr T_{\fk X_b}^*(\Wbb_\blt)$.

If $\Vbb$ is $C_2$-cofinite, rational, and $\Vbb(0)=\Cbb\id$, then by the main results of \cite{DGT19b}, the function $b\mapsto \dim\scr T_{\fk X_b}(\Wbb_\blt)$ has finite values and is locally constant.\footnote{Since each  fiber $\mc C_b$ is a projective variety, and since each $\scr V_{\fk X}^{\leq n}$ is locally free and (by comparing the two constructions) equivalent to the analytification of the corresponding one in \cite{DGT19a,DGT19b},  the algebraic results of $\scr T_{\fk X_b}(\Wbb_\blt)$ also hold in the analytic setting.} By Theorem \ref{lb21} (applied to any subfamily $\fk X_V$ where $V$ is Stein), the $\scr O_{\mc B}$-module $\scr T_{\fk X}(\Wbb_\blt)$ is finitely-generated. Thus, $\scr T_{\fk X}(\Wbb_\blt)$ is locally free by an easy consequence of Nakayama's Lemma. (Cf. for instance \cite[Lemma III.1.6]{BS76}.) Thus, $\scr T_{\fk X}^*(\Wbb_\blt)$ is also locally free, and by \eqref{eq38}, its fibers are naturally equivalent to the spaces of conformal blocks. As a consequence, any conformal block associated to a fiber $\fk X_b$ can be extended to a conformal block associated to a family $\fk X_V$ whenever $V$ is Stein.

If $\Vbb$ is only $C_2$-cofinite and $\Vbb(0)=\Cbb\id$, but $\fk X$ is a smooth family, then our claims about $\scr T_{\fk X}(\Wbb_\blt)$ and $\scr T_{\fk X}^*(\Wbb_\blt)$ and their fibers in the last paragraph are still true since,  by \cite{DGT19b}, the above result about dimensions still holds. 
\end{rem}

\section{A finiteness theorem}

Recall that $\Vbb$ is called \textbf{$C_2$-cofinite} if the subspace of $\Vbb$ spanned by $C_2(\Vbb):=\{Y(u)_{-2}v:u,v\in\Vbb \}$  has finite codimension. $\Vbb$ is called \textbf{rational} if every admissible $\Vbb$-module is completely reducible. (An admissible $\Vbb$-module $\Wbb$ satisfies all the requirements in the definition of finitely-admissible $\Vbb$-modules, expect that each $\Wbb(n)$ is not assumed to be finite-dimensional.) The following important result is due to Miyamoto \cite[Lemma 2.4]{Miy04}; see also \cite{GN03} or \cite[Thm. 1]{Buhl02}.

\begin{thm}\label{lb19}
Assume $\Vbb$ is $C_2$-cofinite. Then there exists a finite set $\Ebb$ of homogeneous vectors of $\Vbb$   such that any weak $\Vbb$-module $\Wbb$ generated by a vector $w_0$ is spanned by
	\begin{align}
	Y_\Wbb(v_k)_{-n_k}Y_\Wbb(v_{k-1})_{-n_{k-1}}\cdots Y_\Wbb(v_1)_{-n_1} w_0\label{eq39}
	\end{align}
where $k\in\Nbb$, $v_1,\dots,v_k\in\Ebb$, and $n_1<n_2<\cdots<n_k$.
\end{thm}

Note that by the lower-truncation property (i.e., for each $v,w$, $Y_\Wbb(v)_nw=0$ for sufficiently large $n$) and that $\Ebb$ is finite, we can assume in the above theorem that $n_1>-L$ for some constant $L\in\Nbb$ depending only on $\Vbb,\Wbb,\Ebb,w_0$. 

\begin{rem}\label{lb61}
It follows that if $\Vbb$ is $C_2$-cofinite, a $\Vbb$-module $\Wbb$ is finitely-generated if and only if it is finitely $\Lss$-semisimple. Indeed, if $\Wbb$ is finitely $\Lss$-semisimple, then \cite[Cor. 3.16]{Hua09} shows that it has finite length, and hence is finitely-generated. In particular, since $\Lni\in\End_\Vbb(\Wbb)$, we see that $\Lni^k$ vanishes on $\Mbb$ for sufficiently large $k$ since it vanishes on those generating vectors. Conversely, if $\Wbb$ is finitely generated, then it is finitely generated by $\Lss$-eigenvectors. Then the above theorem proves the $\Lss$-semisimplicity.
\end{rem}

We will fix this $\Ebb$ in this section.  Theorem \ref{lb19} will be used in the following form.

\begin{co}\label{lb20}
Assume that $\Vbb$ is $C_2$-cofinite. Let $\Wbb$ be a finitely-generated $\Vbb$-module. Then for any $n\in\Nbb$, there exists $\nu(n)\in\Nbb$ such that any $\wtd L_0$-homogeneous vector $w\in\Wbb$ whose weight $\wtd\wt(w)>\nu(n)$ is a sum of vectors in   $Y_\Wbb(v)_{-l}\Wbb(\wtd\wt(w)-\wt(v)-l+1)$ where $v\in\Ebb$ and $l>n$.
\end{co}

\begin{proof}
Assume without loss of generality that $\Wbb$ is  generated by a single $\wtd L_0$-homogeneous vector $w_0$. Let $T$ be the set of all vectors of the form \eqref{eq39} where $n_k\leq n$. Then, by the above theorem, $T$ is a finite subset of $\Wbb$. Set $\nu(n)=\max\{\wtd\wt(w_1):w_1\in T\}$. If $w\in\Wbb$ is $\wtd L_0$-homogeneous with weight $\wtd\wt(w)>\nu(n)$, then we can also write $w$ as a sum of $\wtd L_0$-homogeneous vectors (with weight $\wtd\wt (w)$) of the form \eqref{eq39}, but now the $n_k$ must be greater than $\nu$ since such vector is not in $T$. This proves that $w$ is a sum of $\wtd L_0$-homogeneous vectors of the form $Y_\Wbb(v)_{-l}w_2$ where $v\in\Ebb$, $l>n$, and $w_2\in\Wbb$ is $\wtd L_0$-homogeneous. The same is true if $\Wbb$ is finitely-generated. By \eqref{eq40}, we have $\wtd\wt(w_2)=\wtd\wt(w)-\wt(v)-l+1$.
\end{proof}

Let $\fk X=(\pi:\mc C\rightarrow\mc B;\sgm_1,\dots,\sgm_N;\eta_1,\dots,\eta_N)$ be a family of $N$-pointed complex curves with local coordinates, and  let $\Wbb_1,\dots,\Wbb_N$ be finitely-generated $\Vbb$-modules. Notice Theorem \ref{lb18} and recall the definition of $\scr J(\mc B)$ in \eqref{eq51}.

\begin{thm}\label{lb21}
Let $\Vbb$ be $C_2$-cofinite, and assume that $\mc B$ is a  Stein manifold. Then $\scr W_{\fk X}(\Wbb_\blt)(\mc B)/\scr J(\mc B)$ is a finitely-generated $\scr O(\mc B)$-module.
\end{thm}

\begin{proof}
Since local coordinates are chosen, we identify $\scr W_{\fk X}(\Wbb_\blt)$ with $\Wbb_\blt\otimes_\Cbb\scr O_{\mc B}$.  Let $E=\max\{\wt(v):v\in\Ebb\}$. By Theorem \ref{lb9}, there exists $k_0\in\Nbb$ such that
	\begin{align}
	H^1(\mc C_b,\scr V_{\mc C_b}^{\leq E}\otimes\omega_{\mc C_b}(kS_{\fk X}))=0\label{eq41}
	\end{align}
	for any $b\in\mc B$ and $k\geq k_0$. We fix an arbitrary $k\in\Nbb$ satisfying $k\geq E+k_0$.  
	
Introduce a weight $\wtd\wt$ on $\Wbb_\blt$ such that $\wtd\wt(w_\blt)=\wtd\wt(w_1)+\wtd\wt(w_2)+\cdots+\wtd\wt(w_N)$ when $w_1,\dots,w_N$ are $\wtd L_0$-homogeneous. For each $n\in\Nbb$, $\Wbb_\blt^{\leq n}$ (resp. $\Wbb_\blt(n)$) denotes the (finite-dimensional) subspace spanned by all $\wtd L_0$-homogeneous vectors $w\in\Wbb_\blt$ satisfying $\wtd\wt(w)\leq n$ (resp. $\wtd\wt(w)=n$).  We shall prove by induction that for any $n>N\nu(k)$,  any  vector of $\Wbb_\blt(n)$ (considered as constant sections of $\Wbb_\blt\otimes_\Cbb\scr O(\mc B)$) is a (finite) sum of elements of $\Wbb_\blt^{\leq n-1}\otimes_\Cbb\scr O(\mc B)$ mod  $\scr J(\mc B)$. Then the claim of our theorem follows.

Choose any $w_\blt=w_1\otimes\cdots\otimes w_N\in \Wbb_\blt(n)$ such that $w_1,\dots,w_N$ are $\wtd L_0$-homogeneous. Then one of $w_1,\dots,w_N$ must have $\wtd L_0$-weight greater than $\nu(k)$. Assume, without loss of generality, that $\wtd\wt(w_1)>\nu(k)$. Then, by Corollary \ref{lb20}, $w_1$ is a sum of non-zero $\wtd L_0$-homogeneous vectors of the form $Y_{\Wbb_1}(u)_{-l}w_1^\circ$ where $u\in\Ebb$, $l>k$, $w_1^\circ\in\Wbb_1$, and $\wtd\wt(w_1^\circ)=\wtd\wt(w_1)-\wt(u)-l+1$. Thus $\wtd\wt(w_1)-\wtd\wt(w_1^\circ)\geq l-1\geq k\geq E+k_0.$
	
It suffices to show that each $Y_{\Wbb_1}(u)_{-l}w_1^\circ\otimes w_2\otimes\cdots\otimes w_N$ is a sum of elements of $\Wbb_\blt^{\leq n-1}\otimes_\Cbb\scr O(\mc B)$ mod  $\scr J(\mc B)$. Thus, we assume for simplicity that $w_1=Y_{\Wbb_1}(u)_{-l}w_1^\circ$.  Then
	\begin{align*}
	w_\blt=Y_{\Wbb_1}(u)_{-l}w_1^\circ\otimes w_2\otimes\cdots\otimes w_N.
	\end{align*}
	Set also  
	\begin{align*}
	w_\blt^\circ=w_1^\circ\otimes w_2\otimes\cdots\otimes w_N.
	\end{align*}
	Then $n-\wtd\wt(w_\blt^\circ)= \wtd\wt(w_\blt)-\wtd\wt(w_\blt^\circ)\geq E+k_0.$ Thus
	\begin{align}
	\wtd\wt(w_\blt^\circ)\leq n-E-k_0.\label{eq42}
	\end{align}

Consider the short exact sequence of $\scr O_{\mc C}$-modules
	\begin{align*}
	0\rightarrow \scr V_{\mc C}^{\leq E}\otimes\omega_{\mc C/\mc B}(k_0S_{\fk X})\rightarrow \scr V_{\mc C}^{\leq E}\otimes\omega_{\mc C/\mc B}(lS_{\fk X})\rightarrow\scr G\rightarrow 0
	\end{align*}
	where $\scr G$ is the quotient of the previous two sheaves. By \eqref{eq41}, Proposition \ref{lb10}, and Grauert's Theorem \ref{lb11}, we see that $R^1\pi_*(\scr V_{\mc C}^{\leq E}\otimes\omega_{\mc C/\mc B}(k_0S_{\fk X}))=0$, and $\pi_*(\scr V_{\mc C}^{\leq E}\otimes\omega_{\mc C/\mc B}(k_0S_{\fk X}))$ is locally free. Thus, we obtain an exact sequence of $\scr O_{\mc B}$-modules
	\begin{align*}
	0\rightarrow \pi_*\big(\scr V_{\mc C}^{\leq E}\otimes\omega_{\mc C/\mc B}(k_0S_{\fk X})\big)\rightarrow \pi_*\big(\scr V_{\mc C}^{\leq E}\otimes\omega_{\mc C/\mc B}(lS_{\fk X})\big)\rightarrow \pi_*\scr G\rightarrow 0.
	\end{align*}
Since $\mc B$ is Stein, by Cartan's Theorem B, $H^1(\mc B,\pi_*(\scr V_{\mc C}^{\leq E}\otimes\omega_{\mc C/\mc B}(k_0S_{\fk X})))=0$. Thus, there is an exact sequence
	\begin{align}
	&0\rightarrow H^0\big(\mc B,\pi_*\big(\scr V_{\mc C}^{\leq E}\otimes\omega_{\mc C/\mc B}(k_0S_{\fk X})\big)\big)\rightarrow H^0\big(\mc B,\pi_*\big(\scr V_{\mc C}^{\leq E}\otimes\omega_{\mc C/\mc B}(lS_{\fk X})\big)\big)\nonumber\\
	&\rightarrow H^0\big(\mc B,\pi_*\scr G\big)\rightarrow 0. \label{eq43}
	\end{align}
	
	Note that $H^0\big(\mc B,\pi_*\scr G\big)$ is exactly $\scr G(\mc C)$. Choose mutually disjoint neighborhoods $W_1,\dots,W_N$ of $\sgm_1(\mc B),\dots,\sgm_N(\mc B)$ respectively. For each $1\leq i\leq N$, identify $\scr V_{\mc C}^{\leq E}\otimes\omega_{\mc C/\mc B}|_{W_i}$ with $\Vbb^{\leq E}\otimes_\Cbb\omega_{\mc C/\mc B}|_{W_i}$ via $\mc U_\varrho(\eta_i)$, and identify $\eta_i$ with the standard coordinate $z$ by identifying $W_i$ with $(\eta_i,\pi)(W_i)$. Define an element $\upupsilon\in\scr G(\mc C)$ as follows. $\upupsilon|_{W_1}$ is the equivalence class represented by $uz^{-l}dz$, and $\upupsilon|_{\mc C-\sgm_1(\mc B)}=0$. Since the second map in the above exact sequence is surjective, $\upupsilon$ lifts to an element $\wht\upupsilon$ of $H^0\big(\mc B,\pi_*\big(\scr V_{\mc C}^{\leq E}\otimes\omega_{\mc C/\mc B}(lS_{\fk X})\big)\big)$, i.e., of $\big(\scr V_{\mc C}^{\leq E}\otimes\omega_{\mc C/\mc B}(lS_{\fk X})\big)(\mc C)$. Moreover, by the definition of $\scr G$ as a quotient, for each $1\leq i\leq N$ we have an element $v_i$ of $\Vbb^{\leq E}\otimes_\Cbb\scr O_{\mc C}(k_0 S_{\fk X})(W_i)$ (and hence of $\Vbb^{\leq E}\otimes_\Cbb\scr O_{W_i}(k_0 \sgm_i(\mc B))(W_i)$) such that
	\begin{gather*}
	\wht\upupsilon|_{W_1}=uz^{-l}dz+v_1dz,\\
	\wht\upupsilon|_{W_i}=v_idz\qquad(2\leq i\leq N).
	\end{gather*}
	Notice that $\Res_{z=0}Y(\cdot,z)z^ndz=Y(\cdot)_n$. It follows that the element $\wht\upupsilon\cdot w_\blt^\circ$, which is in $\scr J(\mc B)$,  equals $w_\blt+w_\triangle$ where
	\begin{align*}
	w_\triangle=(v_1dz)\cdot w_1^\circ\otimes w_2\otimes\cdots \otimes  w_i\otimes\cdots\otimes w_N+\sum_{i=2}^Nw_1^\circ\otimes w_2\otimes\cdots \otimes (v_idz)\cdot w_i\otimes\cdots\otimes w_N.
	\end{align*}
	Thus $w_\blt$ equals $-w_\triangle$ mod $\scr J(\mc B)$. For each $1\leq i\leq N$, $v_i$ has pole at $z=0$ with order at most $k_0$. Thus, by \eqref{eq40}, the action of $v_idz$ on $\Wbb_i$ increases the $\wtd L_0$-weight by at most $E+k_0-1$. It follows from \eqref{eq42} that $w_\triangle\in\Wbb_\blt^{\leq {n-1}}\otimes_\Cbb\scr O(\mc B)$. The proof is complete.
\end{proof}

\begin{rem}\label{lb54}
Theorem \ref{lb21} is the complex-analytic analogue of \cite{DGT19b} Thm. 8.4.2, which says that for an algebraic family of complex curves, the sheaf of covacua is coherent (assuming that $\Vbb$ is $C_2$-cofinite and $\Vbb(0)=\Cbb\id$). The key ideas in our proof are similar to theirs.

It is clear that Theorem \ref{lb21} implies $\scr T_{\fk X}(\Wbb_\blt)$ is a finitely-generated $\scr O_{\mc B}$-module. However,  Theorem \ref{lb21} does not seem to imply that $\scr T_{\fk X}(\Wbb_\blt)$ is analytically coherent. This is different from the algebraic setting in  which the sheaves of covacua are algebraically coherent since they are quasi-coherent. Nevertheless, one can show that $\scr T_{\fk X}(\Wbb_\blt)$ is locally free (if $\Vbb$ is also rational)  by combining  the algebraic results with Theorem \ref{lb21}, as explained in Remark \ref{lb22}.

We remark that certain forms of Theorem \ref{lb21} are well-known in the  low genus cases: see \cite[Lemma 4.4.1]{Zhu96}, \cite[Cor. 1.2]{Hua05a}, \cite[Cor. 3.4]{Hua05b}. 
\end{rem}

\section{Schwarzian derivatives}

Let $\fk X=(\pi:\mc C\rightarrow\mc B)$ be a family of compact Riemann surfaces.  Choose an open subset $U\subset \mc C$ and holomorphic functions $\eta,\mu:U\rightarrow\Cbb$ univalent on each fiber. If $f\in\scr O(U)$ and $\partial_\eta f$ is nowhere zero, we define the (partial) \textbf{Schwarzian derivative} of $f$ over $\eta$ to be \index{S@$\Sbf_\eta f,\Sbf_\eta\fk P$}
\begin{align}
\Sbf_\eta f=\frac{\partial_\eta^3f}{\partial_\eta f}-\frac 32 \Big(\frac{\partial_\eta^2f}{\partial_\eta f} \Big)^2
\end{align}
where the partial derivative $\partial_\eta$ is defined with respect to $(\eta,\pi)$, i.e., it is annihilated by $d\pi$ and restricts to $d/d\eta$ on each fiber. Similarly, one can define $\Sbf_\mu f$. 

We refer the reader to \cite{Ahl,Gun} for the basic facts about Schwarzian derivatives. The change of variable formula is easy to calculate:
\begin{align}\label{eq45}
\Sbf_\mu f =(\partial_\mu\eta)^2 \Sbf_\eta f+\Sbf_\mu\eta,
\end{align}
Take $f=\mu$ and notice $\Sbf_\mu\mu=0$, we have
\begin{align}
\Sbf_\mu\eta=-(\partial_\mu\eta)^2\Sbf_\eta\mu.\label{eq46}
\end{align}
Assuming $f$ is also univalent on each fiber, we obtain the cocycle relation.
\begin{align}
\Sbf_\mu\eta\cdot d\mu^2=-\Sbf_\eta\mu\cdot d\eta^2,\qquad \Sbf_\mu f\cdot d\mu^2 + \Sbf_f \eta\cdot df^2+\Sbf_\eta\mu\cdot d\eta^2=0.\label{eq49}
\end{align}

The transition functions of $\svir_c$ (which is a subsheaf of $\scr V_{\fk X}$) defined in Section \ref{lb23} can be expressed by Schwarzian derivatives. Note that $L_0\cbf=2\cbf$, $L_1\cbf=0$, $L_2\cbf=\frac c 2\id$,\footnote{$L_2\cbf=L_2L_{-2}\id=[L_2,L_{-2}]\id=4L_0\id+\frac c 2\id=\frac c 2\id$.} and $L_n\cbf=0$ for all $n>2$. Thus,  if $\rho=\rho(z)\in\Gbb$, then using the formula \eqref{eq1}, we have $\mc U(\rho)\cbf=\rho'(0)^{L_0}e^{c_2L_2}\cbf=\rho'(0)^{L_0}(\cbf+\frac c2c_2\id)=\rho'(0)^2\cbf+\frac c2c_2\id$ where $c$ is the central charge of $\Vbb$, and $c_2$, which is given by \eqref{eq48}, is $\frac 16\Sbf_z\rho(0)$. Replace $\rho$ by $\varrho(\eta|\mu):U\rightarrow\Gbb$. Then $\rho^{(n)}(0)$ should be replaced by $\partial_\mu^n\eta$. Thus the transition function $\mc U(\varrho(\eta|\mu))$ is described by
\begin{gather}
\mc U(\varrho(\eta|\mu))\id=\id,\qquad  \mc U(\varrho(\eta|\mu))\cbf=(\partial_\mu\eta)^2\cbf+\frac {c}{12}\Sbf_\mu\eta \cdot \id.\label{eq47}
\end{gather}

We recall some well-known properties of  Schwarzian derivatives. See \cite{Hub81}.

\begin{pp}\label{lb24}
	The following are true.
	\begin{enumerate}[label=(\arabic*)]
		\item If the restriction of $\eta$ to each fiber $U_b=U\cap\pi^{-1}(b)$ (where $b\in\mc B$) is a M\"obius transformation of $\mu$, i.e., of the form $\frac{a\mu+b}{c\mu+d}$ where $ad-bc\neq 0$, then $\Sbf_\mu \eta=0$.
		\item  Let $Q\in\scr O(U)$. Then, for each $x\in U$, one can find a neighborhood $V\subset U$ of $x$ and a function $f\in\scr O(V)$ univalent on each fiber $V_b=V\cap\pi^{-1}(b)$, such that $\Sbf_\eta f=Q$.
		\item If $f,g\in\scr O(U)$ are univalent on each fiber, then $\Sbf_\eta f=\Sbf_\eta g$ if and only if $\Sbf_f g=0$.
	\end{enumerate}
\end{pp}

We remark that the converse of (1) is also true: If $f$ is univalent on each fiber, and if $S_\eta f=0$, then the restriction of $f$ to each fiber is a M\"obius transformation of $\eta$.

\begin{proof}
(1) can be verified directly. To prove (2), we identify $U$ with an open subset of $\Cbb\times\mc B$ via $(\eta,\pi)$. So $\eta$ is identified with the standard coordinate $z$. We choose a neighborhood $V\subset U$ of $x$ of the form $\mc D\times T$ where $T\subset\mc B$ is open, and $\mc D$ is an open disc centered at point $p=\eta(x)\in\Cbb$. By basic theory of ODE, the  differential equation
	\begin{align*}
	\partial_z^2h+Qh/2=0.
	\end{align*}
have solutions $h_1,h_2\in\scr O(V)$ satisfying  the initial conditions $h_1(\cdot,p)=1,\partial_z h_1(\cdot,p)=0$ and $h_2(\cdot,p)=0,\partial_z h_2(\cdot,p)=1$. It is easy to check that $f:=h_2/h_1$ satisfies $\Sbf_z f=Q$, and is defined and satisfies $\partial_zf\neq 0$ near $\{p\}\times T$.
	
(3) follows from \eqref{eq45}, which says $\Sbf_\eta g=(\partial_\eta f)^2\Sbf_fg+\Sbf_\eta f$. 
\end{proof}

\begin{df}
	An open cover $(U_\alpha,\eta_\alpha)_{\alpha\in\fk A}$ of $\mc C$, where each open set $U_\alpha$ is equipped with a function $\eta_\alpha\in\scr O(U_\alpha)$ holomorphic on each fiber, is called a (relative) \textbf{projective chart} of $\fk X$, if for any $\alpha,\beta\in\fk A$, we have $\Sbf_{\eta_\beta}\eta_\alpha=0$ on $U_\alpha\cap U_\beta$. Two projective charts are called equivalent if their union is a projective chart. An equivalence class of projective charts is called a (relative) \textbf{projective structure}. Equivalently, a projective structure is a maximal projective chart.
\end{df}

If $\mc B$ is Stein, then $\fk X$ has a projective structure. See Section \ref{lb27}.

\begin{rem}\label{lb25}
	Let $\fk P$ be a projective chart on $\fk X$. Choose an open subset $U\subset\mc C$ and a fiberwisely univalent $\eta\in\scr O(U)$. One can \index{S@$\Sbf_\eta f,\Sbf_\eta\fk P$} define an element 
	\begin{align*}
	\Sbf_\eta\fk P\in \scr O(U)
	\end{align*}
	as follows. Choose any $(U_1,\mu)\in\fk P$. Then $\Sbf_\eta\fk P=\Sbf_\eta\mu$ on $U\cap U_1$. To check that $\Sbf_\eta\fk P$ is well defined, suppose there is another $(U_2,\zeta)\in\fk P$. Then $\Sbf_\mu\zeta=0$ on $U_1\cap U_2$. Thus $\Sbf_\eta\mu=\Sbf_\eta\zeta$ on $U\cap U_1\cap U_2$ by Proposition \ref{lb24}-(3).
\end{rem}

\section{Actions of $H^0(\mc C,\Theta_{\mc C/\mc B}(\blt\SX))$}\label{lb41}

In this section, we fix  $\fk X=(\pi:\mc C\rightarrow\mc B;\sgm_1,\dots,\sgm_N;\eta_1,\dots,\eta_N)$ to be a family of $N$-pointed compact Riemann surfaces with local coordinates. We assume for simplicity that $\mc B$ is a Stein manifold with coordinates $\tau_\blt=(\tau_1,\dots,\tau_N)$. Let $\Wbb_1,\dots,\Wbb_N$ be $\Vbb$-modules. 

By Lemma \ref{lb4},  there exists $k_0\in\Nbb$ such that for any $k>k_0$ and $b\in \mc B$, we have $H^1(\mc C_b,\omega_{\mc C_b}(kS_{\fk X}))=0$. Thus $R^1\pi_*\omega_{\mc C/\mc B}(kS_{\fk X})=0$ (and also $\pi_*\omega_{\mc C/\mc B}(kS_{\fk X})$ is locally free) due to Grauert's Theorem \ref{lb11}. Therefore, \eqref{eq53} implies an exact sequence
\begin{align*}
0\rightarrow \pi_*\omega_{\mc C/\mc B}(k\SX) \rightarrow \pi_*\big(\svir_c\otimes \omega_{\mc C/\mc B}(k\SX)\big)\xrightarrow{\uplambda}   \pi_*\Theta_{\mc C/\mc B}(k\SX)\rightarrow 0.
\end{align*}
By Cartan's Theorem B, $H^1(\mc B,\pi_*\omega_{\mc C/\mc B}(k\SX))=0$. So we have an exact sequence
\begin{align*}
0\rightarrow H^0\big(\mc B,\pi_*\omega_{\mc C/\mc B}(k\SX)\big) \rightarrow H^0\big(\mc B,\pi_*\big(\svir_c\otimes \omega_{\mc C/\mc B}(k\SX)\big)\big)\xrightarrow{\uplambda}   H^0\big(\mc B,\pi_*\Theta_{\mc C/\mc B}(k\SX)\big)\rightarrow 0.
\end{align*}
Take the direct limit over all $k>k_0$, we get an exact sequence
\begin{align}
0\rightarrow H^0\big(\mc B,\pi_*\omega_{\mc C/\mc B}(\blt\SX)\big) \rightarrow H^0\big(\mc B,\pi_*\big(\svir_c\otimes \omega_{\mc C/\mc B}(\blt\SX)\big)\big)\xrightarrow{\uplambda}   H^0\big(\mc B,\pi_*\Theta_{\mc C/\mc B}(\blt\SX)\big)\rightarrow 0.\label{eq78}
\end{align}

According to Section \ref{lb28}, $H^0\big(\mc B,\pi_*\big(\svir_c\otimes \omega_{\mc C/\mc B}(\blt\SX)\big)\big)$ acts on $\scr W_{\fk X}(\Wbb_\blt)(\mc B)$ which clearly descends to the trivial action on $\scr W_{\fk X}(\Wbb_\blt)(\mc B)/\scr J(\mc B)$. We shall use the above exact sequence to define an action of $H^0\big(\mc B,\pi_*\Theta_{\mc C/\mc B}(\blt\SX)\big)=H^0(\mc C,\Theta_{\mc C/\mc B}(\blt\SX))$ on $\scr W_{\fk X}(\Wbb_\blt)(\mc B)/\scr J(\mc B)$, which turns out to be an $\scr O(\mc B)$-scalar multiplication. This action depends on the local coordinates $\eta_\blt$.

Choose mutually disjoint neighborhoods $U_1,\dots,U_N$ of $\sgm_1(\mc B),\dots,\sgm_N(\mc B)$ on which $\eta_1,\dots,\eta_N$ are defined respectively. Write each $\tau_j\circ\pi$ as $\tau_j$ for short, so that $(\eta_i,\tau_\blt)$ is a set of coordinates of $U_i$. Set $U=U_1\cup\cdots\cup U_N$. Choose any $\theta\in H^0(\mc C,\Theta_{\mc C/\mc B}(\blt\SX))$, which, in each $U_i$, is expressed as
\begin{align}
\theta|_{U_i}=a_i(\eta_i,\tau_\blt)\partial_{\eta_i}.\label{eq52}
\end{align}
Define \index{zz@$\upnu(\theta)$}
\begin{align*}
\upnu(\theta)\in \big(\svir_c\otimes \omega_{\mc C/\mc B}(\blt S_{\fk X})\big)(U)
\end{align*}
such that
\begin{align}
\mc U_\varrho(\eta_i)\upnu(\theta)|_{U_i}=a_i(\eta_i,\tau_\blt)\cbf~d{\eta_i}.\label{eq54}
\end{align}
The action of $\theta$ on $\scr T_{\fk X}(\Wbb_\blt)$ is defined to be the action of $\upnu(\theta)$ as in Section \ref{lb28}, namely, is determined by
\begin{align}
\upnu(\theta)\cdot w_\blt=\sum_{i=1}^N w_1\otimes\cdots\otimes \upnu(\theta)\cdot w_i\otimes\cdots \otimes w_N
\end{align}
for any $w_\blt=w_1\otimes\cdots\otimes w_N\in\Wbb_\blt$, where $\upnu(\theta)\cdot w_i$ is described by \eqref{eq55}.

\begin{lm}\label{lb29}
Assume that $(U_1,\eta_1),\dots,(U_N,\eta_N)$ belong to a projective structure $\fk P$.  Then the action of $\upnu(\theta)$ on $\scr W_{\fk X}(\Wbb_\blt)(\mc B)/\scr J(\mc B)$ is zero.
\end{lm}

\begin{proof}
By \eqref{eq47}, the transition function for $\cbf \otimes \omega_{\mc C/\mc B}$  between two projective coordinates is the same as that for $\Theta_{\mc C/\mc B}$, namely, when $\Sbf_\mu\eta=0$, $\partial_\mu$ changes to $\partial_\mu\eta\cdot \partial_\eta$, and $\cbf d\mu$ changes to $\partial_\mu\eta\cdot \cbf d\eta$, sharing the same transition function $\partial_\mu\eta$. Thus, as $\theta$ is over $\mc C$, $\upnu(\theta)$ can be extended to a global section of $\svir_c\otimes \omega_{\mc C/\mc B}(\blt S_{\fk X})$ on $\mc C$.  Thus, $\upnu(\theta)$ acts trivially since $\upnu(\theta)\cdot \scr W_{\fk X}(\Wbb_\blt)(\mc B)\subset\scr J(\mc B)$.
\end{proof}

\begin{pp}\label{lb30}
Let $\fk P$ be a projective structure of $\fk X$. Choose $\theta\in H^0(\mc C,\Theta_{\mc C/\mc B}(\blt\SX))$ whose local expression is given by \eqref{eq52}. Then the action of $\upnu(\theta)$ on $\scr W_{\fk X}(\Wbb_\blt)(\mc B)/\scr J(\mc B)$ (defined by the local coordinates $\eta_\blt$) is the $\scr O(\mc B)$-scalar multiplication  by
	\begin{align}
	\#(\theta):=\frac{c}{12}\sum_{i=1}^N \Res_{\eta_i=0}~ \Sbf_{\eta_i}\fk P\cdot a_i(\eta_i,\tau_\blt)~d\eta_i.\label{eq56}
	\end{align}
\end{pp}
Note that each $\Sbf_{\eta_i}\fk P$ (defined in Remark \ref{lb25}) is an element of $\scr O(U_i)$.

\begin{proof}
	It suffices to prove that the claim is locally true. Thus, we may shrinking $\mc B$ and $U_1,\dots,U_N$ so  that for each $1\leq i\leq N$, there exists a coordinate  $\mu_i\in\scr O(U_i)$ at $\sgm_i(\mc B)$ such that $(U_i,\mu_i)\in\fk P$. Then
	\begin{align*}
	\theta|_{U_i}=a_i(\eta_i,\tau_\blt)\cdot(\partial_{\mu_i}\eta_i)^{-1}\partial_{\mu_i}.
	\end{align*}
	Our strategy is to compare the action $\wtd\upnu(\theta)$ of $\theta$ defined by the coordinates $\mu_\blt$ (which is trivial by Lemma \ref{lb29}) with the one $\upnu(\theta)$ defined by $\eta_\blt$. We have  $\mc U_\varrho(\mu_i)\wtd\upnu(\theta)|_{U_i}=a_i(\eta_i,\tau_\blt)\cdot(\partial_{\mu_i}\eta_i)^{-1}\cbf~d{\mu_i}$ on each $U_i$. Then
	\begin{align*}
	\mc U_\varrho(\mu_i)\wtd\upnu(\theta)|_{U_i}=a_i(\eta_i,\tau_\blt)\cdot(\partial_{\mu_i}\eta_i)^{-2}\cbf~d{\eta_i}.
	\end{align*}
	By Lemma \ref{lb29}, the action of $\wtd\upnu(\theta)$ on $\scr W_{\fk X}(\Wbb_\blt)(\mc B)/\scr J(\mc B)$ is zero. Notice that the action of $\wtd\upnu(\theta)$ is independent of the choice of local coordinates. (See Theorem \ref{lb13}.) By \eqref{eq47}, we have
	\begin{align*}
	&\mc U_\varrho(\eta_i)\wtd\upnu(\theta)|_{U_i}=\mc U(\varrho(\eta_i|\mu_i))\mc U_\varrho(\mu_i)\wtd\upnu(\theta)|_{U_i}\\
	=&a_i(\eta_i,\tau_\blt)\cbf~d{\eta_i}+\frac{c}{12} a_i(\eta_i,\tau_\blt)\cdot(\partial_{\mu_i}\eta_i)^{-2}\Sbf_{\mu_i}\eta_i\cdot \id~d{\eta_i}
	\end{align*}
	By \eqref{eq54} and \eqref{eq46}, we have
	\begin{align*}
	\mc U_\varrho(\eta_i)\wtd\upnu(\theta)|_{U_i}=&\mc U_\varrho(\eta_i)\upnu(\theta)|_{U_i}-\frac{c}{12} a_i(\eta_i,\tau_\blt)\cdot\Sbf_{\eta_i}\mu_i\cdot \id~d{\eta_i}\\
	=&\mc U_\varrho(\eta_i)\upnu(\theta)|_{U_i}-\frac{c}{12} a_i(\eta_i,\tau_\blt)\cdot\Sbf_{\eta_i}\fk P\cdot \id~d{\eta_i}.
	\end{align*}
	Since the action of $\wtd\upnu(\theta)$ is zero, the action of $\upnu(\theta)$ equals the sum over $i$ of the actions of $\frac{c}{12} a_i(\eta_i,\tau_\blt)\cdot\Sbf_{\eta_i}\fk P\cdot \id~d{\eta_i}$, which is exactly the scalar multiplication by \eqref{eq56}.
\end{proof}

\section{Sewing conformal blocks}\label{lb35}

In this and the following sections, we let $\fk X=(\pi:\mc C\rightarrow\mc B;\sgm_1,\dots,\sgm_N;\eta_1,\dots,\eta_N)$ be a family of $N$-pointed complex curves with local coordinates obtained by sewing the following smooth family
\begin{align*}
\wtd{\fk X}=(\wtd\pi:\wtd{\mc C}\rightarrow\wtd{\mc B};\sgm_1,\dots,\sgm_N,\sgm',\sgm'';\eta_1,\dots,\eta_N,\xi,\varpi).
\end{align*}
(See Section \ref{lb6}.) Recall that the $N$-points  $\sgm_1,\dots,\sgm_N$ and the local coordinates $\eta_1,\dots,\eta_N$ of $\fk X$ are constant with respect to sewing. Choose $\Vbb$-modules $\Wbb_1,\dots,\Wbb_N,\Mbb$, which together with the contragredient module $\Mbb'$ are associated to $\sgm_1,\dots,\sgm_N,\sgm',\sgm''$ respectively. 

\begin{ass}\label{lb63}
In this section, we assume only that each connected component of each fiber $\wtd{\mc C}_b$ of $\wtd{\fk X}$ contains at least one of $\sgm_1(b),\dots,\sgm_N(b),\sgm'(b),\sgm''(b)$. This is slightly weaker than Assumption \ref{lb62}.
\end{ass}

\subsection*{Sewing conformal blocks}

Note $\Wbb_\blt\otimes\Mbb\otimes\Mbb'$ is $\Wbb_1\otimes\cdots\otimes \Wbb_N\otimes\Mbb\otimes\Mbb'$. Note also $(\Mbb'\otimes\Mbb)^*$ can be regarded as the algebraic completion of $\Mbb\otimes\Mbb'$.  Define \index{zzz@$\btr\otimes\btl$}
\begin{align*}
\btr\otimes\btl\in (\Mbb'\otimes\Mbb)^*
\end{align*}
such that for any $m'\in\Mbb',m\in\Mbb$,
\begin{align}
\bk{\btr\otimes\btl,m'\otimes m}=\bk{m',m}.\label{eq95}
\end{align}
Let $A\in\End(\Mbb)$ whose transpose $A^\tr\in\End(\Mbb')$ exists, i.e.,
\begin{align}
\bk{Am,m'}=\bk{m,A^\tr m'}\label{eq57}
\end{align}
for any $m'\in\Mbb',m\in\Mbb$. Then we have an element
\begin{align}
A\btr\otimes\btl\equiv \btr\otimes A^\tr\btl\quad\in (\Mbb'\otimes\Mbb)^*\label{eq58}
\end{align}
whose value at each $m'\otimes m$ is \eqref{eq57}.

More explicitly, for each $n\in\Nbb$ we choose a basis $\{m(n,a)\}_a$ of the finite-dimensional vector space $\Mbb(n)$. Its dual basis $\{\wch m(n,a)\}_a$ is a basis of $\Mbb'(n)=\Mbb(n)^*$ satisfying $\bk{m(n,a),\wch m(n,b)}=\delta_{a,b}$. Then we have
\begin{gather*}
\btr\otimes\btl=\sum_{n\in\Nbb}\sum_am(n,a)\otimes\wch m(n,a),
\end{gather*}
and
\begin{align*}
&A\btr\otimes\btl=\sum_{n\in\Nbb}\sum_a A\cdot m(n,a)\otimes\wch m(n,a)\\
=& \btr\otimes A^\tr\btl=\sum_{n\in\Nbb}\sum_am(n,a)\otimes A^\tr\cdot \wch m(n,a).
\end{align*}

For each $n\in\Nbb$, let $P(n)$ be the projection of $\Mbb$ onto $\Mbb(n)$. \index{Pn@$P(n)$} Its transpose, which is the projection of $\Mbb'$ onto $\Mbb'(n)$, is also denoted by $P(n)$. Then we clearly have
\begin{align*}
P(n)\btr\otimes\btl=\btr\otimes P(n)\btl=\sum_am(n,a)\otimes\wch m(n,a)\qquad\in\Mbb\otimes\Mbb'.
\end{align*}
Recall $\wtd L_0^\tr=\wtd L_0$ by \eqref{eq99}. Define
\begin{align*}
q^{\wtd L_0}=\sum_{k\in\Nbb}P(n)q^n\qquad\in\End(\Mbb)[[q]].
\end{align*} 
Then we have
\begin{align}
q^{\wtd L_0}\btr\otimes\btl=\btr\otimes q^{\wtd L_0} \btl\qquad \in (\Mbb\otimes\Mbb')[[q]].\label{eq59}
\end{align}

We set
\begin{align}
q^{L_0}\btr\otimes\btl\qquad\in(\Mbb'\otimes\Mbb)^*[\log q]\{q\}\label{eq109}
\end{align}
sending each $w'\otimes w$ to $\bk{q^{L_0}w,w'}=\bk{w,q^{L_0}w'}$. When $\Mbb$ is finitely $\Lss$-semisimple, $\Lss-\wtd L_0$ is a scalar $\lambda\in\Cbb$.  Notice 
\begin{align}\label{eq100}
\Lni=L_0-\wtd L_0-\lambda.
\end{align}
For each vector $w$ in $\Mbb(n)$, we set
\begin{align*}
q^\Lni w=e^{\Lni\log q}w=\sum_{n\in\Nbb}\frac 1{n!}\Lni^nw\cdot (\log q)^n\in	\Mbb(n)[\log q].
\end{align*}
Then
\begin{align}
q^{L_0}\btr\otimes\btl=q^\lambda\cdot q^{\Lni}\cdot q^{\wtd L_0}\btr\otimes\btl=\sum_{n\in\Nbb}q^{\lambda+n}\sum_a q^\Lni m(n,a)\otimes\wch m(n,a)
\end{align}
is in $q^{\lambda}(\Mbb\otimes\Mbb')[\log q][[q]]$. Thus, when $\Mbb$ is in general only finitely $\Lss$-semisimple, we have
\begin{align}
q^{L_0}\btr\otimes\btl\in(\Mbb\otimes\Mbb')[\log q]\{q\}.	
\end{align}
From \eqref{eq101}, we have
\begin{align}
q\partial_q (q^{L_0}\btr\otimes\btl)=L_0q^{L_0}\btr\otimes\btl.\label{eq102}	
\end{align}

Note that $\scr O(\mc B)$ can be viewed as a subring of $\scr O(\wtd{\mc B})[[q]]$ by taking power series expansions. So $\scr O(\wtd{\mc B})[[q]]$ and $\scr O(\wtd{\mc B})[\log q]\{q\}$ are $\scr O(\mc B)$-modules.  For any $\uppsi\in\scr T_{\wtd{\fk X}}^*(\Wbb_\blt\otimes\Mbb\otimes\Mbb')(\wtd{\mc B})$, we define its (normalized) \textbf{sewing} $\wtd{\mc S}\uppsi$ which is an $\scr O(\mc B)$-module homomorphism \index{S@$\wtd{\mc S}\uppsi,\mc S\uppsi$}
\begin{align*}
\wtd{\mc S}\uppsi:\scr W_{\fk X}(\Wbb_\blt)(\mc B)=\Wbb_\blt\otimes_\Cbb\scr O(\mc B)\rightarrow\scr O(\wtd{\mc B})[[q]],
\end{align*}
and, in the case that $\Mbb$ is finitely $\Lss$-semisimple, the (standard) \textbf{sewing}
\begin{align*}
\mc S\uppsi:\scr W_{\fk X}(\Wbb_\blt)(\mc B)=\Wbb_\blt\otimes_\Cbb\scr O(\mc B)\rightarrow\scr O(\wtd{\mc B})[\log q]\{q\},
\end{align*}
as follows. Regard $\uppsi$ as an $\scr O(\wtd{\mc B})$-module homomorphism $\Wbb_\blt\otimes\Mbb\otimes\Mbb'\otimes_\Cbb\scr O(\wtd{\mc B})\rightarrow \scr O(\wtd{\mc B})$. $\wtd{\mc S}\uppsi$ is defined such that for any  $w\in\Wbb_\blt\otimes_\Cbb\scr O(\wtd{\mc B})$,
\begin{align}
\wtd{\mc S}\uppsi(w)=\uppsi\big(w\otimes q^{\wtd L_0}\btr\otimes\btl\big).
\end{align}
When $\Mbb$ (and hence $\Mbb'$) is finitely $\Lss$-semisimple, $\mc S\uppsi$ is defined in the same way except that $\wtd L_0$ is replaced by $L_0$, i.e.,
\begin{align}
{\mc S}\uppsi(w)=\uppsi\big(w\otimes q^{L_0}\btr\otimes\btl\big).\label{eq105}
\end{align}

\subsection*{Formal conformal blocks}

Our goal is to show that $\wtd{\mc S}\uppsi$ (and hence $\mc S\uppsi$) is a \textbf{formal conformal block} associated to $\fk X$ and $\Wbb_\blt$, which means $\wtd{\mc S}\uppsi$ vanishes on $\scr J(\mc B)$ (defined by \eqref{eq51}).\footnote{Due to Theorem \ref{lb18}, it would be proper to use this definition only when $\wtd{\mc B}$ and hence $\mc B$ are Stein.} To prove this, we first need:

\begin{lm}\label{lb32}
	Let $R$ be a unital commutative $\Cbb$-algebra.  For any $u\in\Vbb$ and $f\in R[[\xi,\varpi]]$, the following two elements of $(\Mbb\otimes\Mbb'\otimes R)[[q]]$ (where the tensor products are over $\Cbb$) are equal:
	\begin{align}
	&\Res_{\xi=0}~Y_{\Mbb}\big(\xi^{L_0}u,\xi\big)q^{\wtd L_0}\btr\otimes\btl\cdot f(\xi,q/\xi)\frac{d\xi}{\xi}\nonumber\\
	=&\Res_{\varpi=0}~q^{\wtd L_0}\btr\otimes Y_{\Mbb'}\big(\varpi^{L_0}\mc U(\upgamma_1)u,\varpi\big)\btl\cdot f(q/\varpi,\varpi)\frac{d\varpi}{\varpi}.\label{eq60}
	\end{align}
When $\Mbb$ is finitely $\Lss$-semisimple, the same relation holds on the level of $(\Mbb\otimes\Mbb'\otimes R)[\log q]\{q\}$ if $\wtd L_0$ is replaced by $L_0$.
\end{lm}

\begin{rem}\label{lb33}
We  explain the meaning of the left hand side; the other side can be understood in a similar way. In the case that $\Mbb$ is finitely $\Lss$-semisimple, as   $q^{L_0}\btr\otimes\btl$  is an element of $(\Mbb\otimes\Mbb')[\log q]\{q\}$, $Y_{\Mbb}\big(\xi^{L_0}u,\xi\big)q^{L_0}\btr\otimes\btl$ is an element of $(\Mbb\otimes\Mbb')((\xi))[\log q]\{q\}$. 

Identify $(\Mbb\otimes\Mbb')((\xi))[\log q]\{q\}\simeq (\Mbb\otimes\Mbb'\otimes 1)((\xi))[\log q]\{q\}$, which is a subspace of the $R((\xi))[\log q]\{q\}$-module $(\Mbb\otimes\Mbb'\otimes R)((\xi))[\log q]\{q\}$.  On the other hand,  write $f(\xi,\varpi)=\sum_{m,n\in\Nbb}f_{m,n}\xi^m\varpi^n$ where each $f_{m,n}$ is in $R$. Then
	\begin{align*}
	f(\xi,q/\xi)=\sum_{n\geq 0}\sum_{k\geq -n}f_{n+k,n}\xi^k q^n,
	\end{align*}
	which shows $f(\xi,q/\xi)\in R((\xi))[[q]]$. Thus, the term in the residue on the left hand side is an element in
	\begin{align*}
	(\Mbb\otimes\Mbb'\otimes R)((\xi))[\log q]\{q\},
	\end{align*}
	whose residue is in $(\Mbb\otimes\Mbb'\otimes R)((\xi))[\log q]\{q\}$. In the case of $L_0$, the explanation is the same if we replace $[\log q]\{q\}$ by $[[q]]$.
	
We note that once \eqref{eq60} has been proved for $\wtd L_0$, when $\Mbb$ is $\Lss$-simple, we multiply both sides of \eqref{eq60} by $q^\Lni$ and notice that $q^\Lni$ commutes with the action of $Y_\Mbb$. So \eqref{eq60} holds on the level of $(\Mbb\otimes\Mbb'\otimes R)[\log q][[q]]$ if $\wtd L_0$ is replaced by $\wtd L_0+\Lni$. So \eqref{eq60} holds, for every $\Lss$-simple and hence finitely $\Lss$-semisimple $\Mbb$, if $\wtd L_0$ is replaced	by $L_0$.
\end{rem}

\begin{proof}[Proof of Lemma \ref{lb32}]
	Consider $Y_{\Mbb}\big(\xi^{L_0}u,\xi\big)q^{\wtd L_0}$ as an element of $\End(\Mbb)[[\xi^{\pm1},q]]$. Since $\wtd L_0^\tr=\wtd L_0$,  we have the following relations of elements of $\End(\Mbb')[[\xi^{\pm1},q^{\pm 1}]]$:
	\begin{align*}
	&\big(Y_{\Mbb}\big(\xi^{L_0}u,\xi\big)q^{\wtd L_0}\big)^\tr=q^{\wtd L_0}\big(Y_{\Mbb}\big(\xi^{L_0}u,\xi\big)\big)^\tr\xlongequal{\eqref{eq62}}q^{\wtd L_0}Y_{\Mbb'}\big(\mc U(\upgamma_\xi)\xi^{L_0}u,\xi^{-1}\big)\\
	\xlongequal{\eqref{eq22}}&q^{\wtd L_0}Y_{\Mbb'}\big(\xi^{-L_0}\mc U(\upgamma_1)u,\xi^{-1}\big)\xlongequal{\eqref{eq63}}Y_{\Mbb'}\big((q/\xi)^{L_0}\mc U(\upgamma_1)u,q/\xi\big)q^{\wtd L_0}.
	\end{align*}
Thus, by \eqref{eq58}, we have the following equations of elements in $(\Mbb'\otimes\Mbb)^*[[\xi^{\pm1},q^{\pm 1}]]$:
	\begin{align}
	&Y_{\Mbb}\big(\xi^{L_0}u,\xi\big)q^{\wtd L_0}\btr\otimes\btl=\btr\otimes \big(Y_{\Mbb}\big(\xi^{L_0}u,\xi\big)q^{\wtd L_0}\big)^\tr\btl\nonumber\\
	=&\btr\otimes~Y_{\Mbb'}\big((q/\xi)^{L_0}\mc U(\upgamma_1)u,q/\xi\big)q^{\wtd L_0}\btl=q^{\wtd L_0}\btr\otimes~Y_{\Mbb'}\big((q/\xi)^{L_0}\mc U(\upgamma_1)u,q/\xi\big)\btl.\label{eq61}
	\end{align}
	Since for each $n$, $P(n)\btr\otimes\btl$ is in $\Mbb\otimes\Mbb'$, \eqref{eq61} is actually an element in $(\Mbb\otimes\Mbb')[[\xi^{\pm1},q^{\pm 1}]]$.
	
	Let 
	\begin{gather*}
	A(\xi,q)=Y_{\Mbb}\big(\xi^{L_0}u,\xi\big)q^{\wtd L_0}\btr\otimes\btl,\\
	B(\varpi,q)=q^{\wtd L_0}\btr\otimes~Y_{\Mbb'}\big(\varpi^{L_0}\mc U(\upgamma_1)u,\varpi\big)\btl,
	\end{gather*}
	considered as elements of $(\Mbb\otimes\Mbb')[[\xi^{\pm1},q^{\pm 1}]]$ and $(\Mbb\otimes\Mbb')[[\varpi^{\pm1},q^{\pm 1}]]$ respectively. Then \eqref{eq61} says $A(\xi,q)=B(q/\xi,q)$. Let $C(\xi,\varpi)\in(\Mbb\otimes\Mbb')[[\xi^{\pm 1},\varpi^{\pm1}]]$  be $A(\xi,\xi\varpi)$, which also equals $B(\varpi,\xi\varpi)$. Since $A(\xi,q)$ contains only non-negative  powers of $q$, so does $A(\xi,\xi\varpi)$ for $\varpi$.   Similarly, since $B(\varpi,q)$ contains only non-negative powers of $q$, so does $B(\varpi,\xi\varpi)$ for $\xi$. Therefore $C(\xi,\varpi)$ is an element in  $(\Mbb\otimes\Mbb')[[\xi,\varpi]]$, where the latter  can be identified with the subspace $(\Mbb\otimes\Mbb'\otimes 1)[[\xi,\varpi]]$ of the $R[[\xi,\varpi]]$-module $(\Mbb\otimes\Mbb'\otimes R)[[\xi,\varpi]]$. Thus $D(\xi,\varpi):=f(\xi,\varpi)C(\xi,\varpi)$ is well-defined as an element in $(\Mbb\otimes\Mbb'\otimes R)[[\xi,\varpi]]$. It is easy to check that
	\begin{align*}
	\Res_{\xi=0}\bigg(D(\xi,q/\xi)\frac{d\xi}{\xi}\bigg)=\Res_{\varpi=0}\bigg(D(q/\varpi,\varpi)\frac{d\varpi}{\varpi}\bigg).
	\end{align*}
	(Indeed, they both equal $\sum_{n\in\Nbb}D_{n,n}q^n$ if we write $D(\xi,\varpi)=\sum_{m,n\in\Nbb}D_{m,n}\xi^m\varpi^n$.) This proves \eqref{eq60}.
\end{proof}

Recall $\mc B=\wtd{\mc B}\times\mc D_{r\rho}=\mc D_{r\rho}\times\wtd{\mc B}$, and the order of Cartesian products will be switched when necessary.

\begin{thm}\label{lb34}
	Let $\uppsi\in\scr T_{\wtd{\fk X}}^*(\Wbb_\blt\otimes\Mbb\otimes\Mbb')(\wtd{\mc B})$. Then $\wtd{\mc S}\uppsi$ vanishes on $\scr J(\mc B)$. If $\Mbb$ is  finitely $\Lss$-semisimple, then  $\mc S\uppsi$ also vanishes on $\scr J(\mc B)$.
\end{thm}

\begin{proof}
	Step 1. Note that we have divisors $S_{\fk X}=\sum_{i=1}^N\sgm_i(\mc B)$ and $S_{\wtd{\fk X}}=\sum_{i=1}^N\sgm_i(\wtd{\mc B})+\sgm'(\wtd{\mc B})+\sgm''(\wtd{\mc B})$ of $\mc C$ and $\wtd{\mc C}$ respectively. Choose any $v$ in $H^0\big(\mc C,\scr V_{\fk X}\otimes\omega_{\mc C/\mc B}(\blt S_{\fk X})\big)$.  In this first step, we would like to construct a formal power series expansion
	\begin{align}
	v=\sum_{n\in\Nbb}v_{n}q^n\label{eq64}
	\end{align}
	where each $v_{n}$ is in $H^0\big(\wtd{\mc C},\scr V_{\wtd{\fk X}}\otimes\omega_{\wtd{\mc C}/\wtd{\mc B}}(\blt S_{\wtd{\fk X}})\big)$.  
	
	First, choose any precompact open subset $U$ of $\wtd{\mc C}$ disjoint from the double points $\sgm'(\wtd{\mc B})$ and $\sgm''(\wtd{\mc B})$. Then  one can find small enough positive numbers $\epsilon<r,\lambda<\rho$  such that $U\times\mc D_{\epsilon\lambda}$ is an open subset of  $\wtd {\mc C}\times\mc D_{r\rho}- F'-F''$ in \eqref{eq67}, and hence an open subset of $\mc C$.   Moreover, by \eqref{eq68}, the projection $\pi:\mc C\rightarrow\mc B$ equals $\wtd \pi\times \id:\wtd {\mc C}\times\mc D_{r\rho}\rightarrow \wtd{\mc B}\times \mc D_{r\rho}$ when restricted to  $U\times \mc D_{\epsilon\lambda}$. It follows that the section $v|_{U\times \mc D_{\epsilon\lambda}}$  of $\scr V_{\fk X}\otimes\omega_{\mc C/\mc B}(\blt S_{\fk X})$ can be regarded as a section of $\scr V_{\wtd{\fk X}\times\mc D_{r\rho}}\otimes\omega_{\wtd{\mc C}\times\mc D_{r\rho}/\wtd{\mc B}\times\mc D_{r\rho}}(\blt S_{\fk X})$, which, by taking power series expansions at $q=0$, is in turn an element of $\big(\scr V_{\wtd{\fk X}}\otimes\omega_{\wtd{\mc C}/\wtd{\mc B}}(\blt S_{\wtd{\fk X}})\big)(U)[[q]]$. The coefficient before $q^{n}$ defines $v_{n}|_{U}$. This defines the section $v_{n}$ of $\scr V_{\wtd{\fk X}}\otimes\omega_{\wtd{\mc C}/\wtd{\mc B}}(\blt S_{\wtd{\fk X}})$ on $\wtd{\mc C}-\sgm'(\wtd{\mc B})- \sgm''(\wtd{\mc B})$ satisfying \eqref{eq64}.
	
	We now show that $v_{n}$ has poles of orders at most $n+1$ at $\sgm'(\wtd{\mc B})$ and $\sgm''(\wtd{\mc B})$. Let $W,W',W''$ be as described near \eqref{eq10}. By \eqref{eq21} and \eqref{eq24}, $v|_{W-\Sigma}$ is a sum of those whose restrictions to $W',W''$  under the trivializations $\mc U_\varrho(\xi),\mc U_\varrho(\varpi)$ are
	\begin{align}
	f(\xi,q/\xi,\cdot)\xi^{L_0}u\cdot \frac{d\xi}{\xi}  \qquad\text{resp.}\qquad -f(q/\varpi,\varpi,\cdot)\varpi^{L_0}\mc U(\upgamma_1)u\cdot \frac{d\varpi}{\varpi}\label{eq65}
	\end{align}
	where $u\in\Vbb$ and $f=f(\xi,\varpi,\cdot)\in\scr O(W)$, and the coordinates of $\wtd {\mc B}$ are suppressed as the dot. (Recall $q=\xi\varpi$.) In the above two terms, if we take power series expansions of $q$, then it is obvious that the coefficients before  $q^{n}$ have poles of orders at most $n+1$ at $\xi=0$ and $\varpi=0$ respectively. This proves the claim.

Step 2. Let us assume $\Mbb$ is finitely $\Lss$-semisimple (or just $\Lss$-simple, for the sake of simplicity) and prove that $\mc S\uppsi$ vanishes on $\scr J(\mc B)$. A similar method proves that $\wtd{\mc S}$ vanished on $\scr J(\mc B)$ in  the general case. 

By \eqref{eq10}, we can regard $f(\xi,\varpi,\cdot)$ as an element of $\scr O(\wtd{\mc B})[[\xi,\varpi]]$. Thus, by Lemma \ref{lb32} (applied to $R=\scr O(\wtd{\mc B})$) and the fact that $v|_{W-\Sigma}$ is a (finite) sum of those of the form \eqref{eq65}, we have the following equation of elements in $(\Mbb\otimes\Mbb'\otimes \scr O(\wtd{\mc B}))[\log q]\{q\}$:
	\begin{align}
	\sum_{n\in\Nbb}\big(v_{n}\cdot q^{L_0}\btr\otimes\btl+ q^{L_0}\btr\otimes~ v_{n}\cdot\btl\big)q^{n} =0\label{eq66}
	\end{align}
	where the actions of $v_{n}$ on $\Mbb$ and $\Mbb'$ are as in \eqref{eq55} using the  local coordinates $\xi,\varpi$ of $\wtd{\fk X}$. On the other hand, since $\uppsi$ is  conformal block, for each $n$ and each $w\in\Wbb_\blt$ (considered as a constant section of $\Wbb_\blt\otimes_\Cbb\scr O(\mc B)$), the element $A_{n}\in\scr O(\wtd{\mc B})[\log q]\{q\}$ defined by
	\begin{align*}
	A_{n}:=\uppsi\big(v_{n}\cdot w\otimes (q^{L_0}\btr\otimes\btl)\big)+\uppsi\big(w\otimes (v_n\cdot q^{L_0}\btr\otimes\btl)\big)+\uppsi\big(w\otimes ( q^{L_0}\btr\otimes v_n\cdot\btl)\big)
	\end{align*}
	equals $0$. Here, similarly, the action of $v_{n}$ on $w$ is defined by summing up the componentwise actions described by \eqref{eq55} using the local coordinates $\eta_\blt$.  By \eqref{eq66}, we have
	\begin{align*}
	0=\sum_{n\in\Nbb}A_{n}q^{n}=\sum_{n\in\Nbb}\uppsi\big(v_{n}\cdot w\otimes (q^{L_0}\btr\otimes\btl)\big)q^n,
	\end{align*}
	which is exactly $\mc S\uppsi(v\cdot w)$. This finishes the proof that $\mc S\uppsi$ vanishes on $\scr J(\mc B)$.
\end{proof}

\begin{rem}
The algebraic version of  Theorem \ref{lb34} for $\wtd{\mc S}\uppsi$ (i.e. assuming $\wtd{\fk X}$ is an algebraic family and replacing $\mc D_{r\rho}$ with $\mathrm{Spec}(\Cbb[[q]])$) was proved in \cite[Thm. 8.5.1]{DGT19b} and its proof can be easily adapted to the analytic setting. We have provided a complete proof of Theorem \ref{lb34} for the reader's convenience. We remark that \cite{DGT19b}  proved a version of Lemma \ref{lb32} (for $\wtd L_0$). Their proof uses \cite[Lemma 8.7.1]{NT05} and is different from ours.

In low genus cases, similar versions of Theorem \ref{lb34} were proved in \cite[Prop. 4.3.6]{Zhu96}, \cite[Thm. 1.4]{Hua05a}, \cite[Prop. 3.6]{Hua05b}.
\end{rem}

\section{Convergence of sewing}\label{lb52}

We continue our discussions and assume the setting in Section \ref{lb35}. Moreover, we assume Assumption \ref{lb62}.

Recall that $\mc B$ equals $\mc D_{r\rho}\times\wtd{\mc B}$. Then $\mc B$ is Stein if  $\wtd{\mc B}$ is so. Set $\mc B^\times=\mc D_{r\rho}^\times\times\wtd{\mc B}$. \index{B@$\mc B^\times$} We identify $\scr W_{\fk X}(\Wbb_\blt)$ with $\Wbb_\blt\otimes_\Cbb\scr O_{\mc B}$ via $\mc U(\eta_\blt)$.

\subsection*{Absolute and locally uniform convergence}

\begin{df}
We say that $\wtd{\mc S}\uppsi$  \textbf{converges absolutely and locally uniformly (a.l.u.)} if it sends each element of $\scr W_{\fk X}(\Wbb_\blt)(\mc B)$ to an element of $\scr O(\mc B)$. 
\end{df}

In the case that $\Mbb$ and hence $\Mbb'$ are finitely $\Lss$-semisimple, since $\mc S\uppsi$ is possibly multivalued over $q$, we need to define its a.l.u. convergence in another way. For each $w\in \Wbb_\blt$, considered as a constant section of $\scr W_{\fk X}(\Wbb_\blt)(\mc B)$, we write
\begin{align*}
\mc S\uppsi(w)=\sum_{n\in \Cbb,l\in\Nbb}\mc S\uppsi(w)_{n,l}\cdot (\log q)^lq^n.
\end{align*}
where  $\mc S\uppsi(w)_{n,l}$ is a holomorphic function on $\wtd{\mc B}$.

\begin{df}\label{lb37}
Assume $\mc S\uppsi(w)\in\scr O(\wtd{\mc B})\{q\}[\log q]$ (namely, $\mc S\uppsi(w)_{n,l}=0$ when $l\geq L$ for some $L$ and all $n$). We say that $\mc S\uppsi$ \textbf{converges a.l.u.} if  for any $w\in\Wbb_\blt$, any compact subsets $K\subset\wtd{\mc B}$ and $Q\subset\mc D_{r\rho}^\times$,  there exists $C>0$ such that
\begin{align}
	\sum_{n\in\Cbb}\big|\mc S\uppsi(w)_{n,l}(b)\big|\cdot |q^{n}|\leq C
\end{align}
for any $b\in K$, $q\in Q$, and $l\in\Nbb$. 
\end{df}

When $\Mbb$ is $\Lss$-simple, it is clear that $\wtd{\mc S}\uppsi$ converges a.l.u. if  $\mc S\uppsi$ does since $\wtd{\mc S}\uppsi$ is, up to multiplication by $q^d$ for some $d\in\Cbb$, the part of $\mc S\uppsi$ without $\log q$. Also, if we can prove the a.l.u. convergence of $\mc S\uppsi$ (resp. $\wtd{\mc S}\upphi$) whenever $\Mbb$ is $\Lss$-simple, then we can prove this for all finitely $\Lss$-semisimple $\Mbb$ due to Convention \ref{lb1}.

\begin{thm}\label{lb47}
The following are true.
\begin{enumerate}
\item  If  $\wtd{\mc S}\uppsi$  converges a.l.u., then it is an element  of $\scr T_{\fk X}^*(\Wbb_\blt)(\mc B)$. Similarly, if $\Mbb$ is finitely $\Lss$-semisimple and $\mc S\uppsi$ converges a.l.u., then it is an element of $\scr T_{\fk X_{\mc B^\times}}^*(\Wbb_\blt)(\mc B^\times)$.
\item Instead of Assumption \ref{lb62}, we assume only that for each $b\in\mc B^\times$, each connected component of $\mc C_b$ contains at least one of $\sgm_1(b),\dots,\sgm_N(b)$. If  $\wtd{\mc S}\uppsi$  converges a.l.u., then it is an element  of $\scr T_{\fk X_{\mc B^\times}}^*(\Wbb_\blt)(\mc B^\times)$. Similarly, if $\Mbb$ is finitely $\Lss$-semisimple and $\mc S\uppsi$ converges a.l.u., then it is an element of $\scr T_{\fk X_{\mc B^\times}}^*(\Wbb_\blt)(\mc B^\times)$.
\end{enumerate}
\end{thm}

For instance, case 2 includes the case of sewing a disjoint union $C_1\sqcup C_2$ of two connected compact Riemann surfaces along marked points $x'\in C_1$ and $x''\in C_2$, where $C_1$ has no other marked points and $C_2$ has at least one more marked point.

\begin{proof}
We may shrink $\wtd{\mc B}$ so that $\wtd{\mc B}$ and hence $\mc B$ are Stein. 

Case 1. Assume the a.l.u. convergence. Note that, since $\scr W_{\fk X}(\Wbb_\blt)$ is generated freely by some global sections, the $\scr O(\mc B^\times)$-module homomorphism  $\wtd{\mc S}\uppsi:\scr W_{\fk X}(\Wbb_\blt)(\mc B^\times)\rightarrow\scr O(\mc B^\times)$ (defined by taking the limit of the infinite series \eqref{eq105}) can be regarded as an $\scr O_{\mc B^\times}$-module homomorphism $\wtd{\mc S}\uppsi:\scr W_{\fk X_{\mc B^\times}}(\Wbb_\blt)\rightarrow\scr O_{\mc B^\times}$.  Now, the fact that $\wtd{\mc S}\uppsi$ is a conformal block follows  from Theorems \ref{lb34} and \ref{lb18}. The proof for $\mc S\uppsi$ is similar.

Case 2. For each $n\in\Nbb$ and $k\in\Nbb$, $\pi_*\big(\scr V_{\fk X}^{\leq n}\otimes\omega_{\mc C/\mc B}(k\SX)\big)$ is coherent by Grauert direct image theorem. Thus, for each $b\in\mc B^\times$, by Cartan's theorem A, elements of $H^0(\mc C,\scr V_{\fk X}^{\leq n}\otimes\omega_{\mc C/\mc B}(k\SX))$ generate the fiber $\pi_*\big(\scr V_{\fk X}^{\leq n}\otimes\omega_{\mc C/\mc B}(k\SX)\big)|b$, and the later is isomorphic to $H^0\big(\mc C_b,\scr V_{\mc C_b}^{\leq n}\otimes\omega_{\mc C_b}(k\SX(b))\big)$ for sufficiently large $k$ by Thm. \ref{lb8}. Thus elements of $H^0(\mc C,\scr V_{\fk X}\otimes\omega_{\mc C/\mc B}(\blt\SX))$ form the vector space $H^0\big(\mc C_b,\scr V_{\mc C_b}\otimes\omega_{\mc C_b}(\blt\SX(b))\big)$. This fact and Theorem \ref{lb34} prove the claim.
\end{proof}

\subsection*{Convergence and differential equations}

In the remaining part of this section, \emph{we assume $\Vbb$ is $C_2$-cofinite, the $\Vbb$-modules $\Wbb_1,\dots,\Wbb_N$ are finitely-generated (equivalently, finitely $\Lss$-semisimple cf. Rem. \ref{lb61}), and $\Mbb$ and hence $\Mbb'$ are $\Lss$-simple}. Since  $\Lni^k=0$ on $\Mbb$ when $k$ is sufficiently large (cf. Rem. \ref{lb61}), the powers of $\log q$ in the sewing $\mc S\uppsi$ are uniformly bounded from above. 

As in the proof of Theorem \ref{lb21}, for each $k\in\Nbb$, $\Wbb_\blt^{\leq k}$ (resp. $\Wbb_\blt(k)$) denotes the (finite dimensional) subspace spanned by all $\wtd L_0$-homogeneous vectors $w\in\Wbb_\blt$ satisfying $\wtd\wt(w)\leq k$ (resp. $\wtd\wt(w)=k$). This gives a filtration (resp. grading) of $\Wbb_\blt$. We define
\begin{align*}
\mc S\uppsi^{\leq k}\in (\Wbb_\blt^{\leq k})^*\otimes_\Cbb\scr O(\wtd{\mc B})\{q\}[\log q]
\end{align*}
whose evaluation with each $w\in\Wbb_\blt^{\leq k}$ is $\mc S\uppsi(w)$.

\begin{thm}\label{lb38}
Assume $\wtd{\mc B}$ is a Stein manifold. There exists $k_0\in\Zbb_+$ such that for any $k\geq k_0$, there exists
	\begin{align*}
\Upomega\in \End_{\Cbb}\big((\Wbb_\blt^{\leq k})^* \big)\otimes_\Cbb\scr O(\mc B)
	\end{align*}
not depending on $\Mbb$, such that
	\begin{align}
	q\partial_q (\mc S\uppsi^{\leq k})=\Upomega\cdot\mc S\uppsi^{\leq k}.\label{eq70}
	\end{align}
\end{thm}

Using this theorem, it is easy to prove:

\begin{thm}\label{lb48}
$\wtd{\mc S}\uppsi$ and $\mc S\uppsi$ converge a.l.u..
\end{thm}

\begin{proof}
It suffices to assume $\mc B$ is a Stein open subset of $\Cbb^m$. Then by Theorem \ref{lb39}, $\mc S\uppsi^{\leq k}$ converges a.l.u..
\end{proof}

It follows that $\wtd{\mc S}\uppsi$ is a conformal block associated to $\Wbb_\blt$ and $\fk X$. Outside the discriminant locus $\Delta=\{0\}\times\wtd{\mc B}$, $\mc S\uppsi$ is also a conformal block.

\subsection*{Proof of Theorem \ref{lb38}}

In this subsection, we assume $\wtd{\mc B}$ and hence $\mc B$ are Stein manifolds. By Theorem \ref{lb40}, we are allowed to fix a projective structure $\fk P$ of $\wtd{\fk X}$. Recall $\SX=\sum_{i=1}^N\sgm_i(\mc B)$; set $S_{\wtd{\fk X}}=\sum_{i=1}^N\sgm_i(\wtd{\mc B})+\sgm'(\wtd{\mc B})+\sgm''(\wtd{\mc B})$. As argued for \eqref{eq78}, we may use \eqref{eq18} to obtain an exact sequence
\begin{align}
0&\rightarrow H^0\big(\mc B,\pi_*\Theta_{\mc C/\mc B}(\blt S_{\fk X})\big)\rightarrow H^0\big(\mc B,\pi_*\Theta_{\mc C}(-\log \mc C_\Delta+\blt S_{\fk X})\big)\nonumber\\
&\xrightarrow{d\pi}H^0\big(\mc B,\pi_*\big(\pi^*\Theta_{\mc B}(-\log \Delta)(\blt S_{\fk X})\big)\big)
\rightarrow 0.\label{eq79}
\end{align}
$\yk=q\partial_q$ is a section of  $\Theta_{\mc B}(-\log\Delta)$ and hence of $\pi_*\big(\pi^*\Theta_{\mc B}(-\log \Delta)(\blt S_{\fk X})\big)$ over $\mc B$. Thus, we have
\begin{align*}
\wtd\yk\in H^0\big(\mc C,\Theta_{\mc C}(-\log \mc C_\Delta+\blt S_{\fk X})\big)
\end{align*}
satisfying $d\pi(\wtd\yk)=q\partial_q$. We let
\begin{align*}
\Gamma=\sgm'(\wtd{\mc B})\cup\sgm''(\wtd{\mc B}).
\end{align*}

Our first step is to take the series expansion $\sum\wtd\yk_n^\perp q^n$ (as in the proof of Theorem \ref{lb34}) of the ``vertical part" of $\wtd\yk$. Choose any precompact open subset $U\subset\wtd{\mc C}-\Gamma$ together with a fiberwisely univalent $\eta\in\scr O(U)$.  Then as in that proof, we may find a small subdisc $\mc D=\mc D_{\epsilon\lambda}$ of $\mc D_{r\rho}$ centered at $0$ such that $\mc D\times U\simeq U\times\mc D$ is an open subset of $\wtd{\mc C}\times \mc D_{r\rho}-F'-F''$ and hence of $\mc C$. Extend $\eta$ constantly (over $\mc D$) to a fiberwise univalent function on $\mc D\times U$. Then we may write
\begin{align}
\wtd\yk|_{\mc D\times U}=h\partial_\eta+q\partial_q\label{eq88}
\end{align}
for some $h\in\scr O(\blt\SX)(\mc D\times U)$. Write $h=\sum_{n\in\Nbb}h_nq^n$ where $h_n\in\scr O(\blt S_{\wtd{\fk X}})(U)$. For each $n\in\Nbb$, set an element $\wtd\yk^\perp_n\in\Theta_{\wtd{\mc C}/\wtd{\mc B}}(\blt S_{\wtd{\fk X}})(U)$ by
\begin{align}
\wtd\yk^\perp_n|_U=h_n\partial_\eta.\label{eq72}
\end{align}

\begin{lm}
	The locally defined $\wtd\yk^\perp_n$ is independent of the choice of $\eta$, and hence can be extended to an element of $H^0(\wtd {\mc C}-\Gamma,\Theta_{\wtd{\mc C}/\wtd{\mc B}}(\blt S_{\wtd{\fk X}}))$
\end{lm}

\begin{proof}
	Suppose we have another $\mu\in\scr O(U)$ univalent on each fiber, which is extended constantly to $\mc D\times U$. So $\partial_q\mu=0$ and hence $\wtd\yk|_{\mc D\times U}=h\cdot\partial_\eta\mu\cdot\partial_\mu+q\partial_q.$ Note that $\partial_\eta\mu$ is constant over $q$. Thus, if we define $\wtd\yk^\perp_n|_{U}$ using $\mu$, then $\wtd\yk^\perp_n|_{U}=h_n\cdot\partial_\eta\mu\cdot\partial_\mu$, which agrees with \eqref{eq72}.
\end{proof}

We shall show that $\wtd\yk^\perp_n$ has poles of finite orders at $\Gamma$. For that purpose, we need to describe explicitly $\wtd\yk$ near the critical locus $\Sigma$. Recall the open subsets $W,W',W''$ of $\mc C$ described near \eqref{eq10} and $U',U''$ of $\wtd{\mc C}$ described near \eqref{eq84}. Note also $q=\xi\varpi$. In the following, we let $\tau_\blt$ be any biholomorphic map from $\wtd{\mc B}$ to an open subset of a complex manifold. If $\wtd{\mc B}$ is small enough, then $\tau_\blt$ can be a set of coordinates of $\wtd{\mc B}$. The only purpose of introducing $\tau_\blt$ is  to indicate the dependence of certain functions on the points of $\wtd{\mc B}$.  Thus, $(\xi,q,\tau_\blt)$ and $(\varpi,q,\tau_\blt)$ are respectively biholomorphic maps of 
\begin{align*}
W'=\mc D_r^\times\times\mc D_\rho\times\wtd{\mc B},\qquad W''=\mc D_r\times\mc D_\rho^\times\times\wtd{\mc B}
\end{align*}
to complex manifolds. By \eqref{eq16}, we can find $a,b\in\scr O((\xi,\varpi,\tau_\blt)(W))$ such that
\begin{align*}
\wtd\yk|_W=a(\xi,\varpi,\tau_\blt)\xi\partial_\xi+b(\xi,\varpi,\tau_\blt)\varpi\partial_\varpi.
\end{align*}
Since $d\pi(\xi\partial_\xi)=d\pi(\varpi\partial_\varpi)=q\partial_q$ by \eqref{eq81}, we must have
\begin{align}
a+b=1.
\end{align}
This relation, together with \eqref{eq23}, shows that under the coordinates $(\xi,q,\tau_\blt)$ and $(\varpi,q,\tau_\blt)$ respectively,
\begin{gather}
\wtd\yk|_{W'}=a(\xi,q/\xi,\tau_\blt)\xi\partial_\xi+q\partial_q,\qquad \wtd\yk|_{W''}=b(q/\varpi,\varpi,\tau_\blt)\varpi\partial_\varpi+q\partial_q.\label{eq80}
\end{gather}

\begin{lm}
	For each $n\in\Nbb$, $\wtd\yk_n^\perp$ has poles of orders at most $n-1$ at $\sgm'(\wtd{\mc B})$ and $\sgm''(\wtd{\mc B})$. Consequently, $\wtd\yk^\perp_n$ is an element of $H^0(\wtd{\mc C},\Theta_{\wtd {\mc C}/\wtd {\mc B}}(\blt S_{\wtd{\fk X}}))$.
\end{lm}

\begin{proof}
	Let us write
	\begin{gather*}
	a(\xi,\varpi,\tau_\blt)=\sum_{m,n\in\Nbb}a_{m,n}(\tau_\blt)\xi^m\varpi^n,\qquad b(\xi,\varpi,\tau_\blt)=\sum_{m,n\in\Nbb}b_{m,n}(\tau_\blt)\xi^m\varpi^n
	\end{gather*}
	where $a_{m,n},b_{m,n}\in\scr O(\tau_\blt(\wtd{\mc B}))$. Then
	\begin{gather}
	a(\xi,q/\xi,\tau_\blt)=\sum_{n\geq 0,l\geq -n}a_{l+n,n}(\tau_\blt)\xi^lq^n,\qquad b(q/\varpi,\varpi,\tau_\blt)=\sum_{m\geq 0,l\geq -m}b_{m,l+m}(\tau_\blt)\varpi^lq^m.\label{eq82}
	\end{gather}
	Combine these two relations with \eqref{eq72} and \eqref{eq80}, and take the coefficients before $q^n$. We obtain
	\begin{align}
	\wtd\yk_n^\perp\Big|_{U'-\sgm'(\wtd{\mc B})}=\sum_{l\geq -n}a_{l+n,n}(\tau_\blt)\xi^{l+1}\partial_\xi,\qquad \wtd\yk_n^\perp\Big|_{U''-\sgm''(\wtd{\mc B})}=\sum_{l\geq -n}b_{n,l+n}(\tau_\blt)\varpi^{l+1}\partial_\varpi,\label{eq83}
	\end{align}
	which finishes the proof.
\end{proof}

The description \eqref{eq83} of $\wtd\yk_n^\perp$ near $\Gamma$ can be found in \cite[Lemma 33]{Loo10}. Next, we shall apply the results of Section \ref{lb41} to the smooth family $\wtd{\fk X}$. In particular, $\svir_c$ is defined for $\wtd{\fk X}$ and is an $\scr O_{\wtd{\mc C}}$-module. We let $\upnu(\wtd\yk_n^\perp)$  be a section of $\svir_c\otimes\omega_{\wtd{\mc C}/\wtd{\mc B}}(\blt S_{\wtd{\fk X}})$ defined on $U'\cup U''$ (near $\sgm'(\wtd{\mc B}),\sgm''(\wtd{\mc B})$) and near $\sgm_1(\wtd{\mc B}),\dots,\sgm_N(\wtd{\mc B})$ as in Section \eqref{lb41}, \emph{which relies on the local coordinates $\eta_1,\dots,\eta_N,\xi,\varpi$ of $\wtd{\fk X}$}. Recall the correspondence $\partial_\xi\mapsto \cbf d\xi,\partial_\varpi\mapsto \cbf d\varpi$. We calculate the actions of $\upnu(\wtd\yk_n^\perp)$ on $\Mbb$ and on $\Mbb'$ to be respectively
\begin{align}
\Res_{\xi=0}\sum_{l\geq -n}a_{l+n,n}Y_\Mbb(\cbf,\xi)\xi^{l+1}d\xi,\qquad \Res_{\varpi=0}\sum_{l\geq -n}b_{n,l+n}Y_{\Mbb'}(\cbf,\varpi)\varpi^{l+1}d\varpi.\label{eq85}
\end{align}
In the following proofs, we will suppress the symbol $\tau_\blt$ when necessary.

The next lemma is crucial to finding the differential equation \eqref{eq70}, and was observed in  \cite[Rem. 8.5.2]{DGT19b} when there is no $\log q$.

\begin{lm}\label{lb42}
	The following equation of elements of $(\Mbb\otimes\Mbb')\{q\}[\log q]$ is true.
	\begin{align}
	L_0q^{ L_0}\btr\otimes~\btl=\sum_{n\in\Nbb}\upnu(\wtd\yk_n^\perp)q^{n+ L_0}\btr\otimes~\btl+\sum_{n\in\Nbb}q^{n+ L_0}\btr\otimes ~\upnu(\wtd\yk_n^\perp)\btl\label{eq86}
	\end{align}
\end{lm}

\begin{proof}
	It is obvious that $\mc U(\upgamma_1)\cbf=\cbf$, $\xi^{L_0}\cbf=\xi^2\cbf$, $\varpi^{L_0}\cbf=\varpi^2\cbf$. Notice Remark \ref{lb33}. We have  
	\begin{align*}
	&Y_\Mbb(\xi^{L_0}\cbf,\xi) q^{ L_0}\btr\otimes~\btl\cdot a(\xi,q/\xi)\frac{d\xi}{\xi}\\
	=&\sum_{n\geq 0}\sum_{l\geq -n}Y_\Mbb(\cbf,\xi) q^{n+ L_0}\btr\otimes~\btl\cdot a_{l+n,n}\xi^{l+1}d\xi
	\end{align*}	
	as elements of $(\Mbb\otimes\Mbb'\otimes\scr O(\wtd{\mc B}))((\xi))\{q\}[\log q]d\xi$.  Take $\Res_{\xi=0}$ and notice \eqref{eq85}. Then, the above expression becomes the first summand on the right hand side of \eqref{eq86}. A similar thing could be said about the second summand. Thus, the right hand side of \eqref{eq86} equals
	\begin{align*}
	&\Res_{\xi=0}Y_\Mbb(\xi^{L_0}\cbf,\xi) q^{ L_0}\btr\otimes~\btl\cdot a(\xi,q/\xi)\frac{d\xi}{\xi}\\
	+&\Res_{\varpi=0} q^{ L_0}\btr\otimes~Y_{\Mbb'}(\varpi^{L_0}\mc U(\upgamma_1)\cbf,\varpi)\btl\cdot b(q/\varpi,\varpi)\frac{d\varpi}{\varpi}.
	\end{align*}
	By Lemma \ref{lb32} and that $a+b=1$, it equals
	\begin{align*}
	&\Res_{\xi=0}Y_\Mbb(\xi^{L_0}\cbf,\xi) q^{ L_0}\btr\otimes~\btl\cdot \frac{d\xi}{\xi}=\Res_{\xi=0}Y_\Mbb(\cbf,\xi) q^{ L_0}\btr\otimes~\btl\cdot \xi d\xi\\
	=&Y_\Mbb(\cbf)_1 q^{ L_0}\btr\otimes~\btl=L_0 q^{ L_0}\btr\otimes~\btl.
	\end{align*}
\end{proof}

\begin{lm}\label{lb43}
	For any $w_\blt\in\Wbb_\blt$, we have the following relation of elements of $\scr O(\wtd{\mc B})\{q\}[\log q]$.
	\begin{align*}
	q\partial_q\mc S\uppsi(w_\blt)=\sum_{n\in\Nbb}\uppsi(w_\blt\otimes \upnu(\wtd\yk_n^\perp)q^{n+L_0}\btr\otimes~\btl)+\sum_{n\in\Nbb}\uppsi(w_\blt\otimes q^{n+L_0}\btr\otimes ~\upnu(\wtd\yk_n^\perp)\btl).
	\end{align*}
\end{lm}

\begin{proof}
By \eqref{eq102}, we have
	\begin{align*}
	q\partial_q\mc S\uppsi(w_\blt)=q\partial_q\uppsi(w_\blt\otimes q^{L_0}\btr\otimes~\btl)=\uppsi(w_\blt\otimes L_0 q^{L_0}\btr\otimes~\btl).
	\end{align*}
	By the Lemma \ref{lb42}, the desired equation is proved.
\end{proof}

As usual, we let $\upnu(\wtd\yk_n^\perp)w_\blt$ denote $\sum_i w_1\otimes\cdots\otimes \upnu(\wtd\yk_n^\perp) w_i\otimes\cdots\otimes w_N$.   For any $w_\blt\in\Wbb_\blt$,  define $\nabla_{q\partial_q}w_\blt\in\Wbb_\blt\otimes_\Cbb\scr O(\wtd {\mc B})[[q]]$ to be
\begin{align}
\nabla_{q\partial_q}w_\blt=-\sum_{n\in\Nbb}q^n \upnu(\wtd\yk^\perp_n)w_\blt.\label{eq87}
\end{align}

\begin{pp}\label{lb44}
	There exists $\#(\wtd\yk_n^\perp)\in\scr O(\wtd{\mc B})$ for each $n\in\Nbb$, such that for any $w_\blt\in\Wbb_\blt$, we have the following equation of elements of $\scr O(\wtd{\mc B})\{q\}[\log q]$:
	\begin{align*}
	q\partial_q\mc S\uppsi(w_\blt)=\mc S\uppsi(\nabla_{q\partial_q}w_\blt)+\sum_{n\in\Nbb}\#(\wtd\yk_n^\perp)q^n\cdot\mc S\uppsi(w_\blt).
	\end{align*}
\end{pp}

\begin{proof}
	$\#(\wtd\yk_n^\perp)$ is defined by Proposition \ref{lb30}. Moreover, by that proposition, we have
	\begin{align*}
	&w_\blt\otimes \upnu(\wtd\yk_n^\perp)q^{L_0}\btr\otimes~\btl+w_\blt\otimes q^{L_0}\btr\otimes ~\upnu(\wtd\yk_n^\perp)\btl+\upnu(\wtd\yk_n^\perp)w_\blt\otimes q^{L_0}\btr\otimes ~\btl\\
	=&\#(\wtd\yk_n^\perp)\cdot w_\blt\otimes q^{L_0}\btr\otimes ~\btl.
	\end{align*}
	By Lemma \ref{lb43} and relation \eqref{eq87}, it is easy to prove the desired equation.
\end{proof}

To prove Theorem \ref{lb38}, it remains to check that \eqref{eq87} and the projective term $\sum_n\#(\wtd\yk_n^\perp)q^n$ converge a.l.u.. To treat the first one, we choose mutually disjoint neighborhoods $U_1,\dots,U_N$ of $\sgm_1(\wtd{\mc B}),\dots,\sgm_N(\wtd{\mc B})$  on which $\eta_1,\dots,\eta_N$ are defined respectively, and assume they are disjoint from $U',U''$. Then $\mc D_{r\rho}\times U_1,\dots,\mc D_{r\rho}\times U_N$ are neighborhoods of $\sgm_1(\mc B),\dots,\sgm_N(\mc B)$ disjoint from $F',F''$. Write $\tau_\blt\circ\pi$ and $\tau_\blt\circ\wtd\pi$ as $\tau_\blt$ for simplicity. Then $(\eta_i,\tau_\blt)$ is a set of coordinates of $U_i$. Recall \eqref{eq88}. We may write
\begin{align}
\wtd\yk|_{\mc D_{r\rho}\times U_i}=h_i(q,\eta_i,\tau_\blt)\partial_{\eta_i}+q\partial_q
\end{align}
where $h_i(q,\eta_i,\tau_\blt)\in\scr O_{\mc C}(\blt\SX)(\mc D_{r\rho}\times U_i)$. Let $\upnu(\wtd\yk^\perp)$  be a section of $\scr V_{\fk X}\otimes\omega_{\mc C/\mc B}(\blt\SX)$ on $\mc D_{r\rho}\times(U_1\cup \cdots\cup U_N)$ satisfying
\begin{align}
\mc U(\eta_i)\upnu(\wtd\yk^\perp)|_{\mc D_{r\rho}\times U_i}=h_i(q,\eta_i,\tau_\blt) \cbf d\eta_i.\label{eq97}
\end{align}
Write $h_i=\sum_{n}h_{i,n}q^n$. Then by \eqref{eq72},
\begin{align}
\wtd\yk_n^\perp|_{U_i}=h_{i,n}(\eta_i,\tau_\blt)\partial_{\eta_i}.\label{eq89}
\end{align}
So we have $\nabla_{q\partial_q}w_\blt\in\Wbb_\blt\otimes_\Cbb\scr O(\mc D_{r\rho}\times\wtd {\mc B})=\Wbb_\blt\otimes_\Cbb\scr O(\mc B)$ due to the obvious fact:

\begin{lm}\label{lb46}
We have
\begin{align}
\nabla_{q\partial_q}w_\blt=-\upnu(\wtd\yk^\perp)w_\blt.\label{eq96}
\end{align}
\end{lm}

We now prove the convergence of the projective term.

\begin{pp}\label{lb45}
	$\sum_{n\in\Nbb}\#(\wtd\yk_n^\perp)q^n$ is an element of $\scr O(\mc D_{r\rho}\times\wtd{\mc B})=\scr O(\mc B)$.
\end{pp}

\begin{proof}

	Combine \eqref{eq83} and \eqref{eq89}, and apply Proposition \ref{lb30} to the family $\wtd{\fk X}$. We obtain
	\begin{align*}
	\#(\wtd\yk_n^\perp)=\frac{c}{12}\big(A_n+B_n+\sum_{i=1}^NC_{i,n}\big)
	\end{align*}
	where
	\begin{gather*}
	A_n=\sum_{l\geq -n}\Res_{\xi=0}~\Sbf_\xi\fk P\cdot a_{l+n,n}(\tau_\blt)\xi^{l+1}d\xi,\\
	B_n=\sum_{l\geq -n}\Res_{\varpi=0}~\Sbf_\varpi\fk P\cdot b_{n,l+n}(\tau_\blt)\varpi^{l+1}d\varpi,\\
	C_{i,n}=\Res_{\eta_i=0}~\Sbf_{\eta_i}\fk P\cdot h_{i,n}(\eta_i,\tau_\blt)d\eta_i.
	\end{gather*}
	Notice that $\Sbf_{\eta_i}\fk P=\Sbf_{\eta_i}\fk P(\eta_i,\tau_\blt)$, $\Sbf_\varpi\fk P=\Sbf_\varpi\fk P(\varpi,\tau_\blt)$,  $\Sbf_\xi\fk P=\Sbf_\xi\fk P(\xi,\tau_\blt)$ are  holomorphic functions on $U_i,U',U''$ which are identified with their images under $(\eta_i,\tau_\blt),(\xi,\tau_\blt),(\varpi,\tau_\blt)$   respectively.
	
	We have
	\begin{align}
	\sum_{n\geq 0} A_nq^n=\sum_{n\geq 0}\sum_{l\geq -n}\Res_{\xi=0}~\Sbf_\xi\fk P\cdot a_{l+n,n}(\tau_\blt)\xi^{l+1}q^nd\xi.\label{eq90}
	\end{align}
	We claim that \eqref{eq90} is an element of $\scr O(\mc D_{r\rho}\times\wtd{\mc B})$.  Note that $a(\xi,q/\xi,\tau_\blt)$ is defined when $|q|/\rho<|\xi|<r$. Choose any $\epsilon\in (0,r\rho)$. Choose a circle $\gamma'$ surrounding $\mc D_{\epsilon/\rho}$ and inside $\mc D_r$. Then, when $\xi$ is on $\gamma$, $a(\xi,q/\xi,\tau_\blt)$ can be defined whenever $|q|<\epsilon$. Thus, 
	\begin{align*}
	A:=\oint_{\gamma'}\Sbf_\xi\fk P(\xi,\tau_\blt)\cdot a(\xi,q/\xi,\tau_\blt)\xi d\xi
	\end{align*}
	is a holomorphic function defined whenever  $|q|<\epsilon$. Recall the first equation of \eqref{eq82}, and note that the series converges absolutely and uniformly when $\xi\in\gamma'$ and $|q|\leq \epsilon$, by the double Laurent series expansion of $a(\xi,q/\xi,\tau_\blt)$. Therefore, the above contour integral equals
	\begin{align*}
	\sum_{n\geq 0}\sum_{l\geq -n}\oint_{\gamma'}\Sbf_\xi\fk P(\xi,\tau_\blt)\cdot a_{l+n,n}(\tau_\blt)\xi^{l+1}q^nd\xi,
	\end{align*}
	which clearly equals \eqref{eq90} as an element of $\scr O(\wtd{\mc B})[[q]]$. Thus \eqref{eq90} is an element of $\scr O(\mc D_\epsilon\times\wtd{\mc B})$ whenever $\epsilon<r\rho$, and hence when $\epsilon=r\rho$.
	
	A similar argument shows  $\sum B_nq^n$ converges a.l.u. to
	\begin{align*}
	B:=\oint_{\gamma''}\Sbf_\varpi\fk P(\varpi,\tau_\blt)\cdot b(q/\varpi,\varpi,\tau_\blt)\varpi d\varpi
	\end{align*}
	where $\gamma''$ is any circle in $\mc D_\rho$ surrounding $0$. Finally, we compute
	\begin{align*}
	C_i:=&\sum_{n\geq 0}C_{i,n}q^n=\sum_{n\geq 0}\Res_{\eta_i=0}~\Sbf_{\eta_i}\fk P(\eta_i,\tau_\blt)\cdot h_{i,n}(\eta_i,\tau_\blt)q^nd\eta_i\\
	=&\Res_{\eta_i=0}~\Sbf_{\eta_i}\fk P(\eta_i,\tau_\blt)\cdot h_i(q,\eta_i,\tau_\blt)d\eta_i
	\end{align*}
	which is clearly inside $\scr O(\mc D_{r\rho}\times\wtd{\mc B})$. The proof is now complete. We summarize that the projective term equals
	\begin{align}
	\sum_{n\in\Nbb}\#(\wtd\yk_n^\perp)q^n=\frac c{12}\big(A+B+\sum_{i=1}^N C_i\big).\label{eq91}
	\end{align}
\end{proof}

We can now finish the
\begin{proof}[Proof of Theorem \ref{lb38}]
	
By Theorem \ref{lb21}, we may find $k_0\in\Nbb$ such that the vectors in $\Wbb_\blt^{\leq k_0}$, considered as constant sections of 	$\Wbb_\blt^{\leq k_0}\otimes_\Cbb\scr O(\mc B)$, generate the $\scr O(\mc B)$-module $\scr W_{\fk X}(\Wbb_\blt)(\mc B)/\scr J(\mc B)$. Fix any $k\geq k_0$. We choose a basis $s_1,s_2,\dots$  of $\Wbb_\blt^{\leq k}$. 
	
	By propositions \ref{lb44} and \ref{lb45},  for each $s_i$ of $s_1,s_2,\dots$,  we have the following equation of elements of $\scr O(\wtd{\mc B})\{q\}[\log q]$:
	\begin{align*}
	q\partial_q\mc S\uppsi(s_i)=\mc S\uppsi(\nabla_{q\partial_q}s_i)+g\mc S\uppsi(s_i)
	\end{align*}
	where $g\in\scr O(\mc D_{r\rho}\times\wtd{\mc B})=\scr O(\mc B)$ equals \eqref{eq91}. Since we have $\nabla_{q\partial_q}s_i\in\Wbb_\blt\otimes_\Cbb\scr O(\mc B)=\scr W_{\fk X}(\Wbb_\blt)(\mc B)$ due to Lemma \ref{lb46},  we can find $f_{i,j}\in\scr O(\mc B)$ such that $\nabla_{q\partial_q}s_i$ equals $\sum_j f_{i,j}s_j$ mod  elements  of $\scr J(\mc B)$. Since, by Theorem \ref{lb34}, $\mc S\uppsi$ vanish on $\scr J(\mc B)$, we must have
	\begin{align*}
	q\partial_q\mc S\uppsi(s_i)=\sum_j f_{i,j}\mc S\uppsi(s_j)+g\mc S\uppsi(s_i).
	\end{align*}
	The proof is  completed by setting the matrix-valued holomorphic function $\Upomega$ to be $(f_{i,j}+g\delta_{i,j})_{i,j}$.
\end{proof}

\begin{rem}\label{lb53}
Following \cite{TUY89} or \cite[Chapter 7]{BK01} or \cite{DGT19a}, one can define locally a (logarithmic) connection $\nabla$ on $\scr W_{\fk X}(\Wbb_\blt)$ as follows. Assume first of all that $\fk X$ admits local coordinates $\eta_1,\dots,\eta_N$, and $\mc B$ is a small enough Stein manifold such that $\Theta_{\mc B}(-\log\Delta)$ is $\scr O_{\mc B}$-generated freely by finitely many global sections $\fk y_1,\fk y_2,\dots$. Then each $\fk y_j$ lifts to an element $\wtd{\yk }_j\in H^0(\mc C,\Theta_{\mc C}(-\log\mc C_\Delta+\blt\SX))$, i.e. $d\pi(\wtd\yk_j)=\fk y_j$. As in \eqref{eq97}, one can define a section $\wtd\yk_j^\perp$ of $\Theta_{\mc C/\mc B}(\blt\SX)$ (near $\sgm_1(\mc B),\dots,\sgm_N(\mc B)$) to be the vertical part of $\wtd\yk_j$, and define $\upnu(\wtd\yk_j^\perp)$ as in Section \ref{lb41} using the local coordinates $\eta_\blt$.   Then $\nabla$ is determined by
\begin{align*}
\nabla_{\yk_j}w_\blt=-\upnu(\wtd\yk_j^\perp)w_\blt
\end{align*}
for any constant section $w_\blt\in\Wbb_\blt$. As argued in \cite{DGT19a}, $\nabla$ preserves $\scr J(\mc B)$. So $\nabla$ descends to a connection of $\scr T_{\fk X}(\Wbb_\blt)$. Its dual  connection on $\scr T_{\fk X}^*(\Wbb_\blt)$ is also denoted by $\nabla$. Thus, Proposition \ref{lb44} says that for the $\fk X$ obtained by sewing $\wtd{\fk X}$, and for any  $\uppsi\in\scr T_{\wtd{\fk X}}^*(\Wbb_\blt\otimes\Mbb\otimes\Mbb')(\wtd{\mc B})$,  its sewing $\mc S\uppsi\in\scr T_{\fk X}^*(\Wbb_\blt)(\mc B)$ satisfies
\begin{align}
\nabla_{q\partial_q}\mc S\uppsi=g\mc S\uppsi\label{eq98}
\end{align}
where $g\in\scr O(\mc B)$ is given by \eqref{eq91}. This observation is the analytic analog of \cite[Rem. 8.5.2]{DGT19b}.
\end{rem}

\section{Injectivity of sewing for semisimple modules}

We continue the study of sewing, but assume that $\fk X$ is formed by sewing a single $(N+2)$-pointed compact Riemann surface with local coordinates
\begin{align*}
\wtd{\fk X}=(\wtd C;x_1,\dots,x_N,x',x'';\eta_1,\dots,\eta_N,\xi,\varpi).
\end{align*}
Namely, we assume $\wtd{\mc B}$ is a single point. As usual, each connected component of $\wtd C$ contains one of $x_1,\dots,x_N$. So
\begin{align*}
\fk X=(\pi:\mc C\rightarrow\mc D_{r\rho};x_1,\dots,x_N;\eta_1,\dots,\eta_N)
\end{align*}
where $x_1,\dots,x_N,\eta_1,\dots,\eta_N$ are extended from those of $\wtd C$ and are constant over $\mc D_{r\rho}$.

Assume $\Vbb$ is $C_2$-cofinite. Let $\mc E$ be a complete list of mutually inequivalent irreducible (ordinary)   $\Vbb$-modules. ``Complete" means that any simple ordinary $\Vbb$-module is equivalent to an object of $\mc E$. $\Wbb_1,\dots,\Wbb_N$ are finitely-generated $\Vbb$-modules. Then by Theorems \ref{lb47} and \ref{lb48}, for each $q\in\mc D_{r\rho}^\times$, we can define a \index{Sq@$\fk S_q,\wtd{\fk S}_q$} linear map
\begin{gather}
\fk S_q:\bigoplus_{\Mbb\in\mc E}\scr T_{\wtd{\fk X}}^*(\Wbb_\blt\otimes\Mbb\otimes\Mbb')\rightarrow\scr T_{\fk X_q}^*(\Wbb_\blt),\label{eq92}\\
\bigoplus_\Mbb\uppsi_\Mbb\mapsto \sum_\Mbb\mc S\uppsi_\Mbb(q)\nonumber
\end{gather}
where $\fk X_q=(\mc C_q;x_1,\dots,x_N;\eta_1,\dots,\eta_N)$. Similarly, one can define $\wtd{\fk S}_q$ by replacing $\mc S$ with $\wtd{\mc S}$. Notice that $\sum_\Mbb\wtd{\mc S}\uppsi_\Mbb(q)=\sum_\Mbb q^{\lambda_\Mbb}\mc S\uppsi_\Mbb(q)$ for some constants $\lambda_\Mbb$ depending only on $\Mbb$. Thus $\fk S_q$ is injective (resp. bijective) if and only if $\wtd{\fk S}_q$ is. Also, $\fk S_q$ depends on the argument of $q$.

\begin{thm}\label{lb49}
Assume $\Vbb$ is $C_2$-cofinite, and choose any $q\in\mc D_{r\rho}^\times$. Then $\fk S_q$ and $\wtd{\fk S}_q$ are injective linear maps. If $\Vbb$ is also rational and $\Vbb(0)=\Cbb\id$, then $\fk S_q$ and $\wtd{\fk S}_q$ are bijective.
\end{thm}

\begin{proof}
	Let us fix $q_0\in\mc D_{r\rho}^\times$ and let $q$ denote a complex variable. Let us prove that $\fk S_{q_0}$ is injective. Suppose that the finite sum $\sum_\Mbb\mc S\uppsi_\Mbb(q_0)$ equals $0$. We shall prove by contradiction that $\uppsi_\Mbb=0$ for any $\Mbb\in\mc E$. 
	
	Suppose this is not true. Let $\mc F$ be the (finite) subset of all $\Mbb\in\mc E$ satisfying $\uppsi_\Mbb\neq0$. Then $\mc F$ is not an empty set. 	We first show that $\upphi:=\sum_\Mbb\mc S\uppsi_\Mbb$ (which is a multivalued holomorphic function on $\mc D_{r\rho}^\times$) satisfies $\upphi(q)=0$ for each $q\in\mc D_{r\rho}^\times$. Choose any large enough  $k\in\Nbb$. Then, by Theorem \ref{lb38},  $\upphi^{\leq k}$ satisfies a linear differential equation on $\mc D_{r\rho}^\times$ of the form $\partial_q\upphi^{\leq k}=q^{-1}\Upomega\cdot \upphi^{\leq k}$. Moreover, it satisfies the initial condition $\upphi^{\leq k}(q_0)=0$. Thus, $\upphi^{\leq k}$ is constantly $0$. So is $\upphi$.
	
	Consider the $\Vbb\otimes\Vbb$-module $\mbb X:=\bigoplus_{\Mbb\in\mc F}\Mbb\otimes\Mbb'$. Define a linear map $\kappa:\mbb X\rightarrow \Wbb_\blt^*$ as follows. If $m\otimes m'\in\Mbb\otimes\Mbb'$, then the evaluation of $\kappa(m\otimes m')$ with any $w_\blt\in\Wbb_\blt$ is
	\begin{align*}
	\bk{\kappa(m\otimes m'),w_\blt}=\sum_{\Mbb\in\mc F}\uppsi_\Mbb(w_\blt\otimes m\otimes m').
	\end{align*}
	We claim that $\Ker(\kappa)$ is a non-zero subspace of $\mbb X$ invariant under the action of $\Vbb\otimes \Vbb$. If this can be proved, then $\ker(\kappa)$ is a semi-simple $\Vbb\otimes\Vbb$-submodule of $\mbb X$, which must contain $\Mbb\otimes\Mbb'$ for some $\Mbb\in\mc F$. Therefore, $\uppsi_\Mbb(w_\blt\otimes m\otimes m')=0$ for any $w_\blt\in\Wbb_\blt$ and $m\otimes m'\in\Mbb\otimes\Mbb'$. Namely, $\uppsi_\Mbb=0$. So $\Mbb\notin\mc F$, which gives a contradiction. 
	
	For any $n\in\Cbb$, let $P_{(n)}$ be the projection of $\Mbb$ onto its $L_0$-weight $n$ subspace. Then 
	\begin{align*}
	\upphi(w_\blt)(q)=\sum_{\Mbb\in\mc F}\sum_{n\in\Cbb} \uppsi_\Mbb(w_\blt\otimes P_{(n)}\btr\otimes~\btl)q^n.
	\end{align*}
	Since this multivalued function is always $0$, by \cite[Prop. 2.1]{Hua17}, any coefficient before $q^n$ is $0$. Thus $P_{(n)}\btr\otimes~\btl$ (which is an element of $\Mbb_{(n)}\otimes\Mbb_{(n)}'$) is in $\ker(\kappa)$ for any $n$. Thus $\ker(\kappa)$ must be non-empty.
	
	Suppose now that $\sum_j m_j\otimes m'_j\in\Ker(\kappa)$ where each $m_j\otimes m'_j$ belongs to some $\Mbb\otimes\Mbb'$. We set $\uppsi_{\wtd\Mbb}(w_\blt\otimes m_j\otimes m_j')=0$ if $\Mbb,\wtd\Mbb\in\mc F$ and $\Mbb\neq \wtd\Mbb$. Choose any $n\in\Nbb$ and $l\in\Zbb$.  We shall show that $\sum_j Y(u)_l m_j\otimes m_j'\in\Ker(\kappa)$ for any $u\in\Vbb^{\leq n}$. (Here $Y$ denotes $Y_\Mbb$ for a suitable $\Mbb$.) Thus $\Ker(\kappa)$  is $\Vbb\otimes\id$-invariant. A similar argument will show that $\Ker(\kappa)$ is $\id\otimes\Vbb$-invariant, and hence $\Vbb\otimes\Vbb$-invariant.
	
	Set divisors $D_1=x_1+\cdots+x_N$ and $D_2=x'+x''$. Choose a natural number $k_2\geq l$ such that $Y(u)_km_j=Y(u)_km_j'=0$ for any $j$, any $k\geq k_2$,  and any $u\in\Vbb^{\leq n}$. This is possible by the lower truncation property. By Serre's vanishing Theorem, we can find $k_1\in\Nbb$ such that $H^1(\wtd C,\scr V_{\wtd C}^{\leq n}\otimes\omega_{\wtd C}(k_1D_1-k_2D_2))=0$. Thus, the short exact sequence
	\begin{align*}
	0\rightarrow \scr V_{\wtd C}^{\leq n}\otimes\omega_{\wtd C}(k_1D_1-k_2D_2) \rightarrow \scr V_{\wtd C}^{\leq n}\otimes\omega_{\wtd C}(k_1D_1-lD_2)\rightarrow\scr G\rightarrow 0
	\end{align*}
	(where $\scr G$ is the quotient of the previous two sheaves) induces another one
	\begin{align*}
	&0\rightarrow H^0\big(\wtd C,\scr V_{\wtd C}^{\leq n}\otimes\omega_{\wtd C}(k_1D_1-k_2D_2)\big) \rightarrow H^0\big(\wtd C,\scr V_{\wtd C}^{\leq n}\otimes\omega_{\wtd C}(k_1D_1-lD_2)\big)\\
	&\rightarrow H^0(\wtd C,\scr G)\rightarrow 0.
	\end{align*}
	Choose any $u\in\Vbb^{\leq n}$. Choose $v\in H^0(\wtd C,\scr G)$ to be $\mc U_\varrho(\xi)^{-1}u\xi^ld\xi$ near $x'$ and $0$ in $\wtd C-\{x'\}$. Then $v$ has a lift $\nu$ in $H^0\big(\wtd C,\scr V_{\wtd C}^{\leq n}\otimes\omega_{\wtd C}(k_1D_1-lD_2)\big)$, which must be of the form
	\begin{gather*}
	\mc U_\varrho(\xi)\nu|_{U'}=u\xi^ld\xi+\xi^{k_2}(\mathrm{elements~of~}\Vbb^{\leq n}\otimes_\Cbb\scr O(U'))d\xi,\\
	\mc U_\varrho(\varpi)\nu|_{U''}=\varpi^{k_2}(\mathrm{elements~of~}\Vbb^{\leq n}\otimes_\Cbb\scr O(U''))d\varpi,
	\end{gather*}
where $U'\ni x',U''\ni x''$ are open subsets of $\wtd C$ (see \eqref{eq84}). 	It is clear that $\nu\cdot m_j=Y(u)_lm_j$ and $\nu\cdot m_j'=0$. Thus, as each $\uppsi_\Mbb$ vanishes on $\nu\cdot(\Wbb_\blt\otimes\Mbb\otimes\Mbb')$, we have
	\begin{align*}
	&\sum_{\Mbb\in\mc F}\sum_j\uppsi_\Mbb(w_\blt\otimes Y(u)_lm_j\otimes m_j')=-\sum_{\Mbb\in\mc F}\sum_j\uppsi_\Mbb((\nu\cdot w_\blt)\otimes m_j\otimes m_j')\\
	=&-\sum_j\bk{\kappa(m_j\otimes m_j'),\nu\cdot w_\blt}=0.
	\end{align*}
	So $\sum_j Y(u)_l m_j\otimes m_j'\in\Ker(\kappa)$.
	
We have proved the injectivity when $\Vbb$ is $C_2$-cofinite and $\Vbb(0)=\Cbb\id$. If $\Vbb$ is also rational, the surjectivity follows by comparing the dimensions of both sides and using the factorization property proved in \cite{DGT19b}.
\end{proof}

As a consequence of this theorem, we obtain again the well-known fact \cite{DLM00} that any $C_2$-cofinite $\Vbb$ with $\Vbb(0)=\Cbb\id$ has finitely many equivalence classes of simple modules. Indeed, we let $\wtd{\fk X}=(\Pbb^1;1,0,\infty;z-1,z,z^{-1})$, and let $\Wbb_1=\Vbb$. Then $\scr T_{\wtd{\fk X}}^*(\Vbb\otimes\Mbb\otimes\Mbb')$ is nontrivial since $Y_\Mbb$ defines a non-zero element of it. Thus, as the range of $\fk S_q$ is finite-dimensional, $\mc E$ is finite.

\section{Convergence of multiple sewing}\label{lb56}

In this section, we prove the convergence of sewing conformal blocks along several pairs of points, which generalizes Theorem \ref{lb48}. Let $N,M\in\Zbb_+$. Let
\begin{align*}
\wtd{\fk X}=(\wtd\pi:\wtd{\mc C}\rightarrow\wtd{\mc B};\sgm_1,\dots,\sgm_N;\sgm_1',\dots,\sgm_M';\sgm_1'',\dots,\sgm_M'';\eta_1,\dots,\eta_N;\xi_1,\dots,\xi_M;\varpi_1,\dots,\varpi_M)
\end{align*}
be a family of $(N+2M)$-pointed compact Riemann surfaces with local coordinates. So each $\eta_i,\xi_j,\varpi_j$ are local coordinates at $\sgm_i(\wtd{\mc B}),\sgm_j'(\wtd{\mc B}),\sgm_j''(\wtd{\mc B})$ respectively. We assume that for every $b\in\wtd{\mc B}$, each connected component of the fiber  $\wtd{\mc C}_b=\wtd\pi^{-1}(b)$ contains one of $\sgm_1(b),\dots,\sgm_N(b)$.

For each $1\leq j\leq M$ we choose $r_j,\rho_j>0$ and a neighborhood $U_j'$ (resp. $U_j''$) of $\sgm_j'(\wtd {\mc B})$ (resp. $\sgm_j''(\wtd {\mc B})$) such that
\begin{gather}
(\xi_j,\wtd\pi):U_j'\xrightarrow{\simeq} \mc D_{r_j}\times\wtd{\mc B}\qquad\text{resp.}\qquad (\varpi_j,\wtd\pi):U_j''\xrightarrow{\simeq} \mc D_{\rho_j}\times\wtd{\mc B}\label{eq94}
\end{gather}
is a biholomorphic map. We also assume that these $r_i$ and $\rho_j$ are small enough such that the neighborhoods $U_1',\dots,U_M',U_1'',\dots,U_M''$ are mutually disjoint and are also disjoint from $\sgm_1(\wtd{\mc B}),\dots,\sgm_N(\wtd{\mc B})$.

Assume $\Vbb$ is $C_2$-cofinite. Associate to $\sgm_1,\dots,\sgm_N$ finitely-generated (equivalently, finitely $\Lss$-semisimple) $\Vbb$-modules $\Wbb_1,\dots,\Wbb_N$. Associate to $\sgm_1',\dots,\sgm_M'$ finitely-generated  $\Vbb$-modules $\Mbb_1,\dots,\Mbb_M$ whose contragredient modules $\Mbb_1',\dots,\Mbb_M'$ are associated to $\sgm_1'',\dots,\sgm_M''$. Each $\Mbb_j'$ is finitely $\Lss$-semisimple and hence, equivalently, finitely generated.  We understand $\Wbb_\blt\otimes\Mbb_\blt\otimes\Mbb_\blt'$ as
\begin{align*}
\Wbb_1\otimes\cdots\otimes \Wbb_N\otimes\Mbb_1\otimes\Mbb_1'\otimes\cdots\otimes\Mbb_M\otimes\Mbb_M',
\end{align*}
where the order has be changed so that each $\Mbb_j'$ is next to $\Mbb_j$. 

For any $\uppsi\in\scr T_{\wtd{\fk X}}^*(\Wbb_\blt\otimes\Mbb_\blt\otimes\Mbb_\blt')(\wtd{\mc B})$ and $w\in\Wbb_\blt$, \index{S@$\wtd{\mc S}\uppsi,\mc S\uppsi$} we define   
\begin{align}
\mc S\uppsi(w)=\uppsi\Big(w\otimes (q_1^{ L_0}\btr\otimes_1\btl)\otimes\cdots\otimes (q_M^{ L_0}\btr\otimes_M\btl)\Big)\qquad \in\scr O(\wtd{\mc B})\{q_\blt\}[\log q_\blt].
\end{align}
(Recall that the powers the $\log q_j$ are bounded above since each $\Mbb_j$ and $\Mbb_j'$ are finitely-generated; see the paragraph above Thm. \ref{lb38}.) Here, each $q_j^{L_0}\btr\otimes_j\btl\in(\Mbb_j\otimes\Mbb_j')\{q_j\}[\log q_j]$ is defined as in \eqref{eq109}. $\wtd{\mc S}\uppsi(w)$ is defined similarly, except that $L_0$ is replaced by $\wtd L_0$.

Let $\mc D_{r_\blt\rho_\blt}=\mc D_{r_1\rho_1}\times\cdots\times\mc D_{r_M\rho_M}$ and $\mc D_{r_\blt\rho_\blt}^\times=\mc D_{r_1\rho_1}^\times\times\cdots\times\mc D_{r_M\rho_M}^\times$. Let $q_j$ be the standard coordinate of $\mc D_{r_j\rho_j}$. We say that $\wtd{\mc S}\uppsi$ \textbf{converges a.l.u.} if $\wtd{\mc S}\uppsi(w)\in\scr O(\wtd{\mc B}\times\mc D_{r_\blt\rho_\blt})$ for each $w\in\Wbb_\blt$. Write
\begin{align*}
\mc S\uppsi(w)=\sum_{n_\blt\in\Cbb^M,l_\blt\in\Nbb^M}\mc S\uppsi(w)_{n_\blt,l_\blt}q_\blt^{n_\blt}(\log q_\blt)^{l_\blt}	
\end{align*}
where 
\begin{align*}
q_\blt^{n_\blt}=q_1^{n_1}\cdots q_M^{n_M},\qquad (\log q_\blt)^{l_\blt}=(\log q_1)^{l_1}\cdots(\log q_M)^{l_M}.	
\end{align*}
We say $\mc S\uppsi$ \textbf{converges a.l.u.} if for each compact subsets $K\subset\wtd{\mc B}$ and $Q\subset\mc D_{r_\blt\rho_\blt}^\times=\mc D_{r_1\rho_1}^\times\times\cdots\times\mc D_{r_M\rho_M}^\times$, there exists $C>0$ such that
\begin{align*}
\sum_{n_\blt\in\Cbb^M}\big|\mc S\uppsi(w)_{n_\blt,l_\blt}(b)\big|\cdot |q_\blt^{n_\blt}|	\leq C
\end{align*}
for any $b\in K,q_\blt=(q_1,\dots,q_M)\in Q$, and $l_\blt\in\Nbb^M$.
Clearly, if $\mc S\uppsi$ converges a.l.u. then so does $\wtd{\mc S}\uppsi$.

\begin{thm}\label{lb55}
Assume that $\Vbb$ is $C_2$-cofinite and all the $\Vbb$-modules $\Wbb_i,\Mbb_j$ are finitely-generated. Then $\mc S\uppsi$ and $\wtd{\mc S}\uppsi$ converge a.l.u..
\end{thm}

\begin{proof}
The main idea is the same as single sewing. We sketch the proof below.

Since the convergence property is local with respect to $\wtd{\mc B}$, we may identify $\wtd{\mc B}$ as an open subset of $\Cbb^m$ and let $\tau_\blt=(\tau_1,\dots,\tau_m)$ denote the set of  $m$ standard  coordinates.


Similar to Sec. \ref{lb6}, we may define a new family of curves $\fk X$ with base manifold $\mc B=\mc D_{r_\blt\rho_\blt}\times\wtd{\mc B}=\mc D_{r_1\rho_1}\times\cdots\times\mc D_{r_M\rho_M}\times\wtd{\mc B}$ to be the simultaneous sewing of $\wtd{\fk X}$ along $\sgm_j'$ and $\sgm_j''$ using the relation $\xi_j\varpi_j=q_j$ for all $1\leq j\leq M$. Let $q_j$ be the standard coordinate of $\mc D_{r_j\rho_j}$. We let $\Sigma\subset\mc C$ be the set of critical points of $\pi:\mc C\rightarrow\mc B$, i.e., the set of all points at which $\pi$ is not submersive. Then $\Delta:=\pi(\Sigma)$ equals $\Delta_1\cup\cdots\cup\Delta_M$ where $\Delta_j$ is the hypersurface of $\mc B$ defined by $q_j=0$. We let $\mc C_\Delta=\pi^{-1}(\Delta)$. $\Theta_{\mc B}(-\log\Delta)$ is defined to be the free $\scr O_{\mc B}$-module generated by
\begin{align*}
	q_1\partial_{q_1},\dots,q_M\partial_{q_M},\partial_{\tau_1},\dots,\partial_{\tau_m}.
\end{align*}
The partial coordinates $\tau_\blt\circ\wtd\pi$ of $\wtd{\mc C}$ and $\tau_\blt\circ\pi$ of $\mc C$ are also denoted by $\tau_\blt$ for simplicity.

To understand $\pi$ locally at $x\in\mc C$, we consider the following two cases:
\begin{enumerate}
\item $x\notin\Sigma$. The fiber $\mc C_{\pi(x)}$ is either smooth or nodal, and $x$ is a smooth point of $\mc C_{\pi(x)}$. Then we may find $r_0>0$ and $r_\blt'<r_\blt,\rho_\blt'<\rho_\blt$ such that $\pi$ on a neighborhood of $x$ is equivalent to part of the projection
\begin{align*}
\mc D_{r_0}\times\mc D_{r_\blt'\rho_\blt'}\times\wtd{\mc B}	\rightarrow\mc D_{r_\blt'\rho_\blt'}\times\wtd{\mc B}.
\end{align*}
The $\scr O_{\mc C}$-module $\Theta_{\mc C}(-\log\mc C_\Delta)$ near $x$ is generated freely by
\begin{align*}
\partial_z,q_1\partial_{q_1},\dots,q_M\partial_{q_M},\partial_{\tau_1},\dots,\partial_{\tau_m}	
\end{align*}
where $z$ is the standard coordinate of $\mc D_{r_0}$. The morphism
\begin{align}
	d\pi:\Theta_{\mc C}(-\log\mc C_\Delta)\rightarrow \pi^*\Theta_{\mc B}(-\log\Delta)\label{eq104}
\end{align}
is defined by killing $\partial_z$, and by keeping all the other generating elements.

\item $x\in\Sigma$. The fiber $\mc C_{\pi(x)}$ is nodal, and $x$ is a node of $\mc C_{\pi(x)}$. Denote by $\mc D_{r_\blt\rho_\blt\setminus j}$ the product of $\mc D_{r_1\rho_1},\dots,\mc D_{r_M\rho_M}$ except $\mc D_{r_j\rho_j}$.  Then $\pi$ at $x$ is locally equivalent to the map
\begin{gather*}
\mc D_{r_j}\times\mc D_{\rho_j}\times\mc D_{r_\blt\rho_\blt\setminus j}\times\wtd{\mc B}\rightarrow \mc D_{r_j\rho_j}\times\mc D_{r_\blt\rho_\blt\setminus j}\times\wtd{\mc B}=\mc B,\\
(\xi_j,\varpi_j,q_{\blt\setminus j},\tau_\blt)\mapsto (\xi_j\varpi_j,q_{\blt\setminus j},\tau_\blt).
\end{gather*}
Near $x$, the $\scr O_{\mc C}$-module $\Theta_C(-\log\Delta)$ is generated freely by
\begin{align*}
\xi_j\partial_{\xi_j},\varpi_j\partial_{\varpi_j},q_1\partial_{q_1},\dots,q_{j-1}\partial_{q_{j-1}},q_{j+1}\partial_{q_{j+1}},\dots,q_M\partial_{q_M},\tau_1,\dots,\tau_m.
\end{align*}
\eqref{eq104} is defined near $x$ by setting
\begin{align}
d\pi(\xi_j\partial_{\xi_j})=d\pi(\varpi_j\partial_{\varpi_j})=q_j\partial_{q_j}\label{eq106}
\end{align}
and keeping the other generators.
\end{enumerate}

We now obtain again an exact sequence \eqref{eq18} where $\Theta_{\mc C/\mc B}$ is defined to be the kernal of $d\pi$, and we again have a natural isomorphism $\Theta_{\mc C/\mc B}|_{\mc C_b}\simeq\Theta_{\mc C_b}$ for each fiber $\mc C_b$ of $\fk X$. So we obtain the long exact sequence \eqref{eq79}. Similar to the proof of Thm. \ref{lb38}, one may lift $q_k\partial_{q_k}\in\Theta_{\mc B}(-\log\Delta)$ to a section
\begin{align*}
\wtd\yk\in H^0(\mc C,\Theta_{\mc C}(-\log\mc C_\Delta+\blt\SX))	
\end{align*}
which has multi series expansion
\begin{align*}
\wtd\yk=\sum_{n_\blt\in\Nbb^M}\wtd\yk_{n_\blt}^\perp \cdot q_\blt^{n_\blt}+q_k\partial_{q_k}
\end{align*}
on $\mc C\setminus\Sigma$. Here, $\partial_{q_k}$ should be understood as the tangent field $\partial_{q_k\circ\pi}$ of $\mc C$ which vanishes when pullbacked to each fiber of $\mc C$. And $\wtd\yk_{n_\blt}^\perp\in H^0(\wtd{\mc C}-\Gamma,\Theta_{\wtd{\mc C}/\wtd{\mc B}})$ where $\Gamma=\bigcup_{1\leq j\leq M}(\sgm_j'(\wtd{\mc B})\cup\sgm_j''(\wtd{\mc B}))$.

Choose $\wtd V_1,\dots,\wtd V_N$ to be neighborhoods of $\sgm_1(\wtd{\mc B}),\dots,\sgm_N(\wtd{\mc B})$. Assume that they are disjoint from $U'_1,\dots,U'_M,U''_1,\dots,U''_M$.  Then $\mc D_{r_\blt\rho_\blt}\times\wtd V_1,\dots,\mc D_{r_\blt\rho_\blt}\times\wtd V_N$ are neighborhoods of $\sgm_1(\mc B),\dots,\sgm_N(\mc B)$ in $\mc C$. By the descriptions of $d\pi$ in \eqref{eq104} and \eqref{eq106}, for each $1\leq i\leq N,1\leq j\leq M$, we can write
\begin{gather*}
	\wtd\yk|_{\mc D_{r_\blt\rho_\blt}\times\wtd V_i}=h_i(q_\blt,\eta_i,\tau_\blt)\partial_{\eta_i}+q_k\partial_{q_k},\\
	\wtd\yk|_{\mc D_{r_j}\times\mc D_{\rho_j}\times\mc D_{r_\blt\rho_\blt\setminus j}\times\wtd{\mc B}}=a_j(\xi_j,\varpi_j,q_{\blt\setminus j},\tau_\blt)\xi_j\partial_{\xi_j}+b_j(\xi_j,\varpi_j,q_{\blt\setminus j},\tau_\blt)\varpi_j\partial_{\varpi_j}+(1-\delta_{j,k})q_k\partial_{q_k}.
\end{gather*}
where $h_i,a_j,b_j$ are holomorphic functions on suitable domains, and
\begin{align*}
	a_j+b_j=\delta_{j,k}.
\end{align*}
This shows in particular that $\wtd\yk_{n_\blt}^\perp\in H^0(\wtd{\mc C},\Theta_{\wtd{\mc C}/\wtd{\mc B}}(\blt\SX))$. Fix a projective structure $\fk P$ on $\wtd{\fk X}$.  Then we can define $\upnu(\wtd\yk_{n_\blt}^\perp)$ and $\#(\wtd\yk_{n_\blt}^\perp)$  as in Sec. \ref{lb41}.

For $q_j^{L_0}\btr\otimes_j\btl\in\Mbb_j\otimes\Mbb_j'\{q_j\}[\log q_j]$, \eqref{eq86} can be  generalized to
	\begin{align}
\delta_{j,k}\cdot L_0q_k^{ L_0}\btr\otimes_k~\btl=\sum_{n_\blt\in\Nbb^M}\upnu(\wtd\yk_{n_\blt}^\perp)q_\blt^{n_\blt}q_j^{L_0}\btr\otimes_j~\btl+\sum_{n_\blt\in\Nbb^M}q_\blt^{n_\blt}q_j^{L_0}\btr\otimes_j ~\upnu(\wtd\yk_{n_\blt}^\perp)\btl
\end{align}
by showing that the right hand side equals
\begin{align*}
&\Res_{\xi_j=0}~Y_{\Mbb_j}(\xi_j^{L_0}\cbf,\xi_j) q_j^{ L_0}\btr\otimes_j~\btl\cdot a_j(\xi_j,q_j/\xi_j,q_{\blt\setminus j},\tau_\blt)\frac{d\xi_j}{\xi_j}\\
+&\Res_{\varpi_j=0}~q_j^{ L_0}\btr\otimes_j~Y_{\Mbb'_j}(\varpi_j^{L_0}\mc U(\upgamma_1)\cbf,\varpi_j)\btl\cdot b_j(q_j/\varpi_j,\varpi_j,q_{\blt\setminus j},\tau_\blt)\frac{d\varpi_j}{\varpi_j}
\end{align*}
and using Lemma \ref{lb32}. Then, similar to Prop. \ref{lb44}, we can prove
\begin{align}
q_k\partial_{q_k}\mc S\uppsi(w_\blt)=\mc S\uppsi(\nabla_{q_k\partial_{q_k}}w_\blt)+\sum_{n_\blt\in\Nbb^M}\#(\wtd\yk_{n_\blt}^\perp)q_\blt^{n_\blt}\cdot\mc S\uppsi(w_\blt)	
\end{align}
where $\nabla_{q_k\partial_{q_k}}w_\blt=-\sum_{n_\blt\in\Nbb^M}q_\blt^{n_\blt}\upnu(\wtd\yk_{n_\blt}^\perp)w_\blt$ is indeed an element of $\Wbb_\blt\otimes\scr O(\mc B)$.

Consider holomorphic functions $\Sbf_{\eta_i}\fk P=\Sbf_{\eta_i}\fk P(\eta_i,\tau_\blt)$, $\Sbf_{\xi_j}\fk P=\Sbf_{\xi_j}\fk P(\xi_j,\tau_\blt)$, $\Sbf_{\varpi_j}\fk P=\Sbf_{\varpi_j}\fk P(\varpi_j,\tau_\blt)$   on $\wtd V_i,U_j',U''_j$ which are identified with their images under $(\eta_i,\tau_\blt),(\xi_j,\tau_\blt),(\varpi_j,\tau_\blt)$   respectively. Under these identifications, choose circles $\gamma_j'$ in the $\xi_j$-plane $\mc D_{r_j}$ of $U_j'$ and $\gamma_j''$ in the $\varpi_j$-plane $\mc D_{\rho_j}$ of $U_j''$ around the origins (see \eqref{eq94}). Then the projective term $\sum_{n_\blt\in\Nbb^M}\#(\wtd\yk_{n_\blt}^\perp)q_\blt^{n_\blt}$ equals
\begin{align}\label{eq108}
	f_k=\frac{c}{12}\cdot\bigg(\sum_{1\leq j\leq M}A_j+\sum_{1\leq j\leq M}B_j+\sum_{1\leq i\leq N}C_i\bigg),
\end{align}
where
\begin{align}\label{eq110}	
\left\{	\begin{array}{l}
	A_j=\displaystyle\oint_{\gamma_j'}\Sbf_{\xi_j}\fk P(\xi_j,\tau_\blt)\cdot a_j(\xi_j,q_j/\xi_j,q_{\blt\setminus j},\tau_\blt)\cdot\xi_jd\xi_j,\vspace{1ex}\\ 
B_j=\displaystyle\oint_{\gamma_j''}\Sbf_{\varpi_j}\fk P(\varpi_j,\tau_\blt)\cdot b_j(q_j/\varpi_j,\varpi_j,q_{\blt\setminus j},\tau_\blt)\cdot \varpi_jd\varpi_j,\vspace{1ex}\\
C_i=\Res_{\eta_i=0}~\Sbf_{\eta_i}\fk P(\eta_i,\tau_\blt)\cdot h_i(q_\blt,\eta_i,\tau_\blt)d\eta_i	
\end{array}	
\right.
\end{align}
whenever $|q_j|/\rho_j$ and $|q_j|/r_j$ are respectively less than the radii of $\gamma_j'$ and $\gamma_j''$. This proves that the projective term is  an element of $\scr O(\mc B)$. 

Thus, as in the last step of the proof of Thm. \ref{lb38}, we conclude that for all large enough $l$, there are $\Upomega_1,\cdots,\Upomega_M\in\End_\Cbb((\Wbb_\blt^{\leq l})^*)\otimes_\Cbb\scr O(\mc B)$ independent on $\Mbb_1,\dots,\Mbb_M$ such that
\begin{align}
	q_k\partial_{q_k}(\mc S\uppsi^{\leq l})=\Upomega_k\cdot \mc S\uppsi^{\leq l}
\end{align}
for every $1\leq k\leq M$. This finishes the proof of that $\mc S\uppsi$ converges a.l.u.,  thanks to Thm. \ref{lb39}.
\end{proof}

In the above proof, we give the formulas \eqref{eq108} \eqref{eq110} for the projective term since we think it might be useful in the future.

\appendix

\section{Differential equations with simple poles and parameters}

Let $V$ be a connected  open subset of $\Cbb^m$ with coordinates $\tau_\blt=(\tau_1,\dots,\tau_m)$. Let $\mc D_{r_\blt}=\mc D_{r_1}\times\cdots\times\mc D_{r_M}$.  We let $q_j$ be the standard variable of $\mc D_{r_j}$.  For each $1\leq j\leq M$, choose
\begin{gather*}
A^j\in\End(\Cbb^N)\otimes\scr O(\mc D_{r_\blt}\times V),\\
\omega^j\in\Cbb^N\otimes\scr O(\mc D_{r_\blt}\times V).
\end{gather*}
Recall $\mc D_r$ is the open disc at $0$ with radius $r$. Consider the following  system of differential equations
\begin{align}
q_j\partial_{q_j}\psi=A^j\psi+\omega^j\label{eq71}
\end{align}
for all $j$, where
\begin{align}
\psi=&\sum_{l_1,\dots,l_M=0}^L\sum_{n_\blt\in \Nbb^M}\psi_{n,l}(\tau_\blt)q_\blt^{n_\blt}(\log q_\blt)^{l_\blt}\nonumber\\
=&\sum_{l_1,\dots,l_M=0}^L\sum_{n_\blt\in \Nbb^M}\psi_{n,l}(\tau_\blt)q_1^{n_1}\cdots q_M^{n_M}(\log q_1)^{l_1}\cdots (\log q_M)^{l_M}\qquad\label{eq77}
\end{align}
is in $\Cbb^N\otimes_\Cbb\scr O(V)[[q_\blt]][\log q_\blt]$ and is a formal solution of this system of equations. The following result is well-known although the reference is not easy to find. Thus, we provide a proof for completeness purpose. (See also \cite{McR21} Appendix A.)

\begin{thm}\label{lb39}
Suppose that the formal series $\psi$ satisfies \eqref{eq71} for all $j$. Then $\psi$ converges a.l.u., namely, for each $l_1,\dots,l_M$ and $|q_j|<r_j$, and each compact subset $K\subset V$,
\begin{align}
\sup_{\tau_\blt\in K}~\sum_{n_\blt\in\Nbb^M}\lVert \psi_{n_\blt,l_\blt}(\tau_\blt) \lVert \cdot |q_\blt^{n_\blt}|<+\infty	\label{eq103}.
\end{align}
\end{thm}

To prove this theorem, we first need:

\begin{lm}\label{lb59}
Suppose $M=1$ and we suppress the subscript $1$.  Suppose that the formal series $\psi$ does not contain $\log q$ and is a formal solution of \eqref{eq71}, namely,
\begin{align*}
q\partial_q \psi=A\psi+\omega.
\end{align*}
Then $\psi$ is an element of $\Cbb^N\otimes_\Cbb\scr O(\mc D_r\times V)$. Equivalently, $\psi$ converges a.l.u..
\end{lm}

\begin{proof}
It suffices to prove that $\psi$ is an $\Cbb^N$-valued holomorphic function on a neighborhood of $\{0\}\times V\subset\mc D_r\times V$. Then, by basic theory of linear differential equations, there is a (possibly multivalued) $\Cbb^N$-valued holomorphic function on $\mc D_r^\times\times V$ whose ``initial value" is given by $\psi$. (See \cite{Kna} the remark after Thm. B.1.) This function must be single-valued since $\psi$ is so. So $\psi$ is holomorphic on $\mc D_r\times V$.

Consider the series expansion of $A$ and $\omega$:
	\begin{align*}
	A(q,\tau_\blt)=\sum_{n\in\mbb N} A_n(\tau_\blt)q^n,\qquad \omega(q,\tau_\blt)=\sum_{n\in\Nbb}\omega_n(\tau_\blt)q^n
	\end{align*}
	where each $ A_n$ is  in $\End(\Cbb^N)\otimes_\Cbb\scr O(V)$, and each $\omega_n$ is in $\Cbb^N\otimes_\Cbb\scr O(V)$.  Then for each $n\in\mbb N$,
	\begin{align*}
	n\psi_n=\omega_n+\sum_{j=0}^n A_{n-j}\psi_j.
	\end{align*}
	Choose any open subset $U$ of $V$ with compact closure, and choose $B>0$ such that $\lVert A_0(\tau_\bullet)\lVert\leq B$ whenever $\tau_\blt\in U$. (Here $\lVert\cdot\lVert$ is the operator norm.) Then for any $n>B$, $n\id- A_0(\tau_\blt)$ is invertible (with inverse $n^{-1}\sum_{j=0}^\infty( A_0(\tau_\blt)/n)^j $). Thus, whenever $n>B$,
	\begin{align}
	\psi_n=(n- A_0)^{-1}\Big(\omega_n+\sum_{j=0}^{n-1} A_{n-j}\psi_j\Big).\label{eq73}
	\end{align}
	Choose any $r_1<r$ and set
	\begin{align}
	\alpha=\sup_{(q,\tau_\blt)\in {\ovl{\mc D_{r_1}}\times U}}\max\{\lVert A(q,\tau_\blt)\lVert,\lVert\omega(q,\tau_\blt)\lVert\}.\label{eq74}
	\end{align}
	Using $ A_n(\tau_\blt)=\oint_{\partial \mc D_{r_1}} A(q,\tau_\blt)q^{-n-1}\frac{dq}{2\im\pi}$ and a similar relation for $\omega_n$, we have
	\begin{align}
	\lVert A_n(\tau_\blt)\lVert,\lVert\omega_n(\tau_\blt)\lVert\leq \alpha r_1^{-n}\label{eq75}
	\end{align}
	for all $n$ and all  $\tau_\blt$ in $U$.

Choose $\beta>0$ such that  $\lVert (n- A_0(\tau_\blt))^{-1} \lVert\leq \beta n^{-1}$ for any $n> B$ and $\tau_\blt\in U$. (Such $\beta$ can be found using the explicit formula of inverse matrix given above.) Set $\gamma\geq\max\{1,\alpha\beta\}$. Then, from \eqref{eq73} and \eqref{eq75}, we see that for any $n>B$ and $\tau_\blt\in U$,
	\begin{align}
	r_1^n\lVert \psi_n(\tau_\blt) \lVert\leq \gamma n^{-1}\Big(1+\sum_{j=0}^{n-1}r_1^j\lVert \psi_j(\tau_\blt) \lVert\Big).\label{eq76}
	\end{align}
	By induction, one can show that there exists $c>0$ such that
	\begin{align*}
	r_1^n\lVert\psi_n(\tau_\blt)\lVert\leq c\gamma^n 
	\end{align*}
	for any $n\in\mbb N$ and $\tau_\blt\in U$. Indeed, if this is true for $0,1,2,\dots,n-1$ where $n>B$, and if we assume $c(\gamma^{n-1}-1)\geq 1$, then by \eqref{eq76},
	\begin{align*}
	&r_1^n\lVert \psi_n(\tau_\blt) \lVert\leq \gamma n^{-1}\Big(1+\sum_{j=0}^{n-1}c\gamma^j\Big)=\gamma n^{-1}\Big(1+c+\sum_{j=1}^{n-1}c\gamma^j\Big)\nonumber\\
	\leq &\gamma n^{-1}\Big(c\gamma^{n-1}+\sum_{j=1}^{n-1}c\gamma^{n-1}\Big)=c\gamma^n.
	\end{align*}
	Thus $\lVert\psi_n(\tau_\blt)\lVert\leq c\gamma^nr_1^{-n}$ for all $n$ and $\tau_\blt\in U$. Therefore, if we choose any $r_0\in(0,\gamma^{-1}r_1)$, then   the series $\sum_n\lVert \wht \psi_n(\tau_\blt) \lVert \cdot|q|^n$ is uniformly bounded by some positive number for all  $|q|\leq r_0$ and $\tau_\blt\in U$. Since each $\psi_n(\tau_\blt)$ is holomorphic over $\tau_\blt$, the series \eqref{eq77} must converge uniformly to a holomorphic function on $\mc D_{r_0}\times U$. This proves $\psi$ is holomorphic on a neighborhood of $\{0\}\times U$, and hence of $\{0\}\times V$ by choosing arbitrary $U$.
\end{proof}

\begin{lm}\label{lb60}
Thm. \ref{lb39} holds if $\psi$ has no log terms. 
\end{lm}
\begin{proof}
We prove this by induction on $M$. The case $M=1$ has been proved. Suppose the case is proved for $M-1$. In the case $M$, write $n_{\blt\setminus 1}=(n_2,\dots,n_M)$. Set
\begin{gather}
\psi(q_1,\dots,q_M)=\sum_{n_{\blt\setminus 1}\in\Nbb^{M-1}}\psi_{n_{\blt\setminus 1}}\psi(q_1)q_2^{n_2}\cdots q_M^{n_M}\label{eq107}
\end{gather}	
and define $A^1_{n_{\blt\setminus 1}},\omega^1_{n_{\blt\setminus 1}}$ in a similar way. Then
\begin{align*}
q_1\partial_{q_1}\psi_{n_{\blt\setminus 1}}(q_1)=\sum_{k_{\blt\setminus 1}\leq n_{\blt\setminus 1}}A_{n_2-k_2,\dots,n_M-k_M}\psi_{k_2,\dots,k_M}(q_1)+\omega^1_{n_{\blt\setminus 1}}(q_1).	
\end{align*}
Therefore, for each $n_{\blt\setminus 1}$, $\bigoplus_{k_{\blt\setminus 1}\leq n_{\blt\setminus 1}}\psi_{k_{\blt\setminus 1}}\in\Cbb^{N(n_2+1)\cdots (n_M+1)}$ satisfies a differential equation of $\partial_{q_1}$ similar to that in Lemma \ref{lb59}. So $\psi_{n_{\blt\setminus 1}}\in\Cbb^N\otimes\scr O(\mc D_{r_1}\times V)$.

Now consider \eqref{eq107} as a formal series of $q_2,\dots,q_M$ whose coefficients are holomorphic over $\mc D_{r_1}\times V$ and which satisfies a system of $M-1$ differential equations as in \eqref{eq71}. Thus, by the case $M-1$, \eqref{eq107} is the series expansion of an element of $\Cbb^N\otimes\scr O(\mc D_{r_\blt}\times V)$. This finishes the proof.
\end{proof}

\begin{proof}[Proof of Thm. \ref{lb39}]
Write
\begin{align*}
\psi=\sum_{n_\blt,l_\blt}\psi_{n_\blt,l_\blt}q_\blt^{n_\blt}(\log q_\blt)^{l_\blt}=\sum_{l_\blt}\mu_{l_\blt}(\log q_\blt)^{l_\blt}	
\end{align*}
where $\mu_{l_\blt}=\sum_{n_\blt}\psi_{n_\blt,l_\blt}q_\blt^{n_\blt}$. Recall that we assume \eqref{eq71} holds. By looking at the  $(\log q_\blt)^{l_\blt}$ part of this equation, we see
\begin{align*}
	q_j\partial_{q_j}\mu_{l_\blt}+(l_j+1)\mu_{l_1,\dots,l_j+1,\dots,l_M}=A\mu_{l_\blt}	
\end{align*}
when one of $l_1,\dots,l_M$ is non-zero; if $l_1=\cdots=l_M=0$, the right hand side is $A\mu_{l_\blt}+\omega^j$. By induction on $l_1+\cdots+l_M$ (from large to small), and by Lemma \ref{lb60}, we see that each $\mu_{l_\blt}$ belongs to  $\Cbb^N\otimes\scr O(\mc D_{r_\blt}\times V)$. So $\psi$ satisfies \eqref{eq103}.
\end{proof}

\section{Existence of projective structures}\label{lb27}

Let $\fk X=(\pi:\mc C\rightarrow\mc B)$ be a family of compact Riemann surfaces. In \cite[Lemma 5]{Hub81} Hubbard showed that $\fk X$ has a projective structure when $\mc B$ is Stein and the connected components of the fibers have genera $>1$. In this section, we cover the low genus cases.

Let $\fk U=\{U_\alpha\}_{\alpha\in\fk A}$ be a Stein cover \footnote{An open cover $\fk U$ of a complex manifold $X$ is called Stein if each open set $U\in\fk U$ is a Stein manifold. Then any finite intersection of open sets of $\fk U$ is also Stein (see \cite[Sec. 1.4.4]{GR84}).} of $\mc C$ such that each member $U_\alpha$ admits $\eta_\alpha\in\scr O(U_\alpha)$ univalent on each fiber of $U_\alpha$. Define a \v Cech $1$-cochain $\sigma=(\sigma_{\alpha,\beta})_{\alpha,\beta\in\fk A}\in C^1(\fk U,\omega_{\mc C/\mc B}^{\otimes 2})$ such that 
\begin{align*}
\sigma_{\alpha,\beta}=\Sbf_{\eta_\alpha}\eta_\beta\cdot d\eta_\alpha^2\qquad \in\omega_{\mc C/\mc B}^{\otimes 2}(U_\alpha\cap U_\beta).
\end{align*}
By \eqref{eq49}, $\sigma$ is a cocycle and hence can be viewed as an element of $H^1(\mc C,\omega_{\mc C/\mc B}^{\otimes 2})$. The proof of the following Lemma is in \cite[Lemma 5]{Hub81}. We recall the proof for the reader's convenience.

\begin{lm}\label{lb26}
	$\fk X$ admits a projective structure if and only if $\sigma$ is the zero element of $H^1(\mc C,\omega_{\mc C/\mc B}^{\otimes 2})$.
\end{lm}

Note that $H^1(\mc C,\omega_{\mc C/\mc B}^{\otimes 2})$ equals $H^1(\fk U,\omega_{\mc C/\mc B}^{\otimes 2})$ by Leray's theorem.

\begin{proof}
	``If": We have a $0$-cochain $s=(s_\alpha)_{\alpha\in\fk A}\in C^0(\fk U,\omega_{\mc C/\mc B}^{\otimes 2})$ such that $\delta s=\sigma$. By Proposition \ref{lb24}-(2) and by passing to a finer Stein cover (still denoted by $\fk U$ for simplicity), we can find $f_\alpha\in\scr O(U_\alpha)$ univalent on each fiber, such that
	\begin{align}
	\Sbf_{\eta_\alpha}f_\alpha\cdot d\eta_\alpha^2=s_\alpha.\label{eq50}
	\end{align}
	Thus, on $U_\alpha\cap U_\beta$ we have
	\begin{align*}
	\Sbf_{\eta_\alpha}\eta_\beta\cdot d\eta_\beta^2=\sigma_{\alpha,\beta}=s_\alpha-s_\beta=\Sbf_{\eta_\alpha}f_\alpha\cdot d\eta_\alpha^2-\Sbf_{\eta_\beta}f_\beta\cdot d\eta_\beta^2.
	\end{align*}
	By \eqref{eq49}, we have
	\begin{align*}
	\Sbf_{\eta_\alpha}\eta_\beta\cdot d\eta_\beta^2=\Sbf_{\eta_\alpha}f_\beta\cdot d\eta_\alpha^2-\Sbf_{\eta_\beta}f_\beta\cdot d\eta_\beta^2.
	\end{align*}
	These two imply $\Sbf_{\eta_\alpha}f_\alpha=\Sbf_{\eta_\alpha}f_\beta$. Thus, by Proposition \ref{lb24}-(3), $\Sbf_{f_\alpha}f_\beta=0$. So $(U_\alpha,f_\alpha)_{\alpha\in\fk A}$ is a projective chart.
	
	``Only if": Assume $(U_\alpha,f_\alpha)_{\alpha\in\fk A}$ is a projective chart. Define $s=(s_\alpha)_{\alpha\in\fk A}$ using \eqref{eq50}. Notice $\Sbf_{\eta_\alpha}f_\alpha=\Sbf_{\eta_\alpha}f_\beta$. One can reverse the argument in the first paragraph to show $\delta s=\sigma$.
\end{proof}

\begin{thm}\label{lb40}
	Assume $\mc B$ is a Stein manifold. Then $\fk X$ admits a projective structure.
\end{thm}

\begin{proof}
Assume without loss of generality that the fibers are connected and have genus $g$. (Recall they are diffeomorphic by Ehresmann's theorem.) 
	
Assume $g=1$. Each fiber $\mc C_b$ is a complex torus and hence obviously has a projective structure (e.g. the one from the standard local coordinates of $\mbb C$). Thus, by Lemma \ref{lb26}, for each $b\in\mc B$ the restriction $\sigma|\mc C_b$ is the zero element of $H^1(\mc C_b,\omega_{\mc C_b}^{\otimes 2})$. Since $b\mapsto \dim H^1(\mc C_b,\omega_{\mc C_b}^{\otimes 2})$ is constantly $1$ (because $\omega_{\mc C_b}\simeq \scr O_{\mc C_b}$), by Grauert's Theorem \ref{lb11}, $R^1\pi_*\omega_{\mc C/\mc B}^{\otimes 2}$ is locally free and each fiber is equivalent to $H^1(\mc C_b,\omega_{\mc C_b}^{\otimes 2})$. So $\sigma$ is the zero section of $R^1\pi_*\omega_{\mc C/\mc B}^{\otimes 2}$ over $\mc B$. Thus, we may find a Stein cover $\fk V=(V^i)_{i\in\mc I}$ of $\mc B$ such that  the restriction of $\sigma$ to $\mc C_{V^i}=\pi^{-1}(V^i)$ is the zero element of $H^1(\mc C_{V^i},\omega_{\mc C/\mc B}^{\otimes 2})$.
	
Let $W_\alpha^i=U_\alpha\cap \mc C_{V^i}$ which is Stein (cf. \cite[Sec. 1.4.4]{GR84}). Then for each $i\in\mc I$, $\fk W^i=(W_\alpha^i)_{\alpha\in\fk A}$ is a Stein cover of $\mc C_{V^i}$.  So $\sigma|\mc C_{V^i}$ is the zero element of  $H^1(\fk W^i,\omega_{\mc C/\mc B}^{\otimes 2})$. Choose $t^i=(t^i_\alpha)_{\alpha\in\fk A}$ (where each $t^i_\alpha\in\omega_{\mc C/\mc B}^{\otimes 2}(W_\alpha^i)$) such that $\delta t^i=\sigma|\mc C_{V^i}$. On $U_\alpha\cap U_\beta\cap \mc C_{V^i}\cap \mc C_{V^j}$ we have $t^i_\alpha-t^i_\beta=\sigma_{\alpha,\beta}=t^j_\alpha-t^j_\beta$ and hence $t^i_\alpha-t^j_\alpha=t^i_\beta-t^j_\beta$. Therefore, we have a well-defined element $u=(u^{i,j})_{i,j\in\mc I}$ of $H^1(\fk V,\pi_*\omega_{\mc C/\mc B}^{\otimes 2})$ such that $u^{i,j}$, which is an element of $(\pi_*\omega_{\mc C/\mc B}^{\otimes 2})(V^i\cap V^j)=\omega_{\mc C/\mc B}^{\otimes 2}(\mc C_{V^i}\cap \mc C_{V^j})$, equals $t^i_\alpha-t^j_\alpha$ when restricted to $\mc C_{V^i}\cap \mc C_{V^j}\cap U_\alpha$ for  any $\alpha$. Since $\mc B$ is Stein, by Cartan's Theorem B,  $H^1(\fk V,\pi_*\omega_{\mc C/\mc B}^{\otimes 2})$ is trivial. So there exists $v=(v^i)_{i\in\mc I}$ (where each $v^i\in \omega_{\mc C/\mc B}^{\otimes 2}(\mc C_{V^i})$) such that $v^i-v^j=u^{i,j}$ on $\mc C_{V^i}\cap \mc C_{V^j}$. So $v^i-v^j=t^i_\alpha-t^j_\alpha$ on $\mc C_{V^i}\cap \mc C_{V^j}\cap U_\alpha$. So there is a well-defined $s_\alpha\in\omega_{\mc C/\mc B}^{\otimes 2}(U_\alpha)$ which equals $t^i_\alpha-v^i$ on  $U_\alpha\cap \mc C_{V^i}$ for each $i$. Let $s=(s_\alpha)_{\alpha\in\fk A}$. One checks easily $\delta s=\sigma$. Thus, by Lemma \ref{lb26}, $\fk X$ has a projective structure.
	
Assume $g>1$. Then $H^1(\mc C_b,\omega_{\mc C_b}^{\otimes 2})$ is trivial (since $\deg\omega_{\mc C_b}=2g-2>0$). A similar argument  shows that $\sigma$ is a coboundary. Assume $g=0$. Then $b\mapsto \dim H^1(\mc C_b,\omega_{\mc C_b}^{\otimes 2})$ is still constant since any fiber $\mc C_b$ is biholomorphic to $\Pbb^1$. Moreover, $\mc C_b$ clearly has a projective structure.  The above argument for $g=1$ applies to this case. 
\end{proof}

	\printindex

\noindent {\small \sc Yau Mathematical Sciences Center, Tsinghua University, Beijing, China.}

\noindent {\textit{E-mail}}: binguimath@gmail.com\qquad bingui@tsinghua.edu.cn

\newpage
	
	\begin{thebibliography}{999999}
		\footnotesize	
		
		
		
		
		
		\bibitem[ACG11]{ACG11}
		Arbarello, E., Cornalba, M. and Griffiths, P., 2011. Geometry of algebraic curves: volume II with a contribution by Joseph Daniel Harris (Vol. 268). Springer Science \& Business Media.
		
\bibitem[AU07]{AU07}
Andersen, J.E. and Ueno, K., 2007. Abelian conformal field theory and determinant bundles. Internat. J.  Math., 18(08), pp.919-993.



\bibitem[Ahl]{Ahl}
Ahlfors, L.V., 2006. Lectures on quasiconformal mappings (Vol. 38). 2ed. American Mathematical Soc..


\bibitem[BDH17]{BDH17}
Bartels, A., Douglas, C.L. and Henriques, A., 2017. Conformal nets II: Conformal blocks. Communications in Mathematical Physics, 354(1), pp.393-458.

\bibitem[BFM91]{BFM91}
Beilinson, A., Feigin, B. and Mazur, B., 1991. Introduction to algebraic field theory on curves. preprint, 395, pp.462-463.

\bibitem[BK01]{BK01}
Bakalov, B. and Kirillov, A.A., 2001. Lectures on tensor categories and modular functors (Vol. 21). American Mathematical Soc..


\bibitem[BPZ84]{BPZ84}
Belavin, A.A., Polyakov, A.M. and Zamolodchikov, A.B., 1984. Infinite conformal symmetry in two-dimensional quantum field theory. Nuclear Physics B, 241(2), pp.333-380.
		
		\bibitem[BS76]{BS76}
		Bănică, C. and Stănăşilă, O., 1976. Algebraic methods in the global theory of complex spaces.
		
		\bibitem[Buhl02]{Buhl02}
		Buhl, G., 2002. A spanning set for VOA modules. Journal of Algebra, 254(1), pp.125-151.
		
\bibitem[Cod19]{Cod19}
Codogni, G., 2019. Vertex algebras and Teichm\"uller modular forms. arXiv preprint arXiv:1901.03079.		
		
		\bibitem[DGT19a]{DGT19a}
	Damiolini, C., Gibney, A., and Tarasca, N., 2019. Conformal blocks from vertex algebras and their connections on $\ovl{\mc M}_{g,n}$.  To appear in Geometry \& Topology, arXiv:1901.06981.
	
\bibitem[DGT19b]{DGT19b}
Damiolini, C., Gibney, A., and Tarasca, N.. On Factorization and vector bundles of conformal blocks from vertex algebras. To appear in Annales scientifiques de l’École normale supérieure.

\bibitem[DGT19c]{DGT19c}
Damiolini, C., Gibney, A., \& Tarasca, N. (2022). Vertex algebras of CohFT-type. Facets of algebraic geometry, 1, 164-189.
		
\bibitem[DLM97]{DLM97}
Dong, C., Li, H. and Mason, G., 1997. Regularity of Rational Vertex Operator Algebras. Advances in Mathematics, 1(132), pp.148-166.

\bibitem[DLM00]{DLM00}
Dong, C., Li, H. and Mason, G., 2000. Modular-Invariance of Trace Functions¶ in Orbifold Theory and Generalized Moonshine. Communications in Mathematical Physics, 214(1), pp.1-56.
		
		\bibitem[EP96]{EP96}
		Eschmeier, J. and Putinar, M., 1996. Spectral decompositions and analytic sheaves (No. 10). Oxford University Press.
		
		\bibitem[FB04]{FB04}
		E. Frenkel and D. Ben-Zvi, Vertex Algebras And Algebraic Curves, Mathematical Surveys And Monographs 88, second edition (American Mathematical Society, 2004).
		
		\bibitem[FHL93]{FHL93}
		Frenkel, I., Huang, Y.Z. and Lepowsky, J., 1993. On axiomatic approaches to vertex operator algebras and modules (Vol. 494). American Mathematical Soc..
		
\bibitem[FS87]{FS87}
Friedan, D. and Shenker, S., 1987. The analytic geometry of two-dimensional conformal field theory. Nuclear Physics B, 281(3-4), pp.509-545.		
		
\bibitem[Fio16]{Fio16}
Fiordalisi, F., 2016. Logarithmic intertwining operators and genus-one correlation functions. Communications in Contemporary Mathematics, 18(06), p.1650026.		
		
		\bibitem[Fis76]{Fis76}
		Fischer, G., Complex analytic geometry. Lecture Notes in Mathematics, Vol. 538. Springer-Verlag, Berlin-New York, 1976. vii+201 pp.
		
		

\bibitem[GN03]{GN03}
Gaberdiel, M.R. and Neitzke, A., 2003. Rationality, quasirationality and finite W-algebras. Communications in mathematical physics, 238(1), pp.305-331.
		
		\bibitem[GPR94]{GPR94}
		Grauert, H., Peternell, T., and Remmert, R. eds., 1994. Several complex variables VII: sheaf-theoretical methods in complex analysis. Springer-Verlag.
		
		\bibitem[GR84]{GR84}
		Grauert, H. and Remmert, R., 1984. Coherent analytic sheaves (Vol. 265). Springer Science \& Business Media.
		
		\bibitem[Gra60]{Gra60}
		Grauert, H., 1960. Ein Theorem der analytischen Garbentheorie und die Modulr\"aume komplexer Strukturen. Publications Math\'ematiques de l'IH\'ES, 5, pp.5-64.
		
\bibitem[Gun]{Gun}
Gunning, R.C., 1966. Lectures on Riemann surfaces (Vol. 2). Princeton University Press.

\bibitem[HLZ11]{HLZ11}
Huang, Y.Z., Lepowsky, J. and Zhang, L., 2011. Logarithmic tensor category theory, VII: Convergence and extension properties and applications to expansion for intertwining maps. arXiv preprint arXiv:1110.1929.

\bibitem[HLZ14]{HLZ14}
Huang, Y.Z., Lepowsky, J. and Zhang, L., 2014. Logarithmic tensor category theory for generalized modules for a conformal vertex algebra, I: Introduction and strongly graded algebras and their generalized modules. In Conformal field theories and tensor categories (pp. 169-248). Springer, Berlin, Heidelberg.
	
\bibitem[Hua95]{Hua95}
Huang, Y.Z., 1995. A theory of tensor products for module categories for a vertex operator algebra, IV. Journal of Pure and Applied Algebra, 100(1-3), pp.173-216.	

		
		\bibitem[Hua97]{Hua97}
		Huang, Y.Z., 1997. Two-dimensional conformal geometry and vertex operator algebras (Vol. 148). Springer Science \& Business Media.
		


\bibitem[Hua98]{Hua98}
Huang, Y.Z., 1998. Genus-zero modular functors and intertwining operator algebras. International Journal of Mathematics, 9(07), pp.845-863.		
		
\bibitem[Hua05a]{Hua05a}
Huang, Y.Z., 2005. Differential equations and intertwining operators. Communications in Contemporary Mathematics, 7(03), pp.375-400.

\bibitem[Hua05b]{Hua05b}
Huang, Y.Z., 2005. Differential equations, duality and modular invariance. Communications in Contemporary Mathematics, 7(05), pp.649-706.

\bibitem[Hua08a]{Hua08a}
Huang, Y.Z., 2008. Vertex operator algebras and the Verlinde conjecture. Communications in Contemporary Mathematics, 10(01), pp.103-154.

\bibitem[Hua08b]{Hua08b}
Huang, Y.Z., 2008. Rigidity and modularity of vertex tensor categories. Communications in contemporary mathematics, 10(supp01), pp.871-911.

\bibitem[Hua09]{Hua09}
Huang, Y. Z. (2009). Cofiniteness conditions, projective covers and the logarithmic tensor product theory. Journal of Pure and Applied Algebra, 213(4), 458-475.

\bibitem[Hua16]{Hua16}
Huang, Y.Z., 2016. Some open problems in mathematical two-dimensional conformal field theory. arXiv preprint arXiv:1606.04493.

\bibitem[Hua17]{Hua17}
Huang, Y.Z., 2017. On the applicability of logarithmic tensor category theory. arXiv preprint arXiv:1702.00133.

\bibitem[Hub81]{Hub81}
Hubbard, J.H., 1981, May. The monodromy of projective structures. In Riemann surfaces and related topics: Proceedings of the 1978 Stony Brook Conference (State Univ. New York, Stony Brook, NY, 1978) (Vol. 97, pp. 257-275).
		



\bibitem[Kna]{Kna}
Knapp, A.W., 1986. Representation theory of semisimple groups: an overview based on examples (Vol. 36). Princeton university press.




\bibitem[Loo10]{Loo10}
Looijenga, E., 2010. From WZW models to modular functors. arXiv preprint arXiv:1009.2245.

\bibitem[MS89]{MS89}
Moore, G. and Seiberg, N., 1989. Classical and quantum conformal field theory. Communications in Mathematical Physics, 123(2), pp.177-254.


\bibitem[McR21]{McR21}
McRae, R., 2021. On rationality for $C_2$-cofinite vertex operator algebras. arXiv preprint arXiv:2108.01898.

\bibitem[Miy04]{Miy04}
Miyamoto, M., 2004. Modular invariance of vertex operator algebras satisfying C2-cofiniteness. Duke Mathematical Journal, 122(1), pp.51-91.


\bibitem[NT05]{NT05}
Nagatomo, K. and Tsuchiya, A., 2005. Conformal field theories associated to regular chiral vertex operator algebras, I: Theories over the projective line. Duke Mathematical Journal, 128(3), pp.393-471.

\bibitem[Seg88]{Seg88}
Segal, G.B., 1988. The definition of conformal field theory. In Differential geometrical methods in theoretical physics (pp. 165-171). Springer, Dordrecht.

\bibitem[TK88]{TK88}
Tsuchiya, A. and Kanie, Y., 1988. Vertex operators in conformal field theory on $\Pbb^1$ and monodromy representations of braid group. Conformal Field Theory and Solvable Lattice Models, Advanced Studies in Pure Math, 16, pp.297-372.

\bibitem[Ten17]{Ten17}
Tener, J.E., 2017. Construction of the unitary free fermion Segal CFT. Communications in Mathematical Physics, 355(2), pp.463-518.



\bibitem[Ten19a]{Ten19a}
Tener, J.E., 2019. Geometric realization of algebraic conformal field theories. Advances in Mathematics, 349, pp.488-563.


\bibitem[Ten19b]{Ten19b}
Tener, J.E., 2019. Representation theory in chiral conformal field theory: from fields to
observables. Selecta Math. (N.S.) (2019), 25:76

\bibitem[Ten19c]{Ten19c}
Tener, J.E., 2019. Fusion and positivity in chiral conformal field theory. arXiv preprint arXiv:1910.08257.
		

\bibitem[TUY89]{TUY89}
Tsuchiya, A., Ueno, K. and Yamada, Y., 1989. Conformal field theory on universal family of stable curves with gauge symmetries. In Integrable Sys Quantum Field Theory (pp. 459-566). Academic Press.
	
\bibitem[Ueno97]{Ueno97}
Ueno, K., 1997. Introduction to conformal field theory with gauge symmetries. Geometry and Physics, Lecture Notes in Pure and Applied Mathematics, 184, pp.603-745.

		

\bibitem[Vafa87]{Vafa87}
Vafa, C., 1987. Conformal theories and punctured surfaces. Physics Letters B, 199(2), pp.195-202.	
		
\bibitem[Zhu94]{Zhu94}
Zhu, Y., 1994. Global vertex operators on Riemann surfaces. Communications in Mathematical Physics, 165(3), pp.485-531.
		
\bibitem[Zhu96]{Zhu96}
Zhu, Y., 1996. Modular invariance of characters of vertex operator algebras. Journal of the American Mathematical Society, 9(1), pp.237-302.
		
		
		
		
		
	\end{thebibliography}
\end{document}